\documentclass[a4paper,11pt]{article}
\pdfoutput=1 

\usepackage{jheppub} 
\usepackage{amsmath,amssymb,amsthm}
\usepackage{graphicx}
\usepackage{subcaption}
\usepackage[percent]{overpic}
\usepackage[svgnames]{xcolor}
\usepackage{multirow}
\usepackage{slashed}
\usepackage{mathrsfs}
\usepackage{tikz}
\usepackage[boxsize=5pt]{ytableau}
\usepackage{bbold}
\usepackage{tensor}
\usepackage{wasysym}
\usepackage{bbm}
\usepackage{pinlabel}
\usepackage{mathtools}
\mathtoolsset{showonlyrefs}
\usepackage{tikz}
\usetikzlibrary{knots}
\tikzset{knot/.style={double=#1,double distance=1pt,line width=2pt,white}}

\newcommand{\R}{\mathbb{R}}
\newcommand{\C}{\mathbb{C}}
\newcommand{\xUDarrow}[1]{%
 {\left\updownarrow\vbox to #1{}\right.\kern-\nulldelimiterspace}
}

\newcommand{\nolabel}[1]{%
 {}
}

\makeatletter
\newcommand\xLRarrow[2][]{%
  \ext@arrow 9999{\longleftrightarrowfill@}{#1}{#2}}
\newcommand\longleftrightarrowfill@{%
  \arrowfill@\leftarrow\relbar\rightarrow}
 \newcommand{\xdashleftrightarrow}[2][]{\ext@arrow 3359\leftrightarrowfill@@{#1}{#2}}
 \def\leftrightarrowfill@@{\arrowfill@@\leftarrow\relbar\rightarrow}
\def\arrowfill@@#1#2#3#4{%
  $\m@th\thickmuskip0mu\medmuskip\thickmuskip\thinmuskip\thickmuskip
   \relax#4#1
   \xleaders\hbox{$#4#2$}\hfill
   #3$%
}
\makeatother

\newtheorem{thm}{Theorem}[section]
\newtheorem{prp}[thm]{Proposition}

\newtheorem{cnj}[thm]{Conjecture}

\newtheorem{rmk}[thm]{Remark}
\newtheorem{dfn}[thm]{Definition}

\def\be{\begin{equation}}
\def\ee{\end{equation}}

\def\sfA{{\mathsf{A}}}
\def\sfB{{\mathsf{B}}}
\def\sfC{{\mathsf{C}}}
\def\sfD{{\mathsf{D}}}
\def\sfH{{\mathsf{H}}}
\def\sfI{{\mathsf{I}}}

\def\sfP{{\mathsf{P}}}
\def\sfW{{\mathsf{W}}}
\def\sfZ{{\mathsf{Z}}}
\def\sfPsi{{\mathsf{\Psi}}}
\def\IR{{\mathbb{R}}}
\def\IQ{{\mathbb{Q}}}
\def\IZ{{\mathbb{Z}}}
\def\IP{{\mathbb{P}}}
\def\IC{{\mathbb{C}}}
\def\IN{{\mathbb{N}}}

\def\CH{{\mathcal{H}}}
\def\CM{{\mathcal{M}}}
\def\CN{{\mathcal{N}}}

\def\ws{\mathrm{ws}}

\def\CO{\mathcal{O}}

\def\bx{\boldsymbol{x}}

\def\hgt{\mathrm{ht}}

\def\gcd{{\rm{gcd}}}

\def\Sk{{\rm{Sk}}}

\def\bc{{\boldsymbol{c}}}

\newcommand{\pic}[2]{\raisebox{-.5\height}{\includegraphics[scale=#2]{#1}}}
\newcommand{\picLabel}[3]{
	\raisebox{-.5\height}{
		\labellist
		\Label #1
		\endlabels
		\endlabellist
		\includegraphics[scale=#3]{#2}}
	}
	\makeatletter % we need to use kernel commands´
	\newcommand{\Label}{
		\@Labeli
	}
	\newcommand\@Labeli{\@ifnextchar\endlabels{\@Labelend}{\@Labelii}}
	\newcommand\@Labelii[3]{
		\pinlabel {#1} [tr] at #2 #3
		\@Labeli % restart the recursion
	}
	\newcommand\@Labelend[1]{% The argument is \endlabels	
	}
	\makeatother
\newcommand{\newpic}[2]{\raisebox{-.4\height}{\includegraphics[scale=#2]{#1}}}
\newcommand\OC{\pic{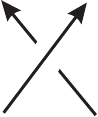}{.50}}
\newcommand\UC{\pic{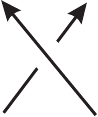} {.50}}
\newcommand\SP{\pic{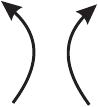} {.50}}
\newcommand\UK{\pic{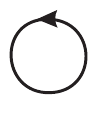} {.50}}
\newcommand\PT{\pic{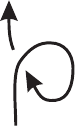} {.50}}
\newcommand\NT{\pic{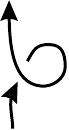} {.50}}
\newcommand\ST{\pic{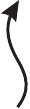} {.50}}

\newcommand\CIo{\newpic{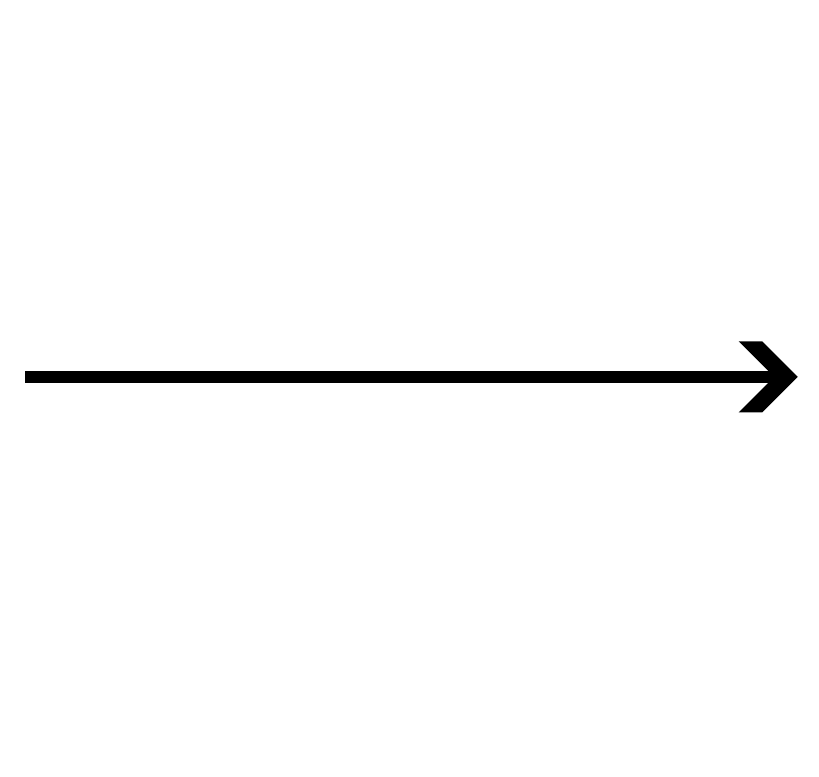} {.075}}
\newcommand\CIIo{\newpic{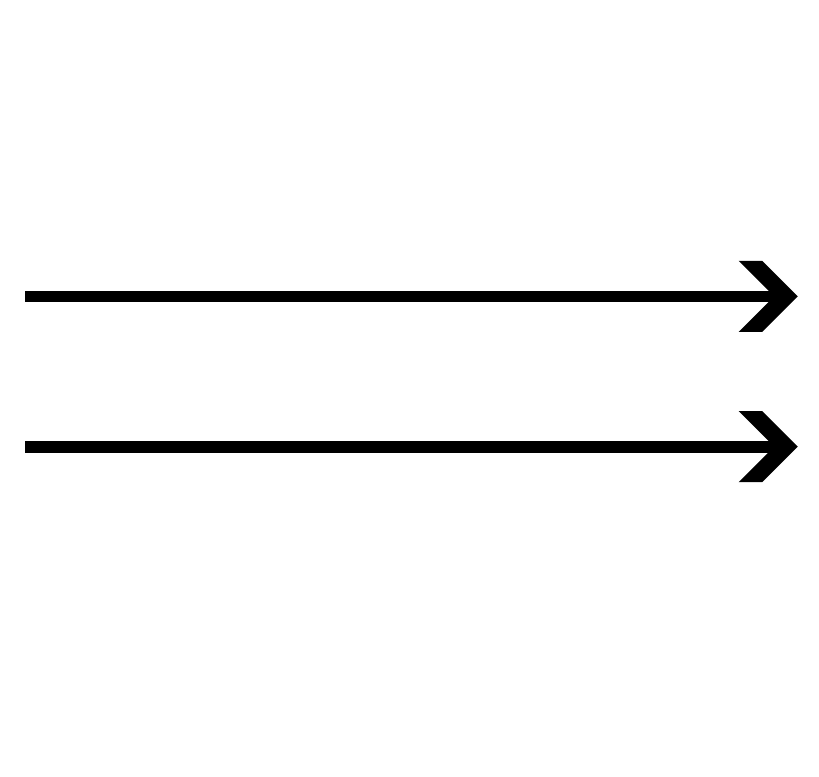} {.075}}
\newcommand\CIIi{\newpic{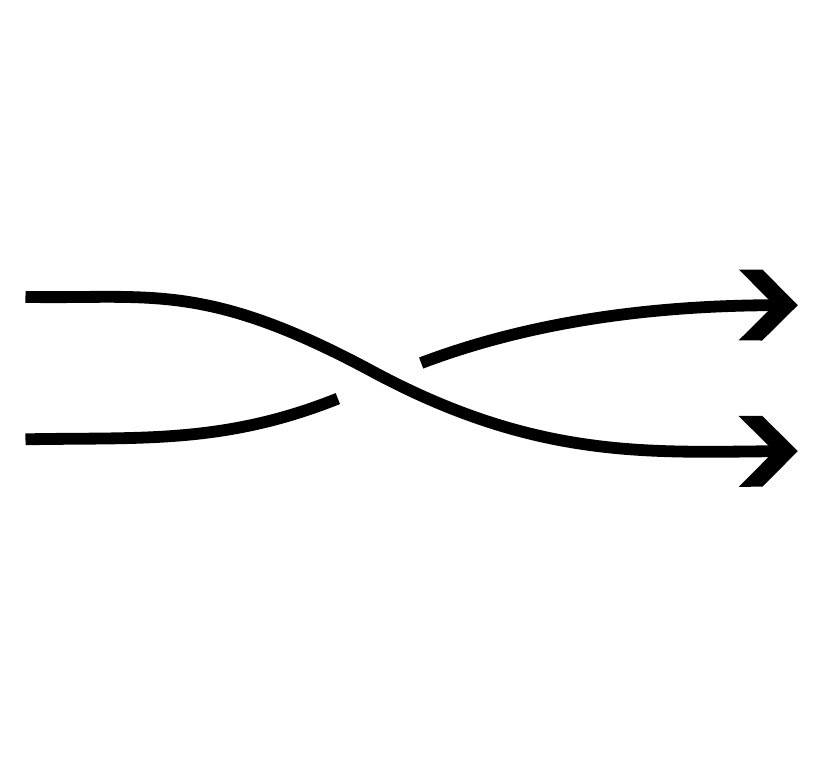} {.075}}

\newcommand\CIIIo{\newpic{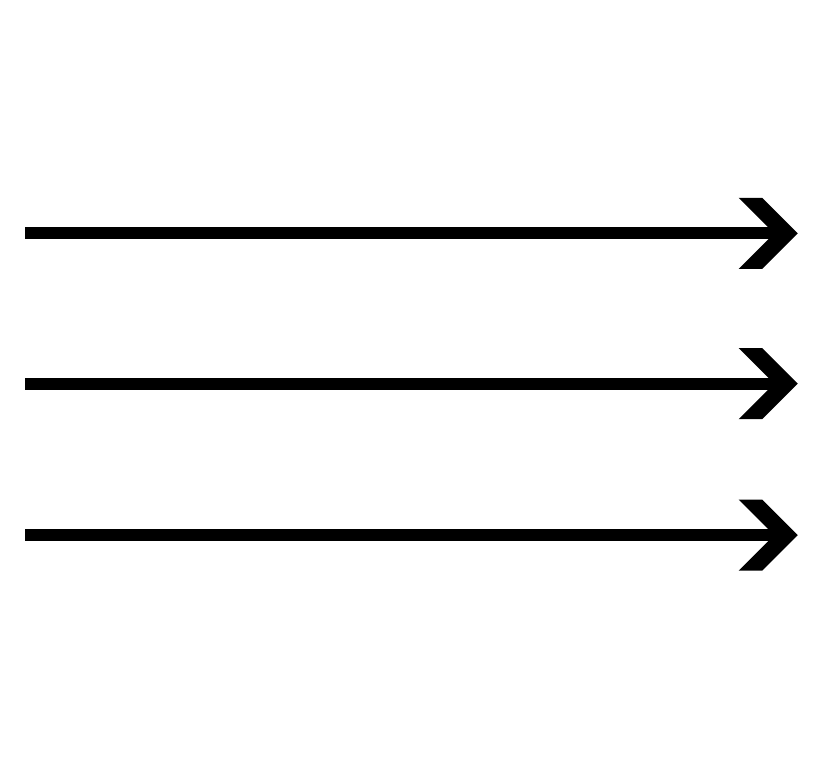} {.075}}
\newcommand\CIIIi{\newpic{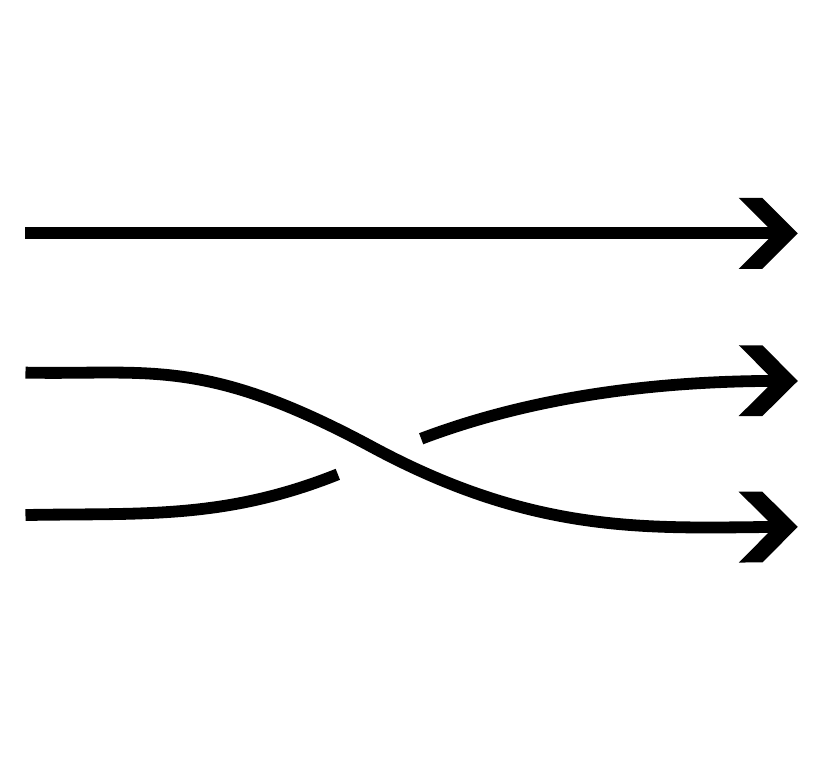} {.075}}
\newcommand\CIIIii{\newpic{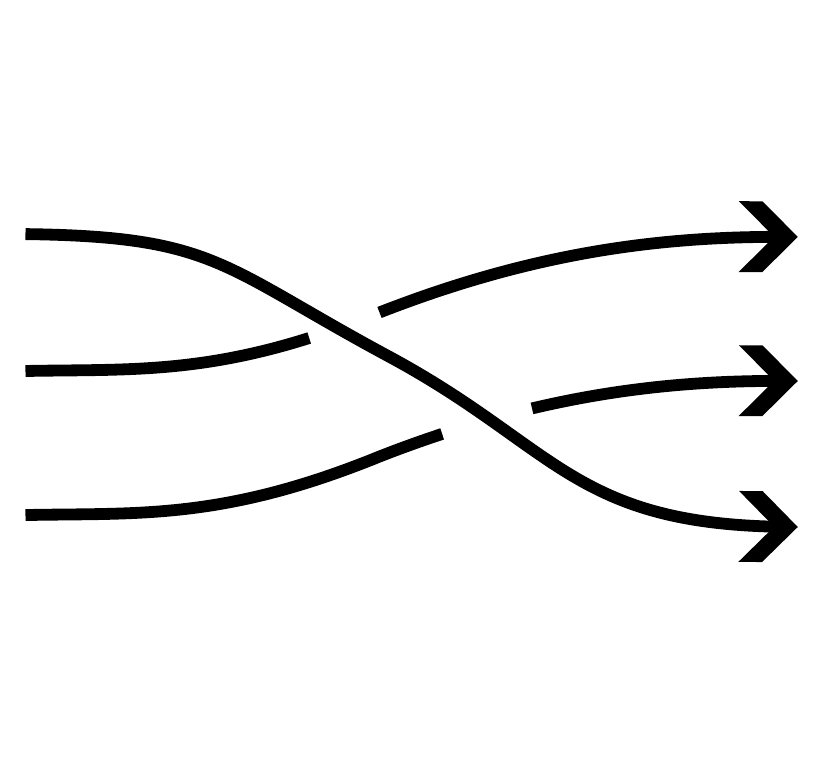} {.075}}

\renewcommand{\(}{\left(}
\renewcommand{\)}{\right)}

\title{The worldsheet skein D-module and basic curves on Lagrangian fillings of the Hopf link conormal}

\author[a,b,c]{Tobias Ekholm}
\author[b,c,d]{Pietro Longhi}
\author[b,c]{Lukas Nakamura}

\affiliation[a]{Institut Mittag-Leffler, Aurav 17, 182 60 Djursholm, Sweden}
\affiliation[b]{Department of Mathematics, Uppsala University, Box 480, 751 06 Uppsala, Sweden}
\affiliation[c]{Centre for Geometry and Physics, Uppsala University, Box 480, 751 20 Uppsala, Sweden}
\affiliation[d]{Department of Physics and Astronomy, Uppsala University, Box 516, 751 20 Uppsala, Sweden}

\emailAdd{tobias.ekholm@math.uu.se, pietro.longhi@physics.uu.se, lukas.nakamura@math.uu.se}

\abstract{
	HOMFLYPT polynomials of knots in the 3-sphere in symmetric representations satisfy recursion relations.
	Their geometric origin is holomorphic curves at infinity on knot conormals that determine a $D$-module
	with characteristic variety the Legendrian knot conormal augmention variety and with the recursion relations as operator polynomial generators \cite{Aganagic:2013jpa,Ekholm:2018iso}. 
	We consider skein lifts of recursions and $D$-modules corresponding to skein valued open curve counts \cite{Ekholm:2019yqp} that encode HOMFLYPT polynomials colored by arbitrary partitions. We define a worldsheet skein module which is the universal target for skein curve counts and a corresponding $D$-module. 
	
	We then consider the concrete example of the Legendrian conormal of the Hopf link. We show that the worldsheet skein $D$-module for the Hopf link conormal is generated by three operator polynomials that  
	annihilate the skein valued partition function for any choice of Lagrangian filling and recursively determine it uniquely. We find Lagrangian fillings for any point in the augmentation variety and show that their skein valued partition functions admit quiver-like expansions where all holomorphic curves are generated by a small number of basic holomorphic disks and annuli and their multiple covers. 
	}  
\begin{document} 

\maketitle

\section{Introduction}
Quantum knot invariants satisfy recursion relations. For example, generating functions of symmetrically colored HOMFLYPT polynomials of knots are known to be $q$-holonomic~\cite{Garoufalidis}. 

HOMFLYPT-polynomials of a knot in~$S^{3}$ are connected to topological string theory of its Lagrangian conormal, first in $T^{\ast} S^{3}$ and then, after conifold transition, in the resolved conifold~\cite{Ooguri:1999bv}. From this perspective, the recursion relation is an operator equation (sometimes called a `quantum curve') that quantizes the moduli space of classical vacua which is given by the corresponding equation in commutative variables~\cite{Aganagic:2003qj}. In the A-model topological string, the classical vacua correspond to (formal) Lagrangian fillings of the Legendrian unit knot conormal, together with their disk potentials (the semi-classical limit of counts of holomorphic curves ending on them) that constitute the augmentation variety of knot contact homology, see~\cite{Aganagic:2013jpa}. The quantum curve lift of the augmentation variety is a $D$-module generated by an  operator polynomial that annihilates the partition function generated by all holomorphic curves on the Lagrangian conormal, and the augmentation variety is the characteristic variety of this $D$-module, see~\cite{Aganagic:2013jpa,Ekholm:2018iso}. 

In~\cite{Ekholm:2019yqp}, a new approach to open curve counting was developed. Deformation invariant counts of holomorphic curves in a Calabi-Yau threefold with boundary on a Lagrangian~$L$ of vanishing Maslov class were obtained by counting the curves by the values of their boundaries in the HOMFLYPT skein module of $L$. 
Here the Euler characteristic of the curve contributes to the power of the skein variable $z$ and a certain linking between $L$ and the curves contributes to the power of the framing variable $a$. 
In~\cite{Ekholm:2020csl} this perspective was used to generalize quantum curves from symmetrically colored HOMFLYPT polynomials to colorings by arbitrary partitions, in the case of the simplest knot, the unknot.  
We refer to this new type of recursion associated to the conormals $L_{K}$ of a link $K$ as \emph{skein valued recursion}. It is a collection of operator equations $\sfA_{K;j}\cdot \sf Z_{K}=0$, where $\sfA_{K;j}$ are linear combinations of operators in the skein module of the unit conormal of $K$ viewed as the boundary of $L_{K}$ and $\sf Z_{K}$ is the partition function in the skein of $L_{K}$. 

In this paper we describe broader geometric and physical frameworks for skein valued recursion. In order to address the naturality of skein valued curve counts we introduce a `universal' skein module $\Sk(X,L)$ of an oriented 3-submanifold $L$ in a 6-manifold $X$ equipped with an identification of its normal bundle $NL$ in $X$ with its cotangent bundle $T^{\ast}L$ and an induced almost complex structure in a neighborhood of $L$. We call it the \emph{worldsheet skein}. Elements in the worldsheet skein are formal linear combinations of embedded surfaces in $X$ with boundary in $L$ with $J$-complex tangent spaces along the boundary and with interior disjoint from $L$, up to skein relations. The skein relations are modeled on elliptic and hyperbolic boundary nodal degenerations of holomorphic curves and the corresponding curve families of their normalizations, and are direct counterparts of the usual HOMFLYPT skein relations. The invariance proof for the usual skein valued counting immediately implies that curve counts in the worldsheet skein are invariant as well.

In this setting the operator polynomials $\sfA_{K;j}$ are polynomials in the worldsheet skein $\Sk(\R\times ST^{\ast}S^{3};\R\times \Lambda_{K})$ of the unit conormal in the unit cotangent bundle. We call this worldsheet skein  
divided by the ideal generated by $\sfA_{K;j}$ the \emph{worldsheet skein $D$-module} of $\Lambda_{K}$, see Section \ref{sec:recursion-worldsheet} for details. 

Our main result, Theorem \ref{t:Hopf}, determines the worldsheet skein $D$-module of the unit conormal of the Hopf link, and shows that it determines the skein valued curve counts for any of its Lagrangian fillings. In the case of the conormal filling, the $D$-module generators give skein recursion identities for HOMFLYPT polynomials of the Hopf link colored by arbitrary partitions. 
From the viewpoint of the $A$-model topological string these relations characterize the generating series of all-genus worldsheet instantons, with boundaries on a stack with arbitrarily many branes wrapped on a Lagrangian filling.
The proof of Theorem \ref{t:Hopf} is geometric, and uses Symplectic Field Theory (SFT) moduli spaces of holomorphic curves at infinity \cite{EGH} on the Hopf link conormal. 

The Lagrangian fillings in Theorem \ref{t:Hopf} have the topology of the product of a torus and the line or of two solid tori. In the former case the partition function is trivial. In the latter we observe that the skein valued partition functions have `quiver-like' structures, compare~\cite{Kucharski:2017ogk, Ekholm:2018eee}, where the partition functions are generated by (skein valued multiple covers of) two basic holomorphic disks on each component (like for the unknot) and one or two basic annuli stretching between the components. For details see Theorem \ref{t: skein quiver 1},  Theorem \ref{thm:alternative quiver formula for Hopf} and Conjecture~\ref{c: skein quiver}.

\subsection*{Acknowledgements}
We are grateful to Marcos Mari\~no and Vivek Shende for discussions.
This work is supported by the Knut and Alice Wallenberg Foundation, KAW2020.0307 Wallenberg Scholar and by the Swedish Research Council, VR 2022-06593, Centre of Excellence in Geometry and Physics at Uppsala University and VR 2020-04535, project grant.

\section{Background, worldsheet skein $D$-modules, and main results}
In Section~\ref{ssec:intrU(1)} we review background results on recursion relations for quantum knot invariants. 
In Section~\ref{ssec:introskeinDmodule}, we give a brief account of worldsheet skein valued curve counting (details appear in Section \ref{ssec:reviewskeinsonbranes}) and recursion relations in that setting. 
In Section \ref{ssec:intromainresults} we state our main results, the description of the worldsheet skein $D$-module of the Hopf link and the skein valued quiver descriptions of corresponding partition functions.

\subsection{Augmentation varieties and $D$-modules for knots and links}\label{ssec:intrU(1)}
We review how quantum curves related to HOMFLYPT partition functions appear from symplectic geometry and topological string perspectives.

Consider a link $K=K_{1}\cup\dots\cup K_{m}\subset S^{3}$ with $m\ge 1$ connected components $K_{j}$. 
The conormal Lagrangian $L_{K}\subset T^{\ast}S^3$ is the family of cotangent planes along $K$ that annihilate its tangent vector field, topologically $L_{K}=\bigcup_{j=1}^{m}L_{K_{j}}\approx\bigsqcup_{j=1}^{m}S^1\times \IR^2$, a disjoint union of $m$ solid tori. At infinity, $(T^{\ast}S^{3},L_{K})$ is asymptotic to $(\R\times U^{\ast}S^{3},\R\times \Lambda_{K})$, where  
$U^{\ast}S^{3}$ is the unit cotangent bundle of $S^{3}$, with its standard contact structure given by the Liouville form and $\Lambda_{K}$ is the unit conormal of $K$ which is a Legendrian submanifold, $\Lambda_{K}=\bigcup_{j=1}^{m}\Lambda_{K_{j}}\approx\bigsqcup_{j=1}^{m} T^{2}$. 

In \cite{Aganagic:2013jpa} it was observed that the symplectic geometric knot contact homology and the physical topological string were connected via interactions between infinite area punctured holomorphic curves with boundary on $\R\times\Lambda_{K}$ and asymptotic to Reeb chords of $\Lambda_{K}$ at infinity and finite area curves with boundary on $L_{K}$.    

This is well-understood at the semi-classical level as follows. Reeb chords of $\Lambda_{K}$ correspond to binormal geodesics of the link $K$ graded by their Morse index (a non-negative integer).
The Chekanov-Eliashberg dg-algebra $CE^{\ast}(\Lambda_{K})$ of $\Lambda_K$ is freely generated over the group ring of $H_{1}(\Lambda_{K})$, i.e., the commutative ring $\C[x_{1}^{\pm},y_{1}^{\pm},\dots,x_{m}^{\pm},y_{m}^{\pm}]$, by Reeb chords. Its differential counts $\R$-invariant holomorphic disks in $\R\times U^{\ast} S^{3}$ with boundary on $\R\times\Lambda_{K}$, see~\cite{EGH,Ekholm:2018eee} and Section~\ref{ssec:groundstatesatinfinity}. In this context, $CE^{\ast}(\Lambda_{K})$ is known as \emph{knot contact homology}~\cite{ekholm2013knot}. Basic representations of $CE^{\ast}(\Lambda_{K})$ are chain maps $CE^{\ast}(\Lambda_{K})\to \C$, which can be identified with algebra maps in the degree $0$ Reeb chords that vanish on the polynomials that are in the image of the differential acting on degree $1$ chords. By elimination theory, the representation variety then projects to an algebraic variety $V_{K}\subset (\C^{\ast})^{2m}$, the \emph{augmentation variety}, defined by a set of polynomial equations,
\be
A_{K;j}(x_1,y_1,\dots,x_{m},y_{m}) \ = \ 0,\qquad j=1,\dots,n.
\ee

In \cite{Aganagic:2013jpa}, it was shown that the augmentation variety has an interpretation as the moduli space of Lagrangian fillings of $\Lambda_{K}$ and is locally parameterized by the corresponding disk potentials. Further, in \cite{Aganagic:2013jpa, Ekholm:2018iso}, first steps towards a (perturbative) understanding beyond the semi-classical level were taken. It was argued that the defining equations of the augmentation variety have a lift to operator equations 
that annihilate the full topological string partition function of the Lagrangian $Z_{K;{\text{top}}}$. The operator polynomials 
generate an ideal $\hat I_{K}$ in the (exponentiated) Weyl algebra $\C[\hat x_{1},\hat y_{1},\dots,\hat x_{m},\hat y_{m}]$. Concretely, the ideal is generated by a finite set of operator polynomials $\hat A_{K;j}(\hat x_{1},\hat y_{1},\dots,\hat x_{m},\hat y_{m})$, $j=1,\dots,n$, that then define the $D$-module 
\be\label{eq:D-module-def-background}
\hat D_{K} \ = \ \C[\hat x_{1},\hat y_{1},\dots,\hat x_{m},\hat y_{m}] \ / \ \hat I_{K}\,, 
\ee
with characteristic variety equal to $V_K$. From the point of view of this $D$-module, the partition function $Z_{K;{\text{top}}}$ corresponds, via \eqref{eq:U1-recursion-review}, to a homomorphism  
\be\label{eq:D-module-hom}
Z_{K;{\text{top}}} \ \in \ {\mathrm{Hom}}(\hat{D}_{K}, \C[x_{1},\dots,x_{m}])\,,
\ee
where $\C[x_{1},\dots,x_{m}]$ should be thought of as the ring of regular functions on $V_K$ locally parameterized by exponentiated generators $x_j$ of the first homology of $L$.

\subsubsection{Quantum curves from a physics perspective}
From the physical point of view, we consider topological string theory on $T^{\ast}S^3$ with one brane on each component $L_{K_{j}}$ of $L_{K}$, 
and $N$ branes on the zero section $S^{3}$. The topological string partition function of this brane system computes $U(N)$-HOMFLYPT polynomials of $K$ with $K_{j}$ colored by the $n_{j}$ single row (or symmetric) partitions, $H_{K;(n_1),\dots,(n_{m})}(q^{N},q)$, see \cite{Ooguri:1999bv}. In the large $N$ limit, the topological string background $T^{\ast}S^{3}$ transitions to the resolved conifold, the total space $X$ of $\CO(-1)\oplus\CO(-1)\to\C \IP^{1}$ and 
boundaries on $S^{3}$ of open string instantons (holomorphic curves) shrink and disappear, without affecting the partition function \cite{Gopakumar:1998ki}. As a result, the HOMFLYPT polynomials counts holomorphic curves in $X$ with boundary on $L_{K}\subset X$:
\be\label{eq:Z-open-abelian}
Z_{K;\text{top}}(\bx) \ = \sum_{n_1,\dots, n_m\geq 0} H_{K;(n_1),\dots,(n_{m})}(a,q)\, x_{1}^{n_{1}}\dots x_{m}^{n_{m}}\,,
\ee
where $a=q^N$ is the closed string K\"ahler modulus of $X$ and $x_j$ is the open string modulus of the $j^{\rm th}$ brane.

In the absence of instantons, this conormal brane system is described by a worldvolume theory, (complexified) abelian Chern-Simons gauge theory. Let $(\hat x_{j},\hat y_{j})$ denote the longitude and meridian holonomy operators on the $j^{\rm th}$ boundary torus $\partial L_{K_{j}}\approx T^2$ in this theory. Then
\be\label{eq:Weyl-algebra-review}
\hat y_i\hat x_j \ = \ q^{2\delta_{ij}} \hat x_j\hat y_i 
\ee
and $(\hat y_j,\hat x_j)$, $j=1,\dots,m$ generate an exponentiated Weyl algebra of line operators in $\partial L_{K}$.

String instantons deform the worldvolume theory by the insertion of Wilson loop operators along their boundaries in~$L_{K}$, see \cite{Witten:1992fb}. From this point of view, \eqref{eq:Z-open-abelian} is the (normalized) expectation value of the instanton Wilson lines in Chern-Simons theory on $L_K$. In other words, \eqref{eq:Z-open-abelian} can be regarded as a wavefunction 
for a quantum state $|Z_{K}\rangle\in \CH(\partial L_K)$ in the Hilbert space of $\partial L_K\approx \bigsqcup_{j=1}^{m}T^{2}$, defined by the Chern-Simons  path integral on $L_K$.

If $K$ is a knot ($m=1$) then the quantum curve is an operator $\hat A_{K}(\hat x_{1},\hat y_{1})$ that annihilates $Z_{K;\text{top}}(x_{1})$ \cite{AganagicVafa2012} and gives a recursion relation for its symmetrically colored HOMFLYPT polynomials. In the many component case ($m>1$), the `quantum curve' is an exponentiated $D$-module \cite{Aganagic:2013jpa} generated by a finite collection of operator polynomials $\hat A_{K;j}(\hat x_1, \hat y_1,\dots,\hat x_{m},\hat y_{m})$, $j=1,\dots,n$, that all annihilate $Z_{K;\text{top}}(\bx)$, 
\be\label{eq:U1-recursion-review}
\hat A_{K;j} \cdot Z_{K;\text{top}} \ = \ 0, \qquad j=1,\dots,n.
\ee

\subsection{Skein valued curve counts, worldsheet skein $D$-modules and recursions}\label{ssec:introskeinDmodule}
In~\cite{Ekholm:2019yqp}, a new approach to open curve counting was developed, where holomorphic curves in a Calabi-Yau threefold with boundary on a Maslov index zero Lagrangian $L$ with $r$ connected components $L=L_{1}\cup\dots\cup L_{r}$ are counted by their boundaries in the framed HOMFLYPT skein module of $L$. 

The framed HOMFLYPT skein module is generated as a module over the ring\linebreak $\mathbb{Q}[a_{L_{1}}^{\pm},\dots,a_{L_{r}}^{\pm},z^{\pm}]$, where $z=(q-q^{-1})$, by all isotopy classes of framed links in $L$ modulo the three framed skein relations \eqref{eq:skein-rel-intro}. For curve counts, the power of $z$ corresponds to the Euler characteristic, and a certain linking between $L_{j}$ and the curves contributes to the power of $a_{L_{j}}$. The role of $q$ and $a_{L_{j}}$ can be understood from a topological string perspective, in the construction of \cite{Ooguri:1999bv}, wrap $M_{j}$ branes on $L_{j}$ then $q$ corresponds to the exponentiated string coupling and $a_{L_{j}} = q^{M_{j}}$.

Here we use a version of skein valued curve counts with nice transformation properties. We use a counterpart of skein modules that we call \emph{worldsheet skein} modules. The worldsheet skein module $\Sk(X,L)$ of a Lagrangian $L$ in a Calabi-Yau $3$-fold $X$ is the $\mathbb{Q}$-vector space generated by isotopy classes of embedded surfaces with boundary in $L$ that are holomorphic along the boundary and have interior disjoint from $L$, divided by two worldsheet skein relations that are direct counterparts of the HOMFLYPT skein relations, see Section \ref{ssec:surfaceskein} for the detailed definition. Fixing a 4-chain for $L$ we get a homomorphism $\rho_L: \, \Sk(X,L)\to\Sk(L)$ that takes a surface to its boundary in the skein times $z^{-\chi}a^{l}$, where $\chi$ is the Euler characteristic of the surface and $l$ is the intersection number with the 4-chain.   

In this context \eqref{eq:D-module-def-background} has a worldsheet skein valued counterpart as follows. As in knot contact homology, we think of $V_{K}$ as the moduli space of Lagrangian fillings $L$ of $\Lambda_{K}$. For such $L$, the skein valued curve count $\sfZ^\textnormal{ws}_{L}\in \widehat{\Sk}(T^*S^3,L)$ gives a homomorphism from the degree $0$ part of a certain BRST-homology, which is the worldsheet skein counterpart of knot contact homology, to the worldsheet skein module of $L$. 

Furthermore, the part of the degree zero BRST-homology which is generated by the empty Reeb chords (on all components) takes the form 
\[
\sfD_K^{\ws} \ = \ \Sk(\R\times U^{\ast}S^3,\R\times \Lambda_{K}) \ / \ \sfI_{K}^{\ws},
\]   
where the ideal $\sfI_{K}^{\ws}$ is generated by a finite number of worldsheet skein operators $\sfA_{K;j}^{\ws}$, $j=1,\dots,n$, which when composed with the homomorphism $\rho_{\partial L}$ to $\Sk(\partial L) \approx \Sk(\IR\times \Lambda_K)$ all annihilate $\sfZ_{L}$, 
	\be
	\rho_{\partial L}\left(\sfA_{K;j}^{\ws}\right)\cdot \sfZ_{L}=0, \quad j=1,\dots,n\,.
	\ee
Here $\rho_{\partial L}$ indicates that we evaluate in $\Sk(\partial L)$ using its $4$-chain.
It follows that
	\[
	\sfZ_{L} \ \in \ \mathrm{Hom}(\sfD_{K}^{\ws},\widehat{\Sk}(L)).
	\] 

We call $\sfD_K^{\ws}$ the \emph{worldsheet skein $D$-module} of $K$. Worldsheet skein $D$-modules and associated skein recursion relations and partition functions are the main subjects of this paper.

\subsubsection{Notation for basic skein modules}
The skein modules of the thickened and solid tori $T^{2}\times\R$ and $S^{1}\times \R^{2}$, respectively, were studied extensively in~\cite{morton2002power,lukac2005idempotents,lukac2001homfly}. 
The skein module of the torus $\Sk(T^{2})$ has a basis $\sfP_{i,j}$, 
where $\sfP_{i,j}$ is a curve in homology class $(i,j)\in H_{1}(T^{2})=H_{1}(S^{1})\otimes H_{1}(S^{1})$. It admits an algebra structure, where multiplication corresponds to stacking links, characterized by the relation
\[
\sfP_{i,j}\sfP_{k,l}-\sfP_{k,l}\sfP_{i,j} \ = \ \left(q^{(il-kj)}-q^{(kj-il)}\right)\sfP_{i+k,j+l}.
\]  

The skein module of the solid torus $\Sk(S^{1}\times\R^{2})$ has a basis $\sfW_{\lambda,\overline{\mu}}$, where $\lambda$ and $\mu$ range over all partitions, and $\overline{\mu}$ denotes the conjugate partition of $\mu$. 
(Physically this is understood in terms of orientation reversal of Wilson lines in the worldvolume theory on the branes. Conjugation of partitions depends on the rank of the group whose representations they label, and when working in arbitrary rank a Wilson line with reversed orientation is labeled by a conjugate partition.)
Here $\sfW_{\lambda,\overline{\mu}}$ lies 
in homology class $|\lambda|-|\overline{\mu}|$ in $H_{1}(S^{1}\times\R^{2})\approx\IZ$, where $|\lambda|$ denotes the number of boxes in the (Young)-diagram of $\lambda$. Furthermore, 
the skein algebra $\Sk(T^{2})$ acts on $\Sk(S^{1}\times\R^{2})$ again by stacking, viewing $T^{2}$ as the (ideal) boundary of $S^{1}\times\R^{2}$. Under this action, $\sfW_{\lambda,\overline{\mu}}$ is a basis of eigenvectors of the operator $\sfP_{1,0}$, all with distinct eigenvalues. In physics this corresponds to a generalization to arbitrary rank of Verlinde's basis for the boundary Hilbert space of the worldvolume Chern-Simons theory \cite{Verlinde:1988sn, Witten:1988hf}.
This feature of the $\sfW_{\lambda,\overline{\mu}}$ basis played an important role in elucidating a connection between HOMFLYPT and Kauffman polynomials \cite{Marino:2010ppv}.

\subsection{Main results}\label{ssec:intromainresults}
Our main results give worldsheet skein operator ideals and closed form partition functions for all Lagrangian fillings of the Hopf link conormal Legendrian. 

Let $\Gamma=\Gamma_{1}\cup\Gamma_{2}\subset S^{3}$ denote the negative Hopf link and let 
$\Lambda_{\Gamma}=\Lambda_{\Gamma_{1}}\cup\Lambda_{\Gamma_{2}}\approx T^2\sqcup T^2$ denote the Legendrian conormal of $\Gamma$. 
Let $L$ be a Lagrangian filling of $\Lambda_{\Gamma}$ with $H_{1}(L)\approx\IZ^{2}$. Below $L$ will have the topology either of two solid tori $S^1\times\R^2$ or the product of a torus and the real line $T^{2}\times\R$. The action completed skein of $L$, $\widehat{\Sk}(L)$, see Section \ref{ssec:univcurvecounts}, will play the role of the sheaf of regular functions (in the local open string coordinates) on the augmentation variety, indeed specializing to the $U(1)$ skein we get usual regular functions.   

The partition function of $L$ is an element $\sfZ_{L}\in\widehat{\Sk}(L)$ which, as explained above gives a homomorphism from the world sheet skein $\Sk(\R\times ST^{\ast}S^{3},\R\times\Lambda)/\sfI^{\ws}_{\Gamma}$ to $ \widehat{\Sk}(L)$. If $\eta_{1},\eta_{2}$ is a basis for $H^{1}(L)$ and $\beta_j\in\{-1,0,1\}$, $j=1,2$, then $(\eta_{1},\beta_1,\eta_{2},\beta_2)$ determines an \emph{initial condition} for partition functions $\sfZ_{L}$ which requires that if we expand $\sfZ_{L}$ as  
\[
\sfZ_{L}=\sum_{(\xi_{1},\xi_{2})\in H_{1}(L),\alpha\in\mathbb{Z}} \mathsf{B}_{\xi_{1},\xi_{2},\alpha},
\]
where $\mathsf{B}_{\xi_{1},\xi_{2},\alpha}$ is the component of $\sfZ_{L}$ in the homology class $(\xi_{1},\xi_{2})$ of homogeneous $a$-degree $\alpha$ then $\mathsf{B}_{\xi_{1},\xi_{2},\alpha}=0$ for all $(\xi_{1},\xi_{2},\alpha)$ such that $\eta_{1}(\xi_{1},\xi_{2})+\beta_1\cdot\alpha<0$ or $\eta_{2}(\xi_{1},\xi_{2})+\beta_2\cdot\alpha<0$.

If $L_{1}$ and $L_{2}$ are two Lagrangians that intersect cleanly along a circle then we can form new Lagrangians by surgery on the intersection: remove small neighborhoods of the circle in each Lagrangian and join the resulting boundaries. Lagrangians that intersect cleanly along a circle admits skein modules with extra relations as curves cross the intersection locus, see Remark \ref{r: formal fillings}. Here such singular Lagrangian fillings $L'$ sit at singular points in the augmentation variety and the skein $D$-module encodes also information about the deformations of $L'$ corresponding both to shifting the intersection off and smoothing the intersection.  

Finally, as above, we write $X$ for the resolved conifold. We note that if $L\subset T^{\ast}X$ is a Lagrangian disjoint from $S^{3}$ then there is a canonical isomorphism $\Sk(T^{\ast}S^{3},L\cup S^3) \cong \Sk(X,L)$: any surface in $T^*S^3$ is equal in the world-sheet skein to a linear combination of surfaces in the complement of the zero-section, and conversely, a generic surface in $X$ does not intersect the central $S^2$.

\begin{thm}\label{t:Hopf}
The augmentation variety $V_{\Gamma}$ of the negative Hopf link $\Gamma$ is a two-dimensional subvariety of $(\C^{\ast})^{4}$, which is Lagrangian with respect to the natural holomorphic symplectic form. Each point in $V_{\Gamma}$ corresponds to a Lagrangian filling $L\subset X$ or $L\subset T^{\ast}S^{3}$ of $\Lambda_{\Gamma}$ that can be constructed by surgery from the conormals of the components of $\Gamma$ and the 0-section in $T^{\ast}S^{3}$. 

The worldsheet skein $D$-module of $\Gamma$ is 
\be
	\sfD_\Gamma^{\ws} = \Sk(\IR\times U^*S^3,\IR\times\Lambda_{\Gamma}) / \sfI_\Gamma^{\ws}
\ee
where $\sfI_\Gamma^{\ws}$ is generated by the following three operators 
\begin{align}\label{eq:HopfA1} 
	\sfA_1^{\ws}
	\ &= \
	\(
	\sfP^{\ws;(1)}_{1;0,0}
	- \sfP^{\ws;(1)}_{1;1,0}
	-\sfP^{\ws;(1)}_{1;0,1}
	+\sfP^{\ws;(1)}_{1;1,1}
	\)\\ \notag
	&\quad- 
	\(
	\sfP^{\ws;(2)}_{1;0,0}
	- \sfP^{\ws;(2)}_{1;1,0}
	-\sfP^{\ws;(2)}_{1;0,1}
	+\sfP^{\ws;(2)}_{1;1,1}
	\),\\ \label{eq:HopfA2}
	\sfA_2^{\ws} 
	\ &= \
     \(
	-\sfP^{\ws;(1)}_{2;-1,0}
	+ \sfP^{\ws;(1)}_{2;0,0}
	+  \sfP^{\ws;(1)}_{2;-1,1}
	- \sfP^{\ws;(1)}_{2;0,1}
	\)\\\notag
	&\quad+ 
	\(\sfP^{\ws;(2)}_{2;0,-1}
	-\sfP^{\ws;(2)}_{2;1,-1}
	-\sfP^{\ws;(2)}_{2;0,0}
	+\sfP^{\ws;(2)}_{2;1,0}
	\),\\ \label{eq:HopfA3}
	\sfA_3^{\ws}
	\ &= \
	\(
    -\sfP^{\ws;(2)}_{3;-1,0}
	+ \sfP^{\ws;(2)}_{3;0,0}
	+ \sfP^{\ws;(2)}_{3;-1,1}
	-\sfP^{\ws;(2)}_{3;0,1}
	\) \\\notag
	&\quad+
	\(
	 \sfP^{\ws;(1)}_{3;0,-1}
	- \sfP^{\ws;(1)}_{3;1,-1}
	-\sfP^{\ws;(1)}_{3;0,0}
	+\sfP^{\ws;(1)}_{3;1,0}
	\).
\end{align}
Here the Euler characteristic of $\sfP^{\ws;(j)}_{k;p,q}$ equals $1$ and the boundary represents $\sfP^{(j)}_{p,q}\in\Sk(\Lambda_{\Gamma})$, up to a monomial in framing variables.

More precisely, each of the Lagrangian fillings in $L\in V_{\Gamma}$ determines an initial condition for partition functions, and there is a unique partition function $\sfZ_{L}\in\widehat{\Sk}(L)$ annihilated by $\sfA_{j} = \rho_L(\sfA^{\ws}_j)$, $j=1,2,3$, that satisfies the initial condition determined by~$L$. 
In particular, if $L=L_\Gamma$ is the Lagrangian conormal filling of the Hopf link, we find that the generating series of the colored HOMFLYPT polynomials $H_{\Gamma;\lambda,\mu}(a,q)$ colored by arbitrary partitions $\lambda,\mu$ on each component  $\Gamma_{1}, \Gamma_{2}$ 
\be\label{eq: Hopf HOMFLY partitionfunction}
	\sfZ_{L} \ = \ \sum_{\lambda,\mu} H_{\Gamma;\lambda,\mu}(a,q)\ \sfW_{\lambda,\emptyset}\otimes \sfW_{\mu,\emptyset} \ \in \ \widehat{\Sk}(L),
\ee
is the unique solution to the skein recursion 
\begin{align}
	\label{eq:A1-conormalintr}
	\sfA_1 
	\ &= \
	\(
	\sfP^{(1)}_{0,0}
	- \sfP^{(1)}_{1,0}
	-a_{L_1} \sfP^{(1)}_{0,1}
	+a^2 \sfP^{(1)}_{1,1}
	\)\otimes a_{L_2}\\ \notag
	&\quad- a_{L_1}\otimes
	\(
	\sfP^{(2)}_{0,0}
	- \sfP^{(2)}_{1,0}
	-a_{L_2} \sfP^{(2)}_{0,1}
	+ a^2 \sfP^{(2)}_{1,1}
	\),\\ 
	\label{eq:A2-conormalintr}
	\sfA_2 
	\ &= \
	a_{L_1}
	\(
	-\sfP^{(1)}_{-1,0}
	+ \sfP^{(1)}_{0,0}\)\otimes 1
	+  \(a_{L_{1}}\sfP^{(1)}_{-1,1}
	- a^2 \sfP^{(1)}_{0,1}
	\)\otimes (a_{L_2})^2\\\notag
	&\quad+ 
	1\otimes a_{L_{2}}\((a_{L_{2}})^{-1}\sfP^{(2)}_{0,-1}
	-\sfP^{(2)}_{1,-1}
	-\sfP^{(2)}_{0,0}
	+a^2 \sfP^{(2)}_{1,0}
	\),\\ 
	\label{eq:A3-conormalintr}
	\sfA_3
	\ &= \
	1\otimes a_{L_2}
	\(
	-\sfP^{(2)}_{-1,0}
	+ \sfP^{(2)}_{0,0}\)
	+ (a_{L_1})^2\otimes \(a_{L_{2}}\sfP^{(2)}_{-1,1}
	- a^2 \sfP^{(2)}_{0,1}
	\) \\\notag
	&\quad+a_{L_1}
	\(
	(a_{L_1})^{-1} \sfP^{(1)}_{0,-1}
	- \sfP^{(1)}_{1,-1}
	-\sfP^{(1)}_{0,0}
	+a^2 \sfP^{(1)}_{1,0}
	\)\otimes 1\,,
\end{align}  
with initial conditions that the coefficients for $\sfW_{\lambda,\overline{\nu}}\otimes \sfW_{\mu,\overline{\rho}}$ vanish whenever $\overline{\nu}\ne \emptyset$ or $\overline{\rho}\ne \emptyset$, and the coefficient of $\sfW_{\emptyset,\emptyset}$ equals $1$. 

Analogous operators and partition functions for Lagrangian fillings at any point in its augmentation variety are given in Section \ref{sec:hopf}.
\end{thm}
Theorem~\ref{t:Hopf} is proved in Section~\ref{ssec:proofThmHopf}. 

We next discuss closed form expressions for the skein valued partition function of the various Lagrangian fillings of the Hopf link Legendrian conormal $\Lambda_{\Gamma}$. 
There are both disconnected and connected fillings.
The former are classified in terms of fillings associated to the knot components of the Hopf link, both of which are unknots. 
The Legendrian conormal $\Lambda_{U}$ of the unknot $U$ has three natural fillings in the resolved conifold. We denote them according to how they appear in a toric diagram (see Section \ref{sec:unknot}): the `exterior leg' fillings $L_{l^{\pm}}$ and $L_{m^{\pm}}$, which correspond to the the conormal and the complement, respectively, and the `middle leg' filling $L_{d}$. All three are topologically solid tori $S^{1}\times\R^2$. 
The Legendrian Hopf link conormal $\Lambda_{\Gamma}$ then has disconnected fillings $L_{st}=L_{s}\cup L_{t}$, labeled by $s,t\in\{l^{\pm},m^{\pm},d\}$. For convenient notation we write $|l^{\pm}|=l$, $|m^{\pm}|=m$, and $|d|=d$. The Legendrian conormal also has a connected filling, the link complement with topology $\R\times T^{2}$, which we denote $L_{0}$. 

We use the uniqueness part of Theorem~\ref{t:Hopf} to show that the partition functions of the various fillings admit expressions in terms of multi-covers of finitely many basic curves, a skein valued analogue of the quiver generating series of knot polynomials, see~\cite{Kucharski:2017ogk, Ekholm:2018eee,Ekholm:2019lmb}, where the 
building blocks of skein-valued quivers are the partition functions for (framed) disks and annuli \cite{Ekholm:2020csl, Ekholm:2021colored} 
\be\label{eq : basic part functions intro}
\begin{split}
	\sfPsi_{\mathrm{di}}^{(f)}[\xi] \ & :=  \ \sum_{\lambda} \left[(-1)^{|\lambda|} q^{\kappa(\lambda)}\right]^{f} \prod_{\ydiagram{1}\in\lambda } \frac{-q^{-c(\ydiagram{1})}}{q^{h(\ydiagram{1})} - q^{-h(\ydiagram{1})}} \ \xi^{|\lambda|}\sfW^{}_{\lambda,\emptyset},\\
	\sfPsi_{\mathrm{an}}^{(f_1, f_2)}[\xi] \ 
	& = \ \sum_{\lambda}  
	\left[(-1)^{|\lambda|} q^{\kappa(\lambda)}\right]^{f_1+f_2} \
	|\xi|^{|\lambda|}
	\sfW_{\lambda,\emptyset}\otimes \sfW_{\lambda,\emptyset},\\
\end{split}
\ee
and certain generalizations of them which are collected in Section \ref{sec:basic-curves}.

The unknot conormal partition function was determined in \cite{Ekholm:2020csl}. It was shows in \cite{2024arXiv240110730N} that it is equal to the product of the partition functions of two disks
$\sfZ^{}_{L_{l^+}}  = \sfPsi_{\mathrm{di}}^{(0)}\,  \cdot \, \sfPsi_{\mathrm{di}}^{(1)}[a^2]$ and noting that the complement is the conormal of a dual unknot, the unknot complement partition function satisfies $\sfZ_{L_{d^+}} 
= 
\sfPsi_{\mathrm{di}}^{(1)}\,  \cdot \, \sfPsi_{\mathrm{di}}^{(0)}[a^2]$.
We insert these along links in Lagrangians $L$ as follows. Given a framed knot or a pair of framed knots in $L$, remove a small neighborhood of it (identified with a solid torus), and consider the skein element represented by $\sfPsi_{\mathrm{di}}^{(f)}$ or by $\sfPsi_{\mathrm{an}}^{(f_1, f_2)}$ included in the neighborhood. We have the following result for $L_{0}$ and $L_{st}$ where $|s|\ne |t|$. 

\begin{thm}\label{t: skein quiver 1}$\quad$
	\begin{itemize}
		\item[$(a)$] The skein valued partition function $\sfZ_{L_{0}}$ of $L_{0}$ equals $\sfZ_{L_{0}}=1$ (corresponding to no holomorphic curves).
		\item[$(b)$] The skein valued partition function $\sfZ_{L_{st}}$ of a filling $L_{st}$, $|s|\ne |t|\in\{l,m,d\}$ is obtained by inserting the two disks of the corresponding unknots and one annulus stretching between the components. 
		 See Figure \ref{fig:hopf-link-quiver-description-conormal-complement} and \eqref{eq:conormal-complement-Z} for the case of $L_{l^{+}m^{+}}$ and Section \ref{sec:conormal-complement-solution} for explicit formulas in all cases.  
	\end{itemize}
\end{thm}

\begin{figure}[h]
	\begin{center}
		\includegraphics[width=0.6\textwidth]{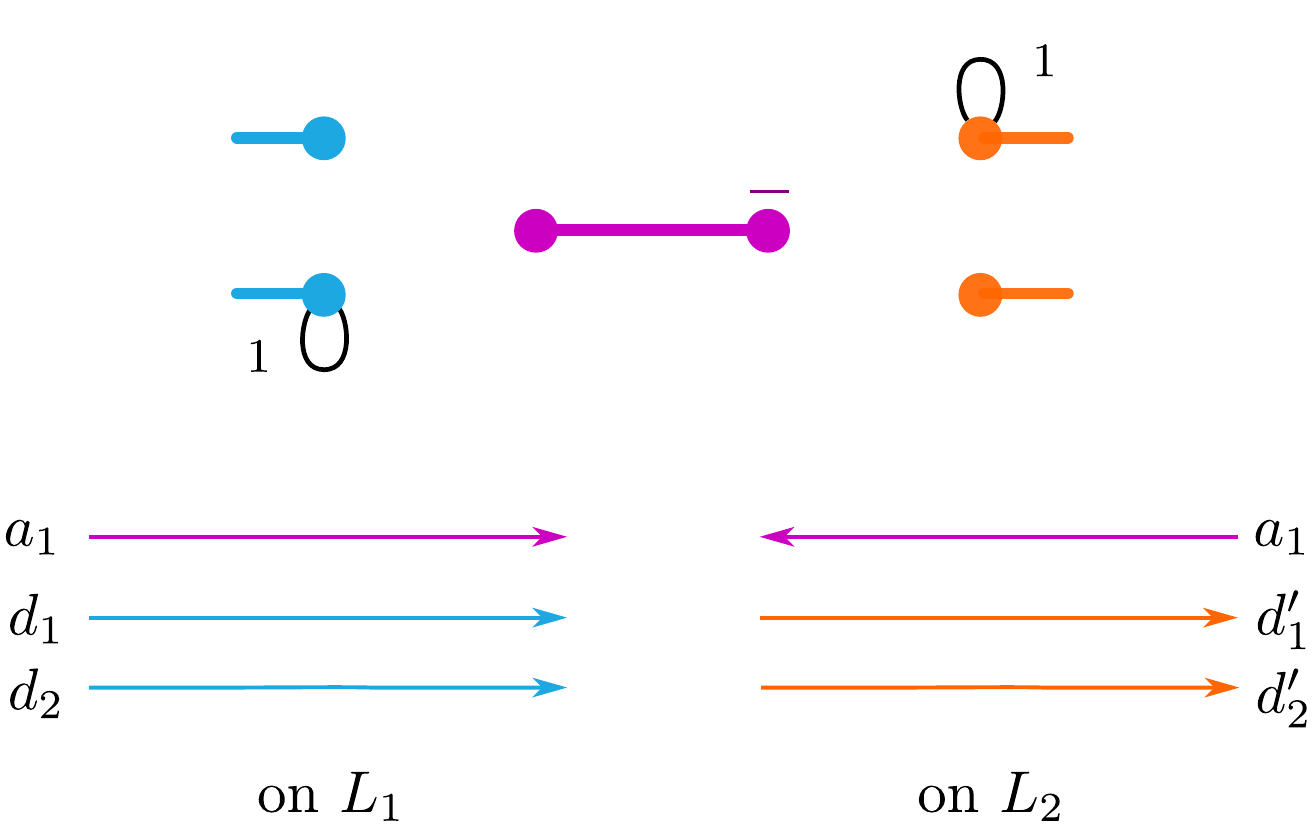}
		\caption{Basic disks and annuli in the quiver description of the partition function $L_{lm}$. On each brane there are two disks ($d_{1},d_{2}$ and $d_{1}',d_{2}'$) that are mutually unlinked. The disks $d_2, d_2'$ come with a power of $a^2$.
		The two branes are connected by a twisted annulus $a_1$ that does not link with any of the disks. The twisted annulus partition function $\sfPsi_{\widetilde{\mathrm{an}}}$ is obtained from $\sfPsi_{\mathrm{an}}$, see \eqref{eq : basic part functions intro}, by replacing $W_{\lambda,\emptyset}$ with $W_{\emptyset,\overline{\lambda}}$ in the second factor. 
		}
		\label{fig:hopf-link-quiver-description-conormal-complement}
	\end{center}
\end{figure}

We conjecture the following similar forms of the partition functions for $L_{st}$, where $|s|=|t|$: 
\begin{cnj}\label{c: skein quiver}
	The partition functions $\sfZ_{L_{st}}$ for $|s|=|t|$ 
	is generated by two basic disks on each Lagrangian component and two basic annuli that differ by a factor of $a^{2}$. For example, the partition function $\sfZ_{L_{l^{+}l^{+}}}$ is obtained by inserting the partition functions $\sfPsi_{\mathrm{di}}^{(0)}, \sfPsi_{\mathrm{di}}^{(1)}[a^2]$ and $\sfPsi_{\mathrm{an}}^{(-1,-1)}$,  
	$\sfPsi_{\overline{\mathrm{an}}}^{(0,0)}[a^2]$ along the curves in Figure~\ref{fig:hopf-link-quiver-description}, see Section~\ref{sec:skein valued-quiver} for (more) detailed formulas. 
\end{cnj}

\begin{figure}[h]
	\begin{center}
		\includegraphics[width=0.6\textwidth]{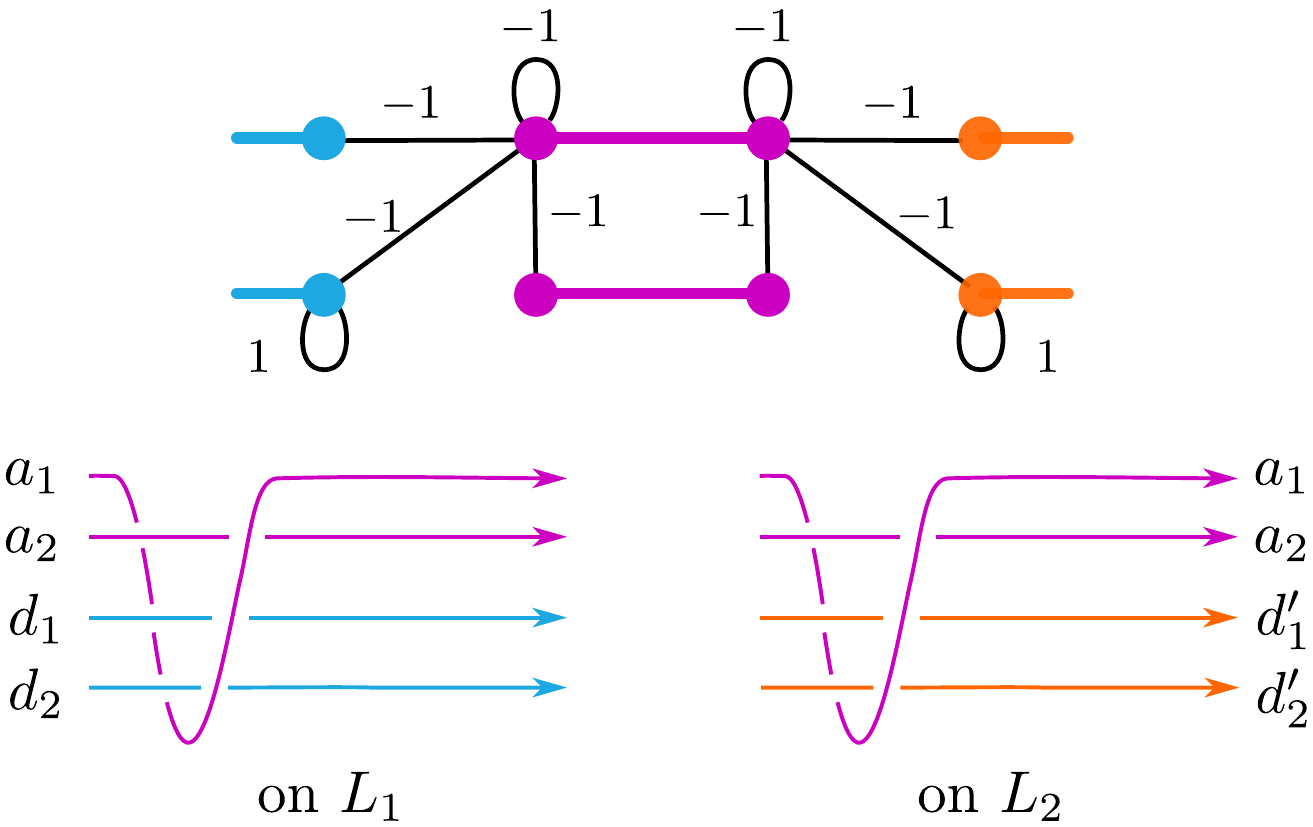}
		\caption{Basic disks and annuli in the quiver description of $\sfZ_{L_{l^{+}l^{+}}}$. On each brane there are two disks ($d_{1},d_{2}$ and $d_{1}',d_{2}'$) that are mutually unlinked (in a suitable choice of framing). Both disks link with one of the annuli ($a_{1}$). The second annulus ($a_{2}$) only links the first annulus, on each brane. Disks $d_2, d_2'$ and annulus $a_2$ come with a power of $a^2$.}
		\label{fig:hopf-link-quiver-description}
	\end{center}
\end{figure}

Specializing the skein valued expression for $\sfZ_{L_{l^{+}l^{+}}}$ in Conjecture \ref{c: skein quiver} to $U(1)$ gives a generalized quiver that except for disk nodes also have annulus nodes. This expression also equals the generating function of the HOMFLYPT polynomials of the Hopf link for all symmetric partitions. For more general partitions, we checked Conjecture~\ref{c: skein quiver} numerically when the number of boxes is less than eight. In Section~\ref{sec:mdliproof} we outline a proof of Conjecture~\ref{c: skein quiver} that uses moduli spaces of holomorphic curves near the SFT stretched limit, a more thorough description of moduli spaces of multiply covered curves is needed in order to make that argument rigorous. In Theorem \ref{thm:alternative quiver formula for Hopf} we give an alternate closed form expression for the Hopf link partition function in terms of two disks and a somewhat formal three holed sphere.

\section{Augmentation varieties and curve counts}\label{sec:backgroundCS}
In this section we discuss recursive properties of knot invariants. We review how recursion relations arise from the viewpoint of topological string theory via the Floer theoretic framework of knot contact homology.

\subsection{The augmentation variety and disk potentials}\label{sec:VK-review}
Let $K\subset S^{3}$ be an $m$-component link and let $\Lambda_{K}\subset U^{\ast} S^{3}$ be its conormal Legendrian. Recall from Section \ref{ssec:intrU(1)} that the knot contact homology of $K$ is the Chekanov-Eliashberg dg-algebra $CE^{\ast}(\Lambda_{K})$, which is a Floer theoretical algebra over the group ring 
\[
\C[H_{1}(\Lambda_{K})] \ = \ \C[x_{1}^{\pm},y_{1}^{\pm},\dots,x_{m}^{\pm},y_{m}^{\pm}],
\]   
generated by Reeb chords of $\Lambda_{K}$
(flow lines of the Reeb vector field, analogous to Lagrangian intersection points in Floer homology, see Section~\ref{ssec: strings at infinity}) 
where the differential counts holomorphic disks up to $\R$-translation in $\R\times U^{\ast} S^{3}$ with boundary on $\R\times\Lambda$, see Section~\ref{ssec:groundstatesatinfinity}.  Recall also that the augmentation variety $V_{K}\subset (\C^{\ast})^{2m}=\{(x_{1}^{\pm},y_{1}^{\pm},\dots,x_{m}^{\pm},y_{m}^{\pm})\}$ is the algebraic variety that is the  locus where there exist chain maps $CE^{\ast}(\Lambda_{K})\to\C$. In all known examples, $V_{K}$ is a Lagrangian subvariety of $(\C^{\ast})^{2m}$ with respect to the symplectic form $\sum_{j} d\log x_{j}\wedge d\log y_{j}$, in line with the interpretation of $V_{K}$ as the characteristic variety of a $D$-module.

The augmentation variety has several asymptotic regions at infinity. If $(u_1,\dots,u_{m})$, such that $(\log u_{1},\dots,\log u_{m})$ parameterizes a linear Lagrangian subspace, with dual coordinates $(\log v_{1},\dots,\log v_{m})$, are coordinates along the asymptotic Lagrangian of a branch $V^{(\alpha)}_{K}$ in an asymptotic region, then the variety $V_{K}$ is locally parameterized by 
\[
\log v_{j} \ = 
\ \frac{\partial W^{(\alpha)}}{\partial\log u_{j}}(u_{1},\dots,u_{m})\,,
\]   
where $W_{K}^{(\alpha)}$ is the (formal) disk potential along the given branch. In case the augmentations along $V^{(\alpha)}_{K}$ correspond to a Lagrangian filling $L^{(\alpha)}$ of $\Lambda_{K}$ then $W^{(\alpha)}$ is given by an actual count of (generalized) holomorphic disks on $L^{(\alpha)}$, see \cite{Aganagic:2013jpa}.

Consider for example the (negative) Hopf link $\Gamma$. There is one branch of $V_{\Gamma}^{(2)}$ that corresponds to the conormal filling $L_{ll}$ with coordinates $(x_{1},x_{2})$. Here
\be
\log y_{j} \ = \ \frac{\partial W_{U}}{\partial\log x_{j}}(x_{j}),\quad j=1,2,
\ee
where $W_{U}$ denotes the disk potential of the unknot conormal. There is another branch $V_{\Gamma}^{(1)}$ that corresponds to the complement filling 
$L_0 \simeq S^3\setminus K$. Since $L_0$ is exact in the complement of $S^{3}$ in $T^{\ast} S^{3}$, its disk potential is trivial $W^{(0)}=0$, the parameterizing coordinates are 
\be
u_{1} =  y_1 \, x_{2}^{-\ell_{12}} = y_{1}\,x_{2},\qquad u_{2} = y_{2}\, x_{1}^{-\ell_{12}}=y_{2}\,x_{1},  
\ee
where $\ell_{ij}\in \IZ$ denotes the linking number between the link components in $S^3$, see~\cite{Aganagic:2013jpa}.

In general it is an open question whether, for a given knot or link, one can construct (generalized) Lagrangians with disk potentials that parameterize all branches of the augmentation variety. In known examples, different branches of $V_K$ intersect along (complex) codimension one loci, which is a property of the characteristic variety of an irreducible $D$-module. We point out that if a branch corresponds to a Lagrangian filling $L^{(\alpha)}$, then that branch is determined by the disk potential of $L^{(\alpha)}$. In examples, it seems that the global structure of the augmentation variety is governed also by holomorphic curves of lower Euler characteristics {with boundaries} on certain `generating' (in the Fukaya category sense) Lagrangian fillings.

We next recall how augmentation varieties are related to quantum knot invariants.
It follows from~\cite{Ooguri:1999bv,Ekholm:2019yqp} that counts of holomorphic curves on the conormal Lagrangian $L_{K}\subset X$ are given by HOMFLYPT polynomials. Therefore, the recursion relations $\hat A_{K;j}$ for HOMFLYPT polynomials of $K$ in symmetric representations, see \cite{1998math.....12048F, 2003math......9214G, garoufalidis2018colored}, written as operator equations in operators $(\hat x_{j},\hat y_{j})$, $j=1,\dots,m$ %$\hat y_{j}\hat x_{j}=q^{2}\hat x_{j}\hat y_{j}$ 
corresponding to Weyl quantization of $(x_{j},y_{j})$, annihilate the wave function counting holomorphic curves on $L_{K}$, see~\eqref{eq:U1-recursion-review}. 
More generally, coordinates $u_{j}$ of an asymptotic branch $V_{K}^{(\alpha)}$ correspond to a choice of polarization, and also specifies an initial condition, $\log v_{j}\to 0$. There are accordingly many solutions to the recursion relation corresponding to different choices of polarizations and initial conditions. 
Taking the semi-classical limit of such a wave function solution $Z^{(\alpha)}$ gives the disk potential $W^{(\alpha)}$ of the corresponding branch,
\be
Z^{(\alpha)}(u_1,\dots,u_{m}) \ = \ \exp\left(\frac{1}{g_s} W^{(\alpha)}(u_1,\dots,u_{m}) + O(1) \right)
\ee 
and its support localizes on the critical manifold of $W^{(\alpha)}$, which coincides with the asymptotic branch $V^{(\alpha)}_{K}$ of $V_K$. Thus, in the limit $g_{s}\to 0$ (or equivalently $q=e^{\frac{g_s}{2}}\to 1$, compare~\cite{Ekholm:2018iso}), the recursion relations expressed by operator equations become polynomials in commutative variables that cut out the augmentation variety $V_K$, see \cite{AganagicVafa2012}.

\subsection{Skein curve counts and the conifold transition}\label{ssec:reviewskeinsonbranes}
In this section we introduce the world sheet skein and review skein valued curve counts in that setting.

\subsubsection{The worldsheet skein}\label{ssec:surfaceskein}
Consider an oriented $3$-manifold $L$ and its cotangent bundle $T^{\ast}L$ with its standard symplectic form $\omega=d(pdq)$, where $pdq$ is the Liouville form. Let $J$ be an almost complex structure on $T^{\ast}L$ {along $L$} that is compatible with $\omega$. We note that such $J$ are unique up to {contractible choice} and that they admit extensions to neighborhoods of $L$ in $T^{\ast}L$ that are compatible with $\omega$ and that any two such extensions are homotopic on some neighborhood of $L$.   

Consider a 6-manifold $W$ and a properly embedded, oriented 3-dimensional submanifold $L\subset W$ with normal bundle $NL$ and an isomorphism $NL\to T^{\ast}L$. (Existence of such an isomorphism is equivalent to $NL$ being trivial). We say that $L$ is \emph{locally Lagrangian in $W$}.
If $L$ is locally Lagrangian in $W$ and $J$ is an almost complex structure along $L$ in $T^{\ast}L$ compatible with $\omega$ along $L$ then $J$ induces a complex structure on $TW|_L$ for which $TL$ is a totally real subbundle. We call such a complex structure $J$ of $TW$ along $L$ \emph{locally compatible}. It follows from the discussion above that locally compatible complex structures admit extensions as almost complex structure to a neighborhood of $L$ in $W$ and that any two such extensions are homotopic on some common neighborhood of definition. 

Consider a pair $(W,L)$ where $W$ is a $6$-manifold and $L$ a local Lagrangian in $W$ and a locally compatible complex structure $J$.
Let $\Sigma$ be a compact Riemann surface with boundary $\partial\Sigma$. A \emph{generic worldsheet} is an embedding $u\colon (\Sigma,\partial\Sigma)\to(X,L)$ with the following properties:
\begin{itemize}	
	\item[$(i)$] $u(\Sigma\setminus\partial\Sigma)\cap L=\emptyset$,
	\item[$(ii)$] The restriction $du\colon T_{\partial\Sigma}\Sigma\to T_{L}X$ is $J$-complex linear, $(du+J\circ du\circ i)|_{\partial\Sigma}=0$.  
\end{itemize}
We say that two generic worldsheets are \emph{worldsheet isotopic} if there is a continuous $1$-parameter family of generic worldsheets connecting them.

Consider a generic surface $S$ in $(X,L)\setminus B$, where $B$ is a $6$-ball centered at a point in $L$. We think of $B$ as a $\C^{3}$-coordinate ball where $\R^{3}$ corresponds to $L$ and the point is the origin. We define two skein relations below, \eqref{eq: ws skein hyp} and \eqref{eq: ws skein ell},  both of which have the form $S_+ - S_- = S_0$,  
where $S_\pm$ and $S_0$ are obtained from $S$ by gluing in local surfaces in $B$. In the $\C^{3}$-coordinates of $B$ we view $S$ as giving linear asymptotic conditions given by complex linear spaces or half spaces that our local models match. It is clear how to complete this data to smooth surfaces and that the result is unique up to small isotopy.  For simpler formulas we parameterize the half disk in $\C$ by $\R\times [0,\pi]$ and the disk by $\R\times S^1$, with $-\infty$ corresponding to $0$, $s+it\mapsto e^{s+it}$.  
\begin{dfn}
	The \emph{worldsheet skein relations} are the following, see also Figure \ref{fig:ws-skein-rel}
	\begin{itemize}
		\item\emph{Hyperbolic.} 
		Assume that the asymptotic conditions of $S$ are the positive coordinate half-planes in the first two components of the $\C^{3}$-coordinates.
		Let $u_{\pm}^{k}\colon \R\times[0,\pi]\to\C^{3}$, $k=1,2$ be 
		\[ 
		u_{\pm}^{1}(s,t)=(e^{s+it},0,\pm 1),\qquad u_{\pm}^{2}(s,t)=(0,e^{s+it},0),
		\]
		and let $S_\pm^{\mathrm{hyp}}$ be the worldsheet obtained by inserting $u^{1}_\pm$ and $u^{2}_{\pm}$ in $S$. 
		Let $u_{0}\colon \R\times[0,\pi]\to\C^{3}$
		\[
		u_{0}(s,t)=(e^{s+it},e^{-(s+it)},0)
		\]
		and let $S_0^{\mathrm{hyp}}$ be the worldsheet obtained by inserting $u_0$ in $S$. Then 
		\begin{equation}\label{eq: ws skein hyp}
			S_+^{\mathrm{hyp}} \ - \ S_-^{\mathrm{hyp}}  \ = \ S_0^{\mathrm{hyp}}
		\end{equation}
		
		\item\emph{Elliptic.}
		Assume that the asymptotic conditions of $S$ is the complex line containing the vector $(1,-i,0)$. 
		Let $v_{\pm}\colon \R\times S^{1}\to\C^{3}$ be 
		\[ 
		v_{\pm}(s,t)=(e^{s+it},-ie^{s+it},\pm i),
		\]
		and let $S_\pm^{\mathrm{ell}}$ be the worldsheet obtained by inserting $v_\pm$ in $S$. 
		Let $v_{0}\colon [0,\infty)\times S^1 \to\C^{3}$
		\[
		v_{0}(s,t)=(e^{s+it}+e^{-(s+it)},i(e^{s+it}-e^{-(s+it)}),0)
		\]
		and let $S_0^{\mathrm{ell}}$ be the worldsheet obtained by inserting $v_0$ in $S$. 
		Then 
		\begin{equation}\label{eq: ws skein ell}
			S_+^{\mathrm{ell}} \ - \ S_-^{\mathrm{ell}} \ =  \ S_0^{\mathrm{ell}}
		\end{equation}
	\end{itemize}
\end{dfn}

\begin{figure}[h!]
\begin{center}
\includegraphics[width=0.7\textwidth]{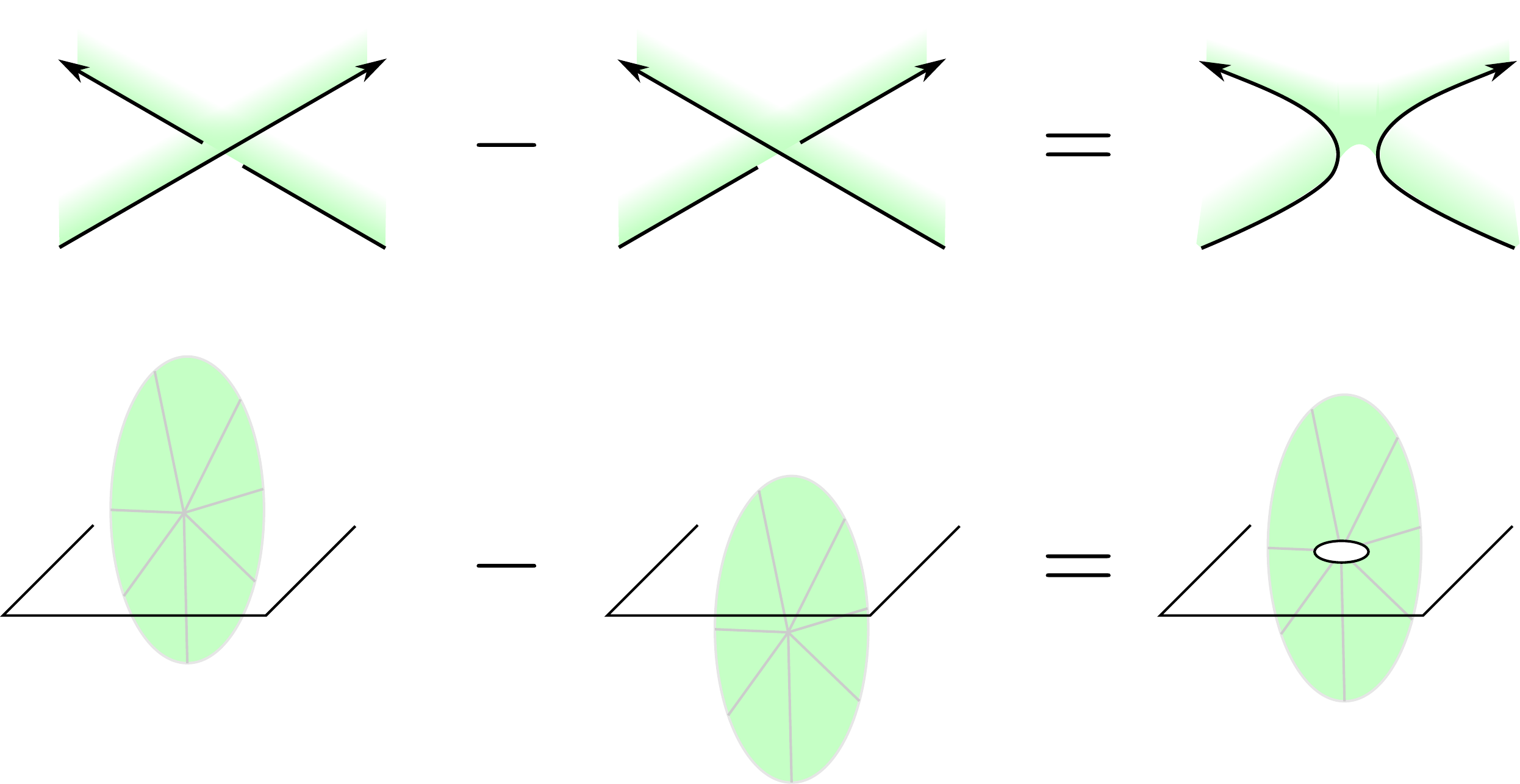}
\caption{The worldsheet skein relations. Top: the hyperbolic skein relation, the black segments are part of the boundaries of the generic worldsheet surfaces on the Lagrangian and the green regions indicates the interior of the surfaces near the boundary. Bottom: the elliptic skein relation, the black rectangular shape indicates the Lagrangian, the green represents interiors of the worldsheet surfaces and the black circle in the right hand side is the newborn boundary component that appears when the interior of the surface crosses the Lagrangian.}
\label{fig:ws-skein-rel}
\end{center}
\end{figure}

\begin{dfn}
	The \emph{worldsheet skein module} $\Sk(X,L)$ is the $\IQ$-vector space generated by worldsheet isotopy classes modulo the skein relations \eqref{eq: ws skein hyp} and \eqref{eq: ws skein ell}.
\end{dfn}

Often in our applications the choice of complex structure $J$ is clear from the context and we will omit it from the notation. Note that $\Sk(W,L)$ is naturally a module over $\Sk(S^{6},S^3)$ where the multiplication is induced by taking connected sum (of ambient spaces). 

If $(W,L)$ and $(W',L')$ are $6$-manifolds with local Lagrangians and $\iota\colon (W,L) \to (W',L')$ is an embedding that preserves local complex structures and the orientations of $L$ and $L'$, then $\iota$ induces a map $\Sk(W,L) \to \Sk(W',L')$. This will be particularly interesting to us in the situation when $(W,L)$, outside of a compact subset, is a disjoint union of the form 
\be
\bigl((-\infty,0] \times M_0 \,,\, (-\infty,0] \times \Lambda_0 \bigr) \ \sqcup \ 
\bigl([0,\infty) \times M_1 \,,\, [0,\infty) \times \Lambda_1\bigr)
\ee
and $(W',L')$ outside of a compact subset is a union of the form 
\be
\bigl((-\infty,0] \times M_1 \,,\, (-\infty,0] \times \Lambda_1 \bigr) \ \sqcup \ 
\bigl([0,\infty) \times M_2 \,,\, [0,\infty) \times \Lambda_2\bigr).
\ee
Here $M_{k}$ are $5$-manifolds, $\Lambda_{k}$ are oriented 2-submanifolds and the local complex structures $J_{k}$ are independent of the $\IR$-coordinate outside of a  compact subset of $\IR$. Then there is a $\Sk(S^6,S^3)$-linear map 
\be\label{eq:universalSkeinProductCobordism}
\Sk(W,L) \times \Sk(W',L') \to \Sk(W'',L'')
\ee
where $(W'',L'')$ is obtained from $(W,L)$ and $(W',L')$ by gluing the matching cylindrical regions. 

In this case we allow generators of $\Sk(W,L)$ to also be non-compact with chord and orbit asymptotics as follows.
Fix finite collections $c_{j}$ of smoothly embedded paths in $M_{k}$ with endpoints on $\Lambda_{k}$, derivative at the endpoint equal to $\pm J_{k}\partial_{s}$, where $\partial_{s}$ denotes the tangent vector of the $\IR$-factor, and otherwise disjoint from $\Lambda_{k}$. We call these arcs with extra data \emph{chords}. 

We fix similarly a finite collection of embedded loops $\gamma_{k}$ in $M_{k}\setminus\Lambda_{k}$ together with an overall decay rate $\epsilon_{k}$. We call the $d$-fold iterate of these loops, with the given decay rate, \emph{orbits}, and denote them by  $\gamma^{(d)}_{k}$ , $d\ge 1$.

Then the Riemann surface $\Sigma$ is allowed to have punctures at interior and boundary punctures where the map $u$ is asymptotic, respectively, to the cylinder over orbits or to the strip over a chord, see Figures \ref{fig:stretching} and \ref{fig:stretching-open}. Details are as follows. 

Consider first the orbit case. We require that the surface is exponentially asymptotic to the orbit $\gamma_{k}^{(d)}$ at some rate $\epsilon$. The asymptotics of the surface near an orbit $\gamma_{k}$ then determines an embedding of a number of circles into the $5$-dimensional solid torus $\gamma_{j}\times D^{4}$. Such a configuration is uniquely determined up to isotopy by the partition of the total homotopy class corresponding to the multiplicities of the orbit.

The case of chords is more complicated since multiple punctures may map to the same chord. However, for the purposes of this paper it is sufficient to consider only boundary punctures that map to distinct $c_{j}$ and we will restrict to that case. (We leave more detailed investigations of chords with multiple boundary punctures to future work.)

The product \eqref{eq:universalSkeinProductCobordism} extends with the convention that it is non-zero only if the positive and negative asymptotic Reeb chords and orbits match so that the two maps can be glued together. At orbits the gluing is canonical: cylinders joining the asymptotic links are unique up to isotopy. At chords the distinct asymptotics then determine a unique link. The interior of the surface is unique up to isotopy as in the orbit case.

\subsubsection{Universal counts of holomorphic curves}\label{ssec:univcurvecounts}
Consider now the case when $(W,\omega)$ is a symplectic Calabi-Yau threefold with a $\omega$-tame almost complex structure $J$, and $L$ is a Maslov zero Lagrangian equipped with a (relative) spin structure so that $(W,L)$ is tame at infinity. In \cite{Ekholm:2019yqp}, \eqref{eq: ws skein hyp} and \eqref{eq: ws skein ell} were identified with wall-crossings for holomorphic curves with boundary in $L$. Consequently, if we define the world-sheet partition function of $L$ as 
\be
\sfZ^{\textnormal{ws}}_{L} \ = \ \sum_{u\in\mathcal{M}} w(u) \langle u\rangle \ \in \ \widehat{\Sk}(W,L),
\ee 
where $\mathcal{M}$ denotes the moduli space of all disconnected bare holomorphic curves, see \cite{Ekholm:2024bare}, $w(u)$ is the rational (orbifold) weight of $u$, $\langle u \rangle$ denotes the element represented by $u$ in $\Sk(W,L)$, and $\widehat{\Sk}(W,L)$ is the action completion of $\Sk(W,L)$, then $\sfZ^{\textnormal{ws}}_{L}$ is invariant under deformations.

Here the \emph{action completion} is the following. For relative homology classes of non-positive symplectic area there are no curves and the invariant is represented by the empty curve. For any fixed positive symplectic area there is a maximal possible Euler characteristic of holomorphic curves with that given area and by compactness the number of bare curves below any given Euler characteristic is finite. We complete in this way in Euler characteristic for fixed action level and then in action level.  

If $(W,L)$ is cylindrical at infinity in the SFT sense (see \cite{EGH}), then for any collections $\bc^\pm$ of Reeb chords of the Legendrians and any collection of closed Reeb orbits $\boldsymbol{\gamma}^\pm$ of the contact manifold at $\pm \infty$ so that $|\bc^+| + |\boldsymbol{\gamma}^+| = |\bc^-| + |\boldsymbol{\gamma}^-|$ (or $|\bc^+| + |\boldsymbol{\gamma}^+| = |\bc^-| + |\boldsymbol{\gamma}^-| + 1$ if $(W,L)$ is cylindrical), we can analogously define
\be
\sfZ^{\textnormal{ws}}_{L}(\bc^+,\bc^-) \ = \ \sum_{u\in  \CM(\bc^+,\boldsymbol{\gamma}^+,\bc^-,\boldsymbol{\gamma}^-)} w(u)\, \langle u \rangle \qquad \in \quad \widehat{\Sk}(W,L)\,,
\ee 
where now $\CM(\bc^+,\boldsymbol{\gamma}^+,\bc^-,\boldsymbol{\gamma}^-)$ denotes the moduli space of all disconnected bare punctured holomorphic curves with positive asymptotics at $\bc^+$ and $\boldsymbol{\gamma}^+$ and negative asymptotics at $\bc^-$ and $\boldsymbol{\gamma}^-$ modulo translation in the $\IR$-coordinate if $(W,L)$ is cylindrical. Note that the asymptotic expansions of the curves near boundary punctures give the required data near Reeb chord endpoints.

If $(W,L)$ splits into $(W_+,L_+)$ and $(W_-,L_-)$ when SFT-stretching \cite{EGH,Ekholm:2019yqp} along a contact hypersurface $(M,\Lambda) \subseteq (W,L)$, then we have that 
\be
\sfZ^{\textnormal{ws}}_L(\bc^+,\boldsymbol{\gamma}^+,\bc^-,\boldsymbol{\gamma}^-) \ = \ \sum_{\bc,\boldsymbol{\gamma}} \sfZ^{\textnormal{ws}}_{L_+}(\bc^+,\boldsymbol{\gamma}^+,\bc,\boldsymbol{\gamma}) \cdot \sfZ^{\textnormal{ws}}_{L_-}(\bc,\boldsymbol{\gamma},\bc^-,\boldsymbol{\gamma}^-)\,,
\ee
where the sum is taken over all collections of Reeb chords $\bc$ of $\Lambda$ and closed Reeb orbits $\boldsymbol{\gamma}$ of $M$ that satisfy $|\bc|+|\boldsymbol{\gamma}|=|\bc^+|+|\boldsymbol{\gamma}^+|=|\bc^-|+|\boldsymbol{\gamma}^-|$. 

As explained above, gluing at orbits gives topologically unique results by general position. For gluing at Reeb chords one must take into account the effect of extending perturbations that move Reeb chords apart. Since we use gluing only for curves with distinct asymptotics at boundary punctures, we leave such questions to future work.

\subsubsection{Curve counts in the HOMFLYPT skein}
The framed skein module $\Sk(L)$ of an oriented 3-manifold $L$ is the module over the ring $\IQ[a_L^{\pm1}, z^{\pm1}]|_{z = q - q^{-1}}$ generated by isotopy classes of framed links in $L$ modulo the three HOMFLYPT skein relations: 
\be\label{eq:skein-rel-intro}
\begin{split}
	& \OC - \UC=(q-q^{-1})\ \SP \\[2mm]
	& \qquad  \UK = \frac{a_L-a_L^{-1}}{q-q^{-1}}\\[2mm]
	& \PT=a_L \ \ST\ ,\qquad \NT =a_L^{-1}\ \ST. 
\end{split}
\ee
In the case that $L$ is not connected and has components $L = L_1 \sqcup \dots \sqcup L_r$, then we extend the coefficient ring to $\IQ[a_{L_1}^{\pm1},...,a_{L_r}^{\pm1}, z^{\pm1}]$, and the second and third skein relation hold for $a_{L_i}$ when applied in $L_i$.

For a finite collection $\gamma_i\colon S^1 \to L, i =1,...,r$, of mutually disjoint, embedded, oriented curves, we may modify this definition to obtain the skein module $\Sk(L;\{\gamma_i\})$ over the ring $\IQ[a_{L}^{\pm1},a_{1}^{\pm1},...,a_{r}^{\pm1}, z^{\pm1}]$ where links are assumed to be disjoint from the $\gamma_i$ and in addition to \eqref{eq:skein-rel-intro} we have the following skein relation for part of a link crossing $\gamma_i$: 
\be\label{eq:gamma-skeinRelation}
\picLabel{{$\gamma_i$}{100}{-10}} {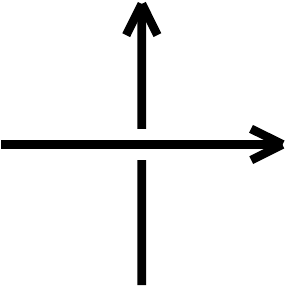} {.250} = a_i \picLabel{{$\gamma_i$}{100}{-10}} {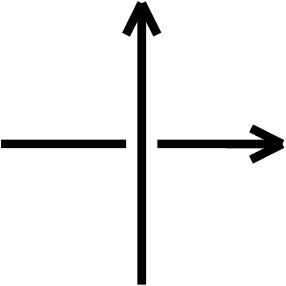} {.250} \,.
\vspace{10pt}
\ee
In our applications, $L$ may be disconnected and to each connected component $L_i$ of $L$, we associate a collection $\{\gamma^i_j\}$ of such curves for which the relation \eqref{eq:gamma-skeinRelation} holds with $a_j = a_{L_i}$. The operation of forming the connected sum with $S^3$ induces on $\Sk(L;\{\gamma_i\})$ the structure of a $\Sk(S^3)$-module. In Section \ref{sec:skeinsolidtorus} below we recall basic algebraic properties of the skein modules of $\IR \times T^2$ and $S^1 \times \IR^2$.

Given a pair $(W,L)$ as in Section \ref{ssec:surfaceskein} where $L$ decomposes into connected components $L=L_{1}\cup\dots\cup L_{r}$, one can obtain non-trivial homomorphisms from $\Sk(W,L)$ to a skein module of $L$ as follows. Choose a nowhere-vanishing vector field $v$ on $L$ and an oriented 4-chain $C = C_1 \cup ... \cup C_r$ in $W$ so that $\partial C_i = 2 L_i$ and near its boundary, $C_i$ is given by $\{\exp_x(\pm t J v)|x \in L_i, t \in [0,\varepsilon) \}$ with respect to some Riemannian metric on $W$. We further assume that the interior $\textnormal{int}\,C$ of $C$ intersects $L$ transversely along a collection of mutually disjoint embedded circles which inherit an orientation from the orientations of $L$ and $C_i$. For any surface $u\colon\Sigma \to W$ representing an element in $\Sk(W,L)$ so that $v$ is nowhere tangent to $u(\partial \Sigma)$ and $u(\partial \Sigma)$ is disjoint from $L \cap\, \textnormal{int}\, C$,
\be
\rho_{L,C}(u) \ \coloneqq \ \prod_{i=1}^r a_{L_i}^{u \cdot C_i} z^{-\chi(\Sigma)} \langle \partial u \rangle
\ee
defines an element in $\Sk(L;\{L \cap\, \textnormal{int}\, C_i\})$, where $u \cdot C_i$ denotes the algebraic count of intersections of $C_i$ with the interior of $u(\Sigma)$, $\chi(\Sigma)$ denotes the Euler characteristic of $\Sigma$, and $\langle \partial u \rangle$ denotes the element in $\Sk(L;\{L \cap\, \textnormal{int}\, C_i\})$ represented by $\partial u$ with the framing given by $v|_{\partial u}$, and the $a$-variables corresponding to $L \cap \textnormal{int}\,C_i$ have been set equal to $a_{L_i}$. The skein relations \eqref{eq:skein-rel-intro} and \eqref{eq:gamma-skeinRelation} ensure that $\rho_{L,C}(u)$ only depends on the class of $u$ in $\Sk(W,L)$, so it descends to a well-defined $\IQ$-linear homomorphism
\be
\rho_{L,C}: \Sk(W,L) \to \Sk(L;\{L \cap\, \textnormal{int}\, C_i\}).
\ee
This map is in fact compatible with the $\Sk(S^6,S^3)$ and $\Sk(S^3)$-module structures on the different skeins. In what follows, we will usually omit the loops $L \cap\, \textnormal{int}\, C$ from the notation.

By applying $\rho_{L,C}$ to the world-sheet partition function $\sfZ^{\textnormal{ws}}_L$, we obtain the partition function 
\be
\sfZ_{L} \ = \ \sum_{u\in\mathcal{M}} w(u) \,\prod_{j=1}^{r} a^{u \cdot C}_{L_{j}} \,z^{-\chi(u)}\, \langle \partial u\rangle \ \in \ \widehat{\Sk}(L),
\ee 
which takes values in the action completion of the HOMFLYPT skein module of $L$. We will usually omit $C$ from the notation and simply write $\rho_L$ instead of $\rho_{L,C}$.

Consider now the Lagrangian conormal $L_{K}\subset T^{\ast}S^{3}$ of a framed $m$ component link $K=K_{1}\cup\dots\cup K_{m} \subset S^{3}$. Shifting $L_{K}$ off $S^{3}$ along a one form (defined in a tubular neighborhood) dual to the tangent vector of $K$, it is easy to see that there exists exactly one annulus stretching from each conormal to the zero section. The total contribution in the skein of such annuli is, see \cite{Ekholm:2021colored} or Section \ref{sec:skein-rec-annulus} 
\be
\sum_{(\lambda_{1},\dots,\lambda_{m})} \langle K_{1;{\lambda_{1}}}\cup\dots\cup K_{m;{\lambda_{m}}}\rangle\, \sfW_{\lambda_{1}}\otimes\dots\otimes \sfW_{\lambda_{m}} \ \in \ \widehat{\Sk}(L),
\ee
where $K_{j,{\lambda_{j}}}$  denotes the skein element $\sfW_{\lambda_{j}}=\sfW_{\lambda_{j},\emptyset}$ 
in the solid torus inserted in a small tubular neighborhood around $K_{j}$ and $\langle K\rangle$ denotes the HOMFLYPT polynomial in $S^{3}$. 

By invariance of the curve count we can now transition to the resolved conifold by SFT-stretching \cite{EGH, Ekholm:2019yqp}. Stretching the complex structure around $S^{3}$ in a neighborhood that is smaller than the shift off of the conormal we find that holomorphic curves fall apart into holomorphic buildings joined at Reeb orbits of $S^{3}$. However, if $S^{3}$ has the round metric then the index of any closed Reeb orbit is $\ge 2$ and curves with negative asymptotics at such Reeb orbits would have negative dimension. Consequently, under stretching boundaries shrink and all curves lift off of the zero section. For a suitable almost complex structure the count in $T^{\ast}S^{3}\setminus S^{3}$ is then the same as in the resolved conifold $X$, where $\log a^{2}$ which in $T^{\ast} S^{3}$ corresponds to the linking dual of the zero section, i.e., the cotangent fiber sphere, becomes the homology class or complexified area of $\C P^{1}$ in $X$, giving a mathematical derivation of the large $N$ transition in \cite{Ooguri:1999bv}.

\section{Quantum groundstates of strings and curves at infinity}\label{sec:recursion-worldsheet}
In this section we consider holomorphic curve counting, or $A$-model open topological strings, in a non-compact Calabi-Yau $W$ with Lagrangian boundary condition or branes on a non-compact Lagrangian $L\subset W$. We will assume that $(W,L)$ is asymptotic to a certain $\R$-invariant symplectic pair at infinity and describe how infinite area holomorphic curves, or infinite-area worldsheet instantons, in the region at infinity relate to finite area curves, corresponding to regular worldsheet instantons, in the bulk.   

Our main examples of such geometry at infinity are Legendrian conormals $\Lambda_{K}\subset U^{\ast} S^{3}$ of links $K\subset S^{3}$, where relevant infinite area curves are directly computable by finite dimensional methods (Morse flow graphs), see \cite{Ekholmflowtrees}.

\subsection{Holomorphic curves and asymptotic contact boundaries}\label{ssec: strings at infinity}
We consider topological string theory in a symplectic Calabi-Yau $W$ with a Maslov index zero Lagrangian $L$.  
We will consider the case when $(W,L)$ is non-compact and asymptotically cylindrical over $(Y,\Lambda)$, a contact manifold $Y$ with Legendrian submanifold $\Lambda$. We begin with a review of relevant notions from contact geometry.

A contact structure $\xi$ on a $(2n-1)$-dimensional manifold $Y$ is a completely non-integrable hyperplane distribution. A contact 1-form of $\xi$ is a $1$-form $\alpha$ on $Y$ such that $\xi = {\rm ker}(\alpha)\subset TY$. The complete non-integrability condition means that $\alpha\wedge (d\alpha)^{\wedge n}$ is everywhere non-zero and $d\alpha$ gives a symplectic form on $\xi$. 
We point out that if $\alpha$ is a contact $1$-form of $\xi$ then so is $e^{f}\cdot\alpha$ for any function $f\colon Y\to\R$. To a contact manifold $Y$ one associates a symplectic manifold $\R\times Y$, the symplectization of $Y$, with symplectic structure $\omega$ which is $\R$-invariant: if $s$ is a coordinate along $\IR$, then $\omega = d(e^s\alpha)$, where $\alpha$ is a contact $1$-form on $Y$.

A Legendrian submanifold $\Lambda\subset Y$ of a contact $(2n-1)$-manifold is an $(n-1)$-dimensional submanifold with tangent space everywhere contained in $\xi$,  i.e., $T\Lambda\subset \xi$. By non-integrability of $\xi$, $(n-1)$ is the maximal dimension of a submanifold with this property. A Legendrian $\Lambda\subset Y$ determines an $\IR$-invariant Lagrangian submanifold $\IR\times \Lambda\subset \R\times Y$.

Associated to a contact $1$-form $\alpha$ is its Reeb vector field $R_{\alpha}$ which is the unique vector field along $Y$ determined by the conditions,
\[
\iota_{R_{\alpha}}d\alpha=0, \qquad \iota_{R_{\alpha}}\alpha=1,
\]
where $\iota$ denotes contraction. Closed orbits of $R_{\alpha}$ are called \emph{Reeb orbits} and flow lines that begin and end on $\Lambda$ \emph{Reeb chords}.

With these notions introduced we explain the notion of $(W,L)$ being asymptotically cylindrical.
At infinity, $W$ is topologically $[0,\infty)\times Y$, where $Y$ is a contact $5$-manifold with contact $1$-form $\alpha$ and its symplectic form is $\omega=d(e^{s}\alpha) + \omega_{0}$, where $s$ is a coordinate on $[0,\infty)$  and $\omega_{0}$ is uniformly bounded. Rescaling by $e^{-s}$ we find that the symplectic geometry of $W$ is arbitrarily close to that of the symplectization $\R\times Y$ outside any sufficiently large compact subset with complement $(T,\infty) \times Y$ for $T \gg 0$. The requirement that $L$ is asymptotic to the Legendrian $\Lambda$ means similarly that, near infinity, in the region where $W$ is asymptotic to $\R\times Y$, $L$ is asymptotic to $\IR\times \Lambda$, where $\Lambda\subset Y$ is a Legendrian submanifold.

There is a natural $\R$-invariant class of almost complex structures $J$ on $\R\times Y$, that restricts to complex structures in the contact hyperplanes and have the following property in remaining directions, 
\be\label{eq:J-symp}
J(\partial_s)=R_{\alpha},\quad J(R_{\alpha}) = -\partial_s\,,
\ee
where $\partial_{s}$ is the standard vector field along $\R$. 

Assume then that $(W,L)$ is asymptotic to $(Y,\Lambda)$ and consider open topological strings on $(W,L)$.
Topological strings quantize the symplectic structure of $W$. We first consider the more standard finite area part of the theory. Here, worldsheet instantons correspond to holomorphic maps $u\colon(\Sigma,\partial \Sigma)\to (W,L)$, where $\Sigma$ is a Riemann surface with boundary~$\partial\Sigma$. The weight of an instanton in the topological string path integral is given by its exponentiated symplectic area, $\exp(-\int_\Sigma u^*\omega)$. Locally in $W$ the symplectic form can be expressed in terms of a Liouville 1-form $\omega = d\lambda$, which means that the symplectic area can be expressed as a boundary line integral
\be\label{eq:string-area}
\int_\Sigma u^*\omega = \oint_{\partial \Sigma} u^*\lambda + \langle t,\beta \rangle
\ee
up to quantized shifts labeled by fluxes $t\in H^{2}(W)$ and homology classes $\beta\in H_2(W;\IZ)$, where $\langle\,,\rangle$ denotes the pairing between cohomology and homology.

\subsection{Holomorphic curves with infinite area and line operators}
Let $(W,L)$ be as in Section \ref{ssec: strings at infinity}. In addition to the finite area worldsheet instantons, we consider also holomorphic curves with infinite area. More precisely, these are curves in $W$ with boundary on $L$ that have finite Hofer energy, see~\cite{EGH}, which implies that the curves have punctures where they are asymptotic to Reeb chords and orbits. The space of such curves admits a compactification by several level curves with additional levels that are curves in the symplectization $\IR\times Y$ that are holomorphic with respect to an $\IR$-invariant complex structure as in \eqref{eq:J-symp}, with finite $(d\alpha)$-area, and again with punctures at Reeb chords and orbits, at both positive and negative infinity. More precisely, the asymptotic condition means that, at infinity, holomorphic string worldsheets are asymptotically close to $\IR\times \gamma$ where $\gamma$ is an integral curve of $R_{\alpha}$ (such products are themselves holomorphic by the definition of $J$ in \eqref{eq:J-symp}). In particular, near interior and boundary punctures, the $Y$-component of such maps limit to Reeb orbits and chords, respectively, see Figures~\ref{fig:stretching} and~\ref{fig:stretching-open}. 

\begin{figure}[h!]
	\begin{center}
		\includegraphics[width=0.70\textwidth]{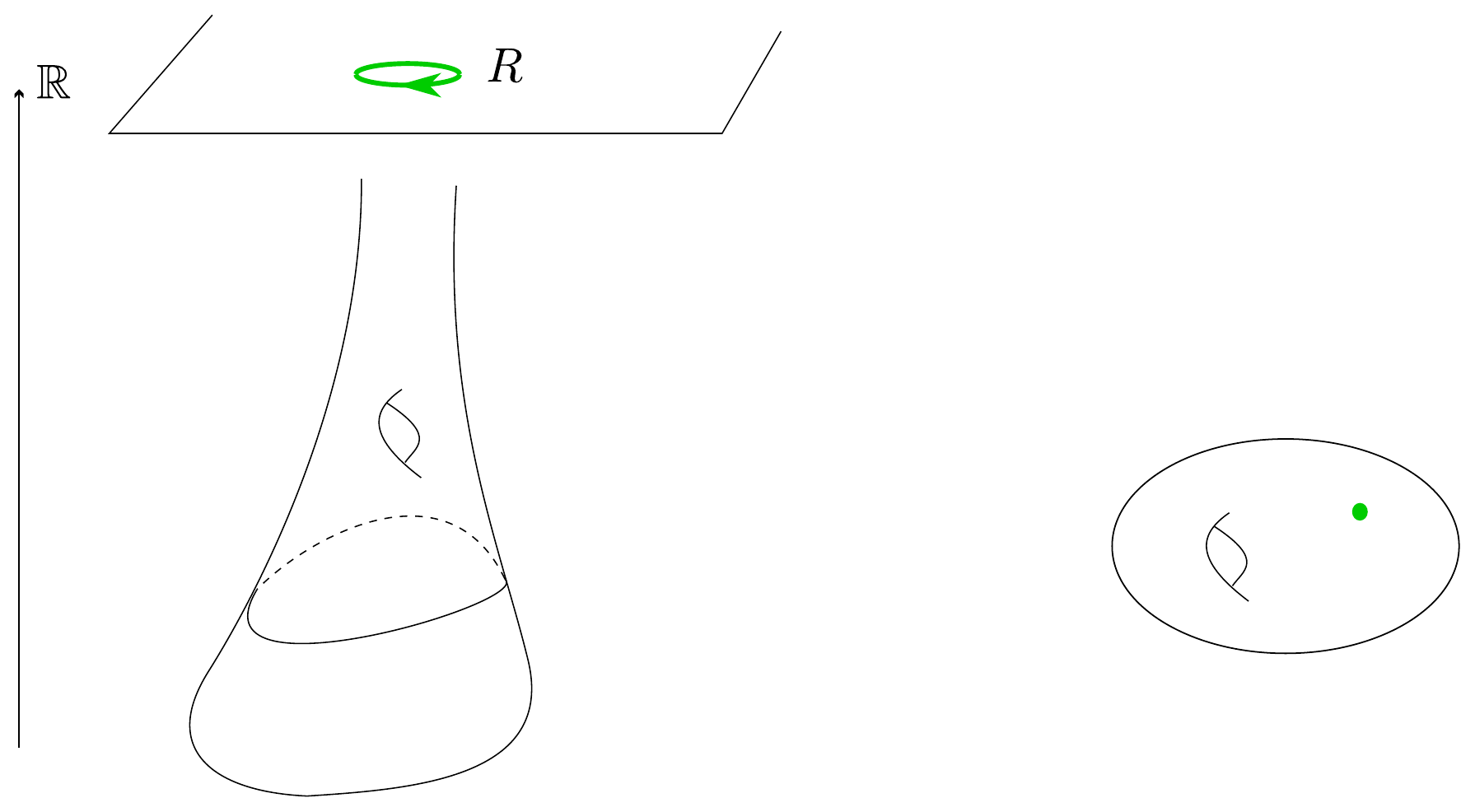}
		\caption{Left: a holomorphic curve stretching away to infinity that asymptotes to a Reeb orbit. Right: the source Riemann surface where the asymptotic Reeb flow loop corresponds to a puncture. }
		\label{fig:stretching}
	\end{center}
\end{figure}

\begin{figure}[h!]
	\begin{center}
		\includegraphics[width=0.70\textwidth]{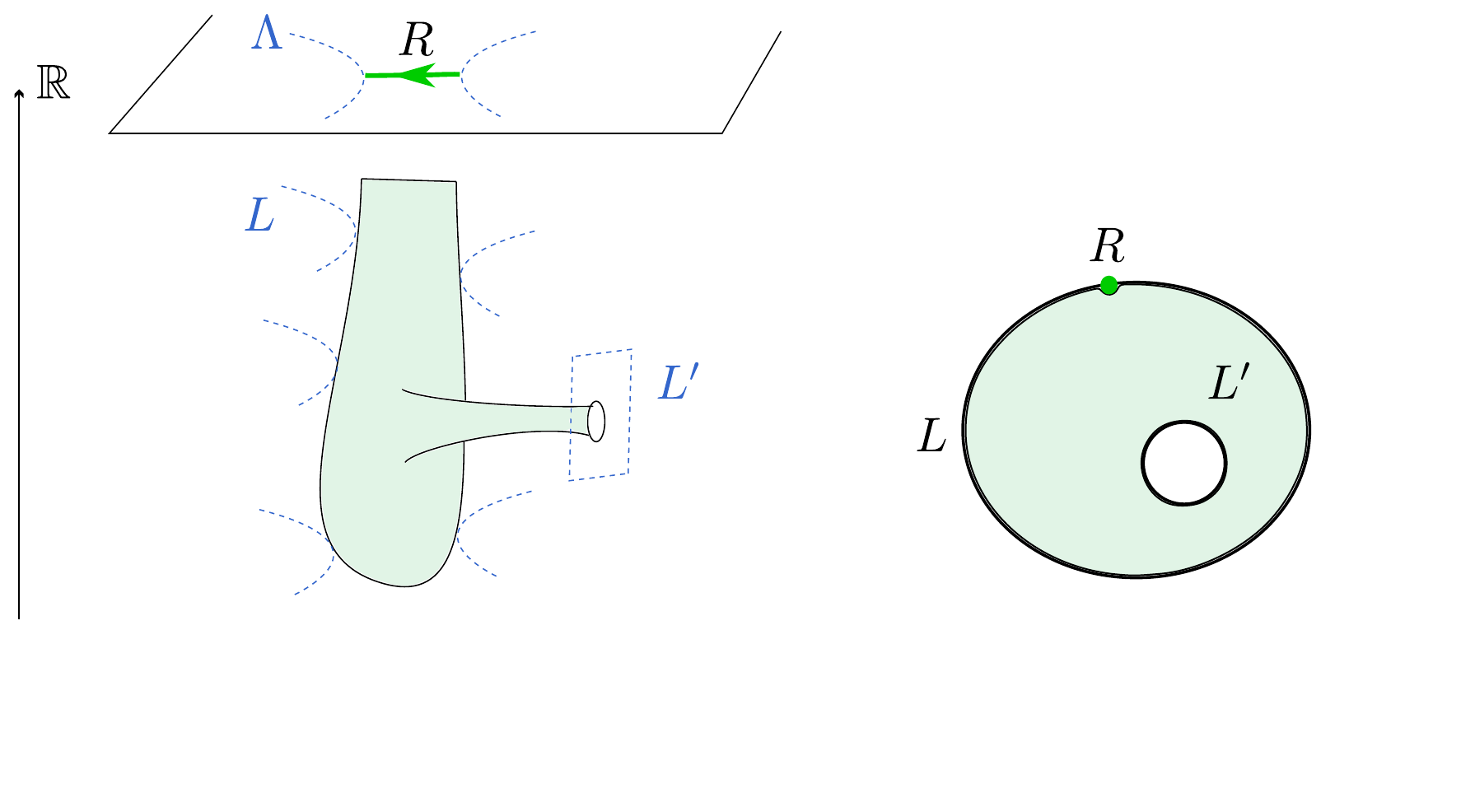}
		\caption{Open holomorphic curves with boundary punctures on the worldsheet corresponding to a Reeb chord $c$ connecting $c^{-}\in\Lambda$ to $c^{+}\in\Lambda$. }
		\label{fig:stretching-open}
	\end{center}
\end{figure}

Unlike finite area curves, which contribute instanton corrections to the vacuum partition function, infinite area curves are non-dynamical and do not contribute to the vacuum sector. However, such curves may contribute to other sectors of the theory, in particular to the expectation value of certain extended operators.

In fact, the $\R$-shift in the cylindrical region $[0,\infty)\times Y$ allows us to regularize infinities of curves with Reeb asymptotics: 
if $C$ is a holomorphic curve in $\R\times Y$ then the ratio of the area of $C$ and an $s$-shift of $C$ equals $e^{s}$ and the $(d\alpha)$-area of $C$ is invariant. 
Moreover the Liouville 1-form is $e^s \alpha$, and since the Reeb vector field obeys $\iota_{R_{\alpha}} \alpha = 1$, the asymptotic Reeb flow line of a holomorphic curve is calibrated by the contact 1-form $\alpha$, and therefore corresponds to critical points of the action 
\be
\gamma \mapsto \int_\gamma \alpha.
\ee
We then regularize the infinite area curves by subtracting half-infinite trivial strips corresponding to its asymptotics and obtain finite operators for finite $d\alpha$-area instantons. This regularization is not unique but is possible to organize in an invariant way, see Section \ref{ssec:groundstatesatinfinity}.

\begin{rmk}[Analogy with framed BPS states] \label{rmk:framed-BPS}
	Recall that branes in topological strings on $(W,L)$ admit natural lifts to M5 branes in M-theory on $W\times S^1\times \IR^4$, wrapped on $L\times S^1\times \IR^2$, and that in this lift, worldsheet instantons (of finite area) organize into contributions from M2 branes wrapped on holomorphic cycles $C$ in $(W,L)$ \cite{Gopakumar:1998ii, Gopakumar:1998jq, Ooguri:1999bv, Labastida:2000zp, Labastida:2000yw}.
	
	On the other hand, M2 branes with infinite area give rise to (non-dynamical) surface operators in the worldvolume theory on the M5 branes. 
	When the surface operator wraps $\gamma\times S^{1}$, where $\gamma$ is a curve $\gamma$ in $L$ and $S^{1}$ the M-theory circle, it descends to a pair of dual line operators: one in the 3d topological gauge theory on~$L$, and one in the dual 3d~$\CN=2$ QFT on~$S^1\times \IR^2$, see~\cite{Dimofte:2011ju}.
	In our setting infinite area curves with finite $d\alpha$-area define line operators that act on the (skein valued) topological partition function. 
	
	There is a close analogy between the infinite area holomorphic curves considered here and framed BPS states \cite{Gaiotto:2010be}, generally understood as infinitely massive states that measure the response of a QFT to the insertion of such BPS line operators. 
	In fact, wall-crossing of framed BPS states can be used to compute the spectrum of finite mass BPS states. In a similar way, the infinite area instantons considered here give recursion relations that determine the contributions from the finite area worldsheet instantons.  
\end{rmk}

\subsection{Perturbative and exact groundstates of strings at infinity}\label{ssec:groundstatesatinfinity}
As already mentioned, Reeb chords and orbits of $\Lambda_{K}$ correspond to critical points of the string action functional near infinity. Therefore, in the quantum theory, Reeb chords and orbits generate the space of perturbative groundstates of the open string. We give a brief description of Reeb chords and orbits in cases relevant here.   

Let $K\subset S^{3}$ be an $m$-component link. For $\Lambda_{K}\subset U^{\ast}S^{3}$, if the unit cotangent bundle $U^{\ast}S^3$ with respect to some Riemannian metric is equipped with the Liouville one-form $\alpha=pdq$, then Reeb orbits are cotangent lifts of geodesic loops and Reeb chords are cotangent lifts of geodesics that connect the link to itself and are perpendicular to $K$ at its endpoints. Further, there is a (Maslov) grading on the space of perturbative groundstates, induced by a relative Chern class, which in the case under consideration agrees with the Morse index of the geodesic in $S^3$ underlying a given Reeb chord or orbit. If we use the round metric on $S^{3}$ then any closed geodesic has Morse index (and therefore Maslov degree) $\ge 2$ and binormal chords have Morse index and Maslov degree $\ge 0$.  
We will be mainly concerned with the degree zero groundstate which is generated by words of Reeb chords of degree zero, where we include also the empty word. 

Reeb chords are not invariant under deformations, and this can lead to jumps in the space of perturbative groundstates. The deformation invariance of the open topological string leads to the expectation that there exists a well-defined quantum ground state which is insensitive to deformations.
The invariant ground state is obtained in a familiar way: perturbative groundstates are corrected by instantons and the true quantum groundstate appears as the homology of a BRST-differential defined by the instantons, as in Witten's realization of Morse theory in quantum mechanics~\cite{Witten:1982im} and in line with the principles underlying the treatment of holomorphic curves in Symplectic Field Theory~\cite{EGH}.  

Mathematically, the complex of groundstates with BRST-differential is a higher genus extension of the Chekanov-Eliashberg dg-algebra, see~\cite{ekholm2013knot} for the dg-algebra in this setting and~\cite{Ekholm:2018iso} for first steps in a perturbative approach (in the string coupling $g_s$) to the generalization to higher genus. 
Here we will follow the approach to curve counting developed in~\cite{Ekholm:2019yqp}, where instead of working perturbatively in $g_s$, we count bare curves in the skein with genus parameter $z=q-q^{-1}=e^{\frac{g_{s}}{2}}-e^{-\frac{g_{s}}{2}}$. A more detailed description from the point of view of perturbative ground states is as follows. 

The Reeb chords on~$\Lambda$ represent open string states that have Chan-Paton degrees of freedom at their endpoints, and the boundary of a string world sheet gives rise to (non-dynamical) line operators. 
With this in mind we think of Reeb chords $c$ as asymptotic data given by their corresponding trivial strips $\R\times c$ with boundary line defects along $\R\times c^{\pm}$. We will modify the definition of the skein modules to allow also for open knots and links that agree with a finite collection of the lines $c^{\pm}\times\R$, for some Reeb chords $c$, outside some compact region in $[0,\infty)\times \Lambda_{K}$, compare Section \ref{sec:skein of torus with two boundary points}.

In order to cut off the otherwise infinite $d (e^{s}\alpha)$-area at the positive infinity of a collection $\bc$ of Reeb chord strips and Reeb orbit cylinders, we fix a choice of capping data~\cite{ekholm2013knot}. 
If the Legendrian is connected, then for any word of Reeb chords and for closed Reeb orbits this is done by choosing a surface with boundary on the Lagrangian $\Lambda \times \R$ whose interior is disjoint from $\Lambda \times \R$ and which has punctures which are negatively asymptotic to the given Reeb orbits and the word of Reeb chords. 
If instead the Legendrian has many connected components there may be Reeb chords with endpoints on distinct components. To define a capping surface for these we first connect the different components of the Lagrangian by adding a $0$-handle (a Legendrian $2$-sphere filled by a Lagrangian disk in the complement of $\Lambda\times\R$) and by connecting this to each component by a $1$-handle (Legendrian $S^{1}\times I$ filled by Lagrangian $D^2\times I$ in the complement). 
On this connected Lagrangian we can then pick capping data for all Reeb chords. We call such a surface a capping surface. 

If we then cut the trivial strips and cylinders off at some (large) $s>0$ and join to them the correspondingly cut-off capping surface, we have a finite area surface corresponding to $\bc$. 
We use such surfaces in combination with cut-off in $s$ to regularize the area of curves with Reeb asymptotics. 
From the viewpoint of the partition function, this regularization can be regarded as a way to compute finite contributions from `framed BPS states' in the background of a line operator whose charge (i.e., topological sector) is encoded by $\bc$, see Remark~\ref{rmk:framed-BPS}.
After fixing a choice of capping surfaces, we view the space of perturbative vacua as a module over the worldsheet skein $\Sk(\R\times U^{\ast}S^{3},\R\times \Lambda_{K})$ generated by words of the capped Reeb chords and orbits, corresponding to the (capped) boundaries of infinite-area holomorphic worldsheet instantons.
This definition depends on a choice of capping surface, however it is easy to keep track of how it would change if different choices are made.

We will be mainly concerned with the degree zero perturbative ground states generated by words of degree zero Reeb chords with endpoints on $\Lambda_{K}$. The grading is induced by the Maslov index discussed above, see \cite{ekholm2013knot} for more detail and Section \ref{sec:unknot} for concrete calculations.
The BRST-differential is given by counts of rigid (up to translation) curves in the symplectization $\IR\times U^{\ast}S^{3}$ with boundary on $\IR\times\Lambda_{K}$ in the skein. 
Adopting standard SFT notation (see~\cite{EGH}) we call the corresponding generating function the \emph{Hamiltonian}.  
The Hamiltonian is then the generating function of all bare rigid holomorphic curves with arbitrary sets of positive and negative punctures $\bc^\pm$:
\be\label{eq:generating-hamiltonian}
\sfH_{\Lambda_K} = \sum_{\scriptsize{\begin{matrix}\bc^+,\bc^-,\chi \\
			|\bc^{+}|=|\bc^{-}|+1\end{matrix}}}\;\sum_{u\in \CM(\bc^+,\bc^-,\chi)} w(u)\, \langle u\rangle \,\bc^+\partial_{\bc^-} \, \,,
\ee 
Here $\CM(\bc^+,\bc^-,\chi)$ denotes the moduli space of bare holomorphic curves with positive asymptotics according to $\bc^{+}$, negative asymptotics according to $\bc^{-}$, and of Euler characteristic $\chi$, $z=q-q^{-1}$, according to the conventions for bare curve counts, $w(u)$ is the weight (orbifold number of points) of the solution $u$, and $\langle u\rangle$ is the element represented by the map in the worldsheet skein $\Sk(\R\times U^{\ast}S^{3},\R\times \Lambda_{K})$. 

Note that capping surfaces of chords in $\bc^{+}$ together with a curve in $\CM(\bc^+,\bc^-,\chi)$ gives a capping surface for $\bc^{-}$, and thus with caps fixed the Hamiltonian gives the BRST-differential, thereby acting as an endomorphism on the space of perturbative vacua.
Note also that $|\bc^{+}|-|\bc^{-}|$ is the dimension of the moduli space and hence rigidity requires that $|\bc^+|=|\bc^-|+1$.  (The notation $\partial_{\bc^{-}}$ is to emphasize that the Hamiltonian acts in a natural way on certain functions with values in an extension of $\Sk(W,L)$.)

To connect the above to Chekanov-Eliashberg dg-algebras, we consider the semi-classical limit $q\to 1$ of the $U(1)$-version of this theory. For the $U(1)$-theory, we use in addition to the specialization $a=a_L=q$ the following splitting of the skein relations~\eqref{eq:skein-rel-intro}
\be\label{eq:U1-skein-rel}
\begin{split}
	q^{-1} \ \OC =  \SP =  q\ \UC\,,\qquad
	q^{-1} \ \PT= \ST = q\ \NT \,,\qquad 
	\UK = 1\,.
\end{split}
\ee
Thus in the $q\to 1$ limit, the skein relation degenerates to the relation in homology. An analogous fact is known to hold in the context of Abelian Chern-Simons theory \cite{Dunne:1989cz}, which corresponds to the worldvolume theory of a single brane on $L$.
The Euler characteristic contribution from the terms in the Hamiltonian (without capping surfaces) are $z^{-\chi+\#(\bc^{+})}$ and thus in the limit that the only contributions come from disks with exactly one positive puncture. 
Since these are the curves counted by the differential of the Chekanov-Eliashberg dg-algebra, the semi-classical approximation to the degree zero ground state is exactly the degree zero knot contact homology.  

Similarly, the $U(1)$-theory provides a rigorous version of the perturbative higher genus knot contact homology of~\cite{Ekholm:2018iso}. In fact the Hamiltonian in~\eqref{eq:generating-hamiltonian} contains more information than the corresponding expression there, which is recovered with the specialization $a_{\Lambda_{K_{j}}}=q$ and replacing the skein module on $\Lambda_{K}$ by the $U(1)$-skein that encodes homology and linking. (Note however, that the concrete calculations in~\cite{Ekholm:2018iso} are pictorial, which means they can be used also for the finer data needed in~\eqref{eq:generating-hamiltonian}.)

\subsection{Skein $D$-modules of link conormals}
Let $K\subset S^{3}$ be an $m$-component link with Legendrian conormal $\Lambda_{K}\subset U^{\ast}S^{3}$ and consider its worldsheet skein algebra $\Sk(\R\times U^{\ast}S^{3},\R\times\Lambda_{K})$. 

Recall that in $U^{\ast}S^{3}$ all Reeb orbits have degree $\ge 2$ and that all Reeb chords have degree $\ge 0$. 
It follows that the degree zero generators of the BRST-complex for $\Lambda_{K}$ are words of Reeb chords, all of degree zero and that the degree $0$ BRST-homology $H^{0}_{\mathrm{BRST}}(\Lambda_{K})$ is the cokernel of the differential acting on degree $1$ generators, which are words of Reeb chords with exactly one chord of degree $1$ and the remaining ones of degree $0$. 
This differential then counts holomorphic curves with positive asymptotics according to such a degree $1$ word and negative asymptotics according to a (possibly empty) degree zero word of Reeb chords. We have
\be
H_{\text{BRST}}^{0}(\Lambda_{K})= \{\text{degree $0$ words of capped chords}\} \ / \ \sfH_{\Lambda_{K}}\{\text{capped degree $1$ words}\}.
\ee

Let $L\subset X$ be a Lagrangian filling of $\Lambda_{K}$, for example $L=L_{K}$, the conormal Lagrangian of $K$. 
Consider the generalized partition function $\overline{\sfZ}_{L}$ that counts holomorphic curves in $X$ with boundary on $L$ and positive punctures at degree $0$ Reeb chords,
\be
\overline{\sfZ}_{L} \ = \ \sfZ_{L} + \sum_{k\ge 1} \sfZ^{(k)}_{L} \ \in \ \widehat{\Sk}(L)
\ee
where $\sfZ_{L}$ counts compact curves, i.e., curves without positive punctures, and $\sfZ^{(k)}_{L_{K}}$ counts curves of finite Hofer energy, see \cite{EGH}, with $k>0$ positive punctures. As in Section \ref{ssec:reviewskeinsonbranes} we count curves by the values of their boundaries in the skein.

The Hamiltonian $\sfH_{\Lambda_K}$ acts on $\overline{\sfZ}_{L}$ by gluing all negative puncture asymptotic Reeb chords (the $\partial_{\bc^{-}}$-factor) in $\sfH_{\Lambda_{K}}$ to positive Reeb chord asymptotics of $\sfZ_{L}$, intersecting with the $4$-chain, linking with $4$-chain intersection circles, and taking the result in $\widehat{\Sk}(L)$, see Figure \ref{fig:HZinSFT}.  

\begin{figure}[h!]
	\begin{center}
		\includegraphics[width=0.550\textwidth]{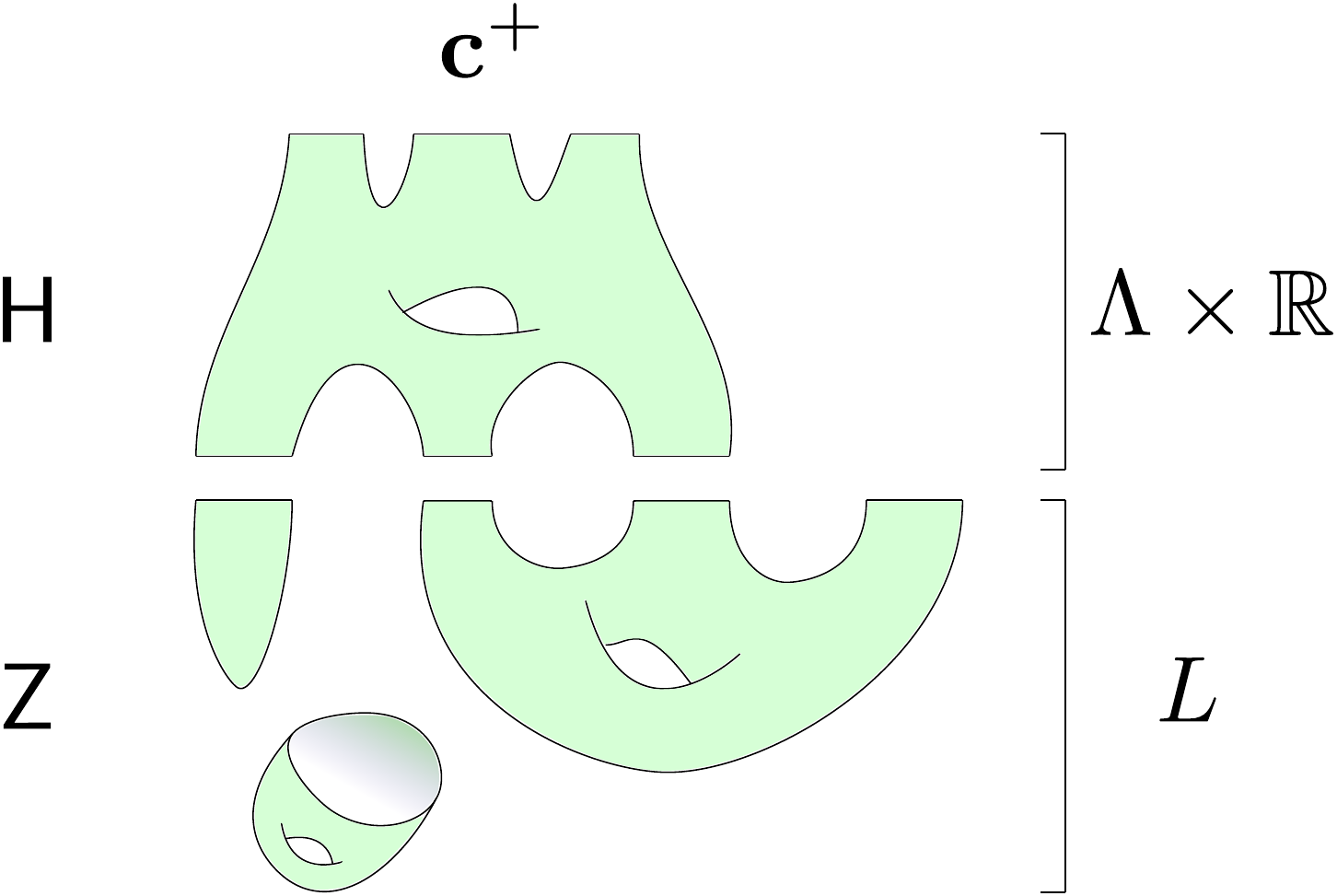}
		\caption{The Hamiltonian acts on the generalized partition function.}
		\label{fig:HZinSFT}
	\end{center}
\end{figure}

Consider a degree $1$ word $\bc^{+}$ of Reeb chords. If $\sfW$ is a degree $1$ element in the BRST-complex then let $\langle \sfW,\bc^{+}\rangle$ denote the component of $\sfW$ with Reeb chord word $\bc^{+}$. The following result is the key basic observation that connects BRST-homology and recursion relations.

\begin{prp}\label{prp: chain map eqn}
	For any degree $1$ Reeb chord word $\bc^{+}$ we have
	\be\label{eq:1-dimboundaryskeinvalued}
	\left\langle\,\rho_{\partial L}(\sfH_{\Lambda_K}) \cdot \overline{\sfZ}_{L}\,,\,\bc^{+}\,\right\rangle \ = \ 0,
	\ee
	where $\rho_{\partial L}(\sfH_{\Lambda_K})$ denotes the worldsheet Hamiltonian $\sfH_{\Lambda_K}$ mapped to the skein of $\Sk(\R\times\Lambda_K)$ viewed as the skein of $\R\times \partial L$, i.e., the skein of the boundary of $L$. 
\end{prp}
\begin{proof}
	SFT-compactness \cite{EGH} together with the fact that the framed skein relations \eqref{eq:skein-rel-intro} correspond to wall crossing \cite{Ekholm:2019yqp} shows that the expression counts the ends in the $1$-dimensional moduli space of holomorphic curves in $X$ with boundary on $L$ and positive asymptotics according to $\bc^{+}$.  	
\end{proof}
Proposition \ref{prp: chain map eqn} implies, in analogy with the $D$-modules discussed above, that the generalized partition function $\overline{\sfZ}_{L}$~corresponds to a homomorphism
\be\label{eq:modulemorphism}
H_{\text{BRST}}(\Lambda_{K}) \ \to \ \widehat{\Sk}(L).
\ee

For a concrete example, Figure~\ref{fig:unknot-Lambda} shows the boundaries of all degree one infinite area holomorphic curves for the conormal $\Lambda_{U}$ of the unknot. There are four holomorphic disks that are asymptotic to the unique degree one Reeb chord $c$ of the unknot conormal in $\R\times U^*S^3$, see Figure \ref{fig:unknot-Lambda}. These disks link the zero section when viewed in $T^*S^3$ differently, resulting in a non-trivial factor of $a$. Using one of the disks itself as capping disk we get 
\begin{figure}[h!]
	\centering
	\includegraphics[width=.7\textwidth]{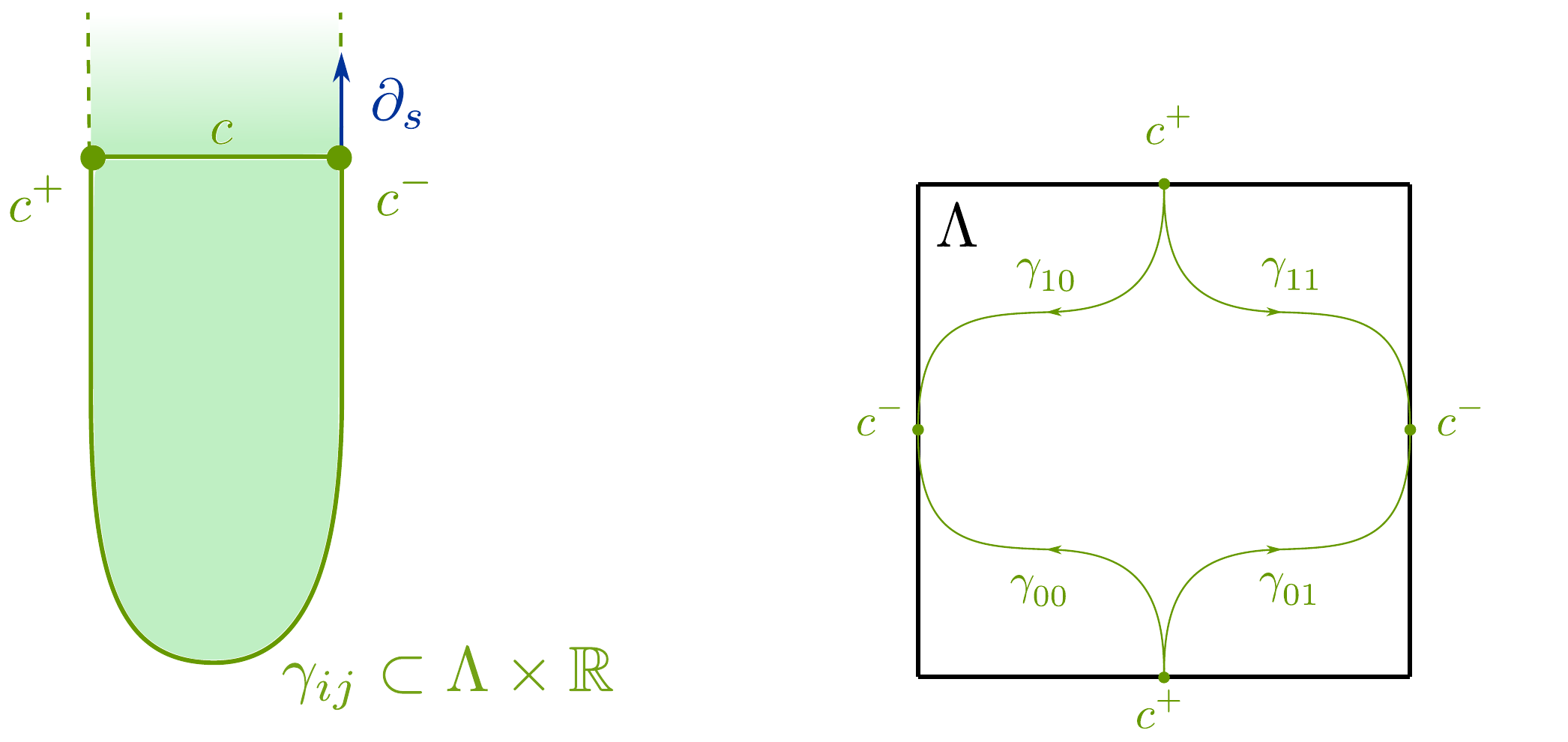}
	\caption{Left: a holomorphic disk with boundary on $L$. 
		Right: the unknot Legendrian torus admits a configuration with a single Reeb chord $c$,. There are four holomorphic disks with pieces of the boundary denoted $\gamma_{ij}$ and extending along $\Lambda$, and with the rest of the boundaries attached to $c^\pm$ and stretched in the transverse $\IR$ direction.}
	\label{fig:unknot-Lambda}
\end{figure}
\be
H_{\text{BRST};\Lambda_{U}}^{0}= \Sk(\R\times U^{\ast} S^{3},\R\times T^{2}) \ / \ 
\left\langle 
\sfP_{0,0}^{\ws}-\sfP_{1,0}^{\ws}-\sfP_{0,1}^{\ws}+\sfP_{1,1}^{\ws} 
\right\rangle.
\ee
{When mapped by the homomorphism $\rho_L$} into the HOMFLYPT skein $\Sk(L_{K})$ of the conormal if we think of $L_{K}\subset X$ or $L_{K}\subset T^{\ast} S^{3}$ the ideal generator maps to 
\be
\sfP_{0,0}-\sfP_{1,0}-a_{L}\sfP_{0,1}+a^{2}\sfP_{1,1} \ \in \ \Sk(L_{K})\text{ or }\Sk(L_{K}\sqcup S^{3}).
\ee

For general knots and links not much is known about the structure of the module $H_{\text{BRST};\Lambda_{K}}$. 
Here we will be mainly interested in the part that connects to the first summand $\sfZ_{L}$ of $\overline{Z}_{L}$ without positive punctures. Consider a Legendrian $\Lambda$ with $k$ components $\Lambda_{j}$, $j=1,\dots,k$. Let $e_{j}$ denote the empty Reeb chord of the $j^{\rm th}$ component. Then the product $e=e_1\dots e_k$ is (as is any degree zero generator) a cycle in the BRST complex. We define the \emph{world sheet skein $D$-module of $\Lambda$} $\sfD_{\Lambda}^{\ws}$ as the submodule in $H_{\text{BRST}}$ generated by $e$, which is then naturally isomorphic to a quotient: 
\[
\sfD^{\ws}_{\Lambda}=\Sk(Y\times\R,\Lambda\times\R)/\sf I^{\ws}.
\]

We will show below that for the Hopf link $\Gamma$, negative asymptotics in $\sfH_{\Lambda_\Gamma}$ can be eliminated by taking linear combinations of positive asymptotics $\bc_{+}$ as in~\eqref{eq:1-dimboundaryskeinvalued}. In other words, we find non-trivial linear combinations of operators $\sfA_{j}^{\ws}\in\Sk(\R\times U^{\ast}S^{3},\R\times\Lambda_{\Gamma})$ in the image of the Hamiltonian that generates $\sfI^{\ws}_{\Gamma}$ and which then satisfy
\be\label{eq:skeinrecursions} 
\langle\sfA_{j}^{\ws}\rangle_{j=1,\dots,n} = \{0\} \ \subseteq \ H_{\text{BRST};\Lambda_{\Gamma}},\qquad \rho_L(\sfA_{j}^{\ws})\cdot \sfZ_{L}=0,\;j=1,\dots,n,
\ee
where $\rho_L$ represents the worldsheet skein element in the skein of $L$ by intersecting with the 4-chain etc, as usual. 

\begin{rmk}\label{rmk:elimination}
	In this paper we restrict attention to the Hopf link. It would be very interesting to know to what extent the picture for the Hopf link generalizes to other knots and links. For example, one could ask whether there are generalized Lagrangians for any point in the augmentation variety that admits a skein module $\Sk(L)$ where we could count curves so that the corresponding wave function corresponds to module homomorphisms as in~\eqref{eq:modulemorphism}. For examples beyond embedded Lagrangian fillings, note that using the recursion in Section \ref{sec:skein-rec-annulus}, one can define skein modules for Lagrangians that intersect cleanly along circles. 
	
	Also, one might ask to what extent it is possible to perform variable elimination in the skein to obtain the counterpart of~\eqref{eq:skeinrecursions}: are polynomial equations sufficient or are power series needed, and if so which power series?
\end{rmk} 

The connection between worldsheet skein $D$-modules, skein valued recursion relations and the augmentation variety is the following. Reducing from the skein to the $U(1)$-skein gives the (exponentiated) $D$-module \eqref{eq:D-module-def-background} and the quantum curve operators \eqref{eq:U1-recursion-review} that give the recursion for symmetric colorings. Then taking the semi-classical limit, $q\to 1$, gives the augmentation variety $V_K$ of knot contact homology.

\section{The HOMFLYPT skein module of the solid torus}\label{sec:skeinsolidtorus}
The framed skein module $\Sk(L)$ of a $3$-manifold $L$ was defined in \eqref{eq:skein-rel-intro}.
If $L$ has boundary $\partial L$ then using a collar neighborhood of $\partial L$, the skein module of $\partial L\times I$, where $I$ is an interval, acts on $\Sk(L)$ by stacking framed links in the collar neighborhood. Similarly, the skein module of $\partial L\times I$ acts on itself and becomes an algebra that we denote $\Sk(\partial L)$ and $\Sk(L)$ is a module over $\Sk(\partial L)$ again using stacking. 

Since the skein relations \eqref{eq:skein-rel-intro} preserve the homology class of a link, the skein module is graded by $H_1(L)$,
\be
	\Sk(L) = \bigoplus_{\alpha\in H_1(L)} \Sk_\alpha (L).
\ee

This section concerns the HOMFLYPT skein module of the solid torus $\Sk(S^{1}\times\R^{2})$ and the skein algebra of its boundary $\Sk(T^{2})$, that was thoroughly studied in~\cite{2014arXiv1410.0859M}, to which we refer for proofs and further detail. In order to invert certain skein operators we will consider a more general base ring. 
Let $R$ denote the ring of Laurent polynomials $\IQ[q^{\pm 1}, a_L^{\pm 1}]$ in $q$ and $a_L$ in with rational coefficients and denominators $q^{k}-q^{-k}$ for positive integers $k$. We consider the skein module as defined over $R$. 
After reviewing the definitions and results of \cite{2014arXiv1410.0859M} for the skein module and skein algebra of the solid torus and its boundary, we introduce in Section \ref{sec:skein of torus with two boundary points} a generalization to include curves ending at punctures.

\subsection{The skein algebra of the torus}\label{sec:skeinoftorus}
The skein algebra $\Sk(T^{2})$ is graded by $H_1(T^{2})\simeq \IZ^2$:
\be
	\Sk(T^{2}) = \bigoplus_{(i,j)\in \IZ^2} \Sk_{(i,j)}(T^{2})\,
\ee
and $\Sk(T^{2})$ acts naturally on $\Sk(\R^{2}\times S^{1})$.

We review the explicit presentation of the algebra $\Sk(T^{2})$ from~\cite{2014arXiv1410.0859M}.
If $\gcd(i,j)=1$ we say that $(i,j)$ is primitive. For primitive $(i,j)$, let $\sfP_{i,j}\in\Sk(T^{2})$ be the element represented by an embedded $(i,j)$-curve in $T^2$, where and $(i,j)$-curve is a curve in the homology class represented by $i$ meridians and $j$ longitudes, see Figure~\ref{fig:skeinalgebra}.
\begin{figure}[h!]
\begin{center}
\includegraphics[width=0.35\textwidth]{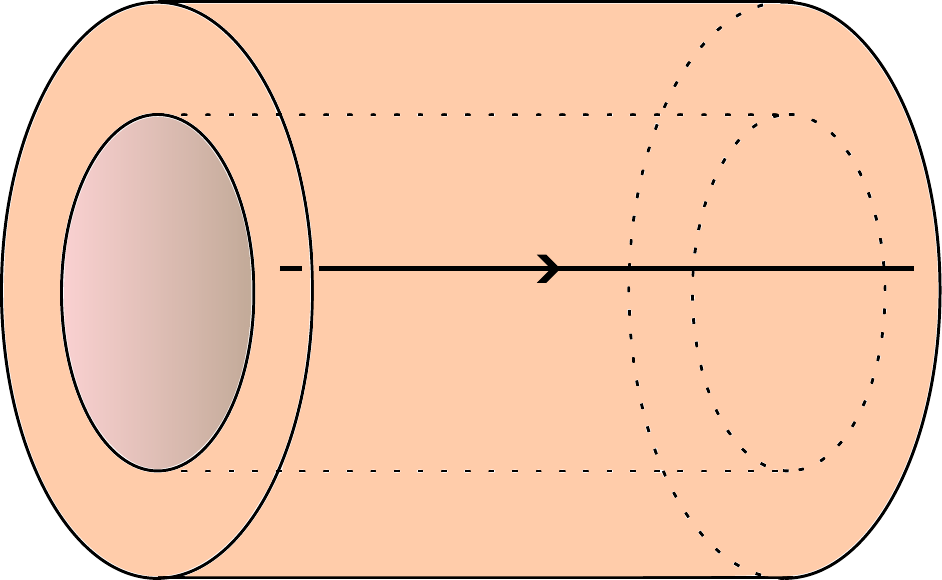}\qquad
\includegraphics[width=0.35\textwidth]{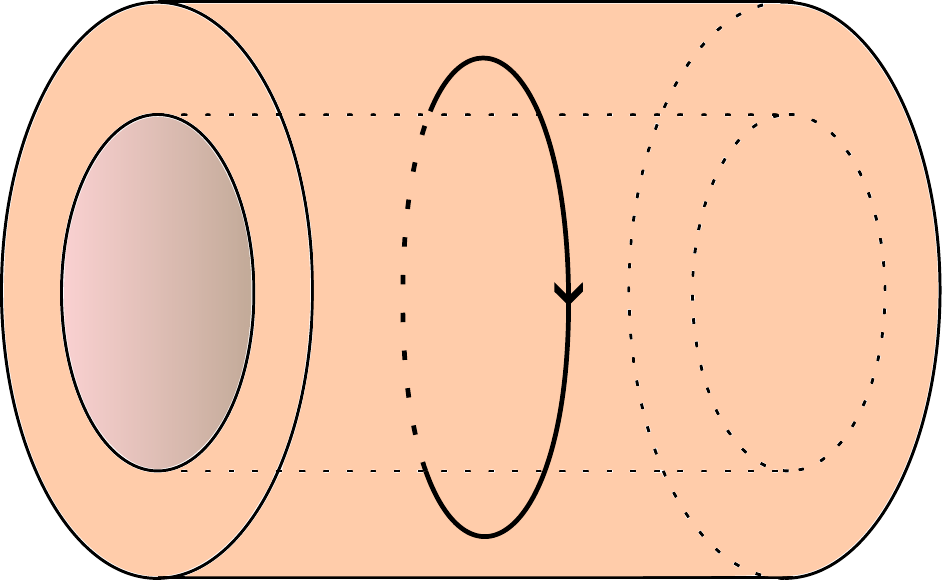}
\caption{Generators $\sfP_{0,1}$ and $\sfP_{1,0}$ of $\Sk(T^{2})$.}
\label{fig:skeinalgebra}
\end{center}
\end{figure}
If $\gcd(i,j)=d >1$, the definition of $\sfP_{i,j}$ is the following. Denote by $(i_0,j_0) = (i/d, j/d)$ the primitive vector of which $(i,j)$ is a multiple. Cut out a small tubular neighborhood $A\subset T^2$ of an embedded $(i_0, j_0)$ curve.
Consider for each $r,s\geq 0$ such that $r+s+1=d>0$, the curves $A_{r,s}$ inside $A\times I$, where $A$ is an annulus, shown in Figure~\ref{fig:powersum}. Identify the tubular neighborhood of the $(i_{0},j_{0})$-curve with $A$ and define $\sfP_{i,j}$ as follows
\be\label{eq:Pij-powersum}
	\sfP_{i,j} = \frac{q-q^{-1}}{q^d-q^{-d}}\sum_{r=0}^{d-1} A_{r,d-1-r}
\ee

\begin{figure}[h!]
\begin{center}
\includegraphics[width=0.35\textwidth]{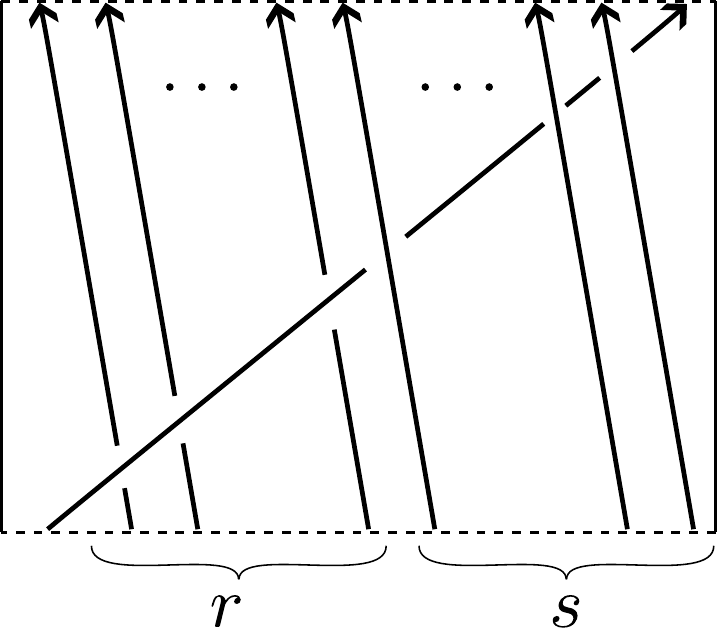}
\caption{The curve $A_{r,s}$ in the skein of the annulus $\Sk(A)$ wraps $d$ times around and has $r$ positive crossings followed by $s$ negative.}
\label{fig:powersum}
\end{center}
\end{figure}

The curves $\sfP_{(i,j)}$, $(i,j)\in \IZ^2$ generate $\Sk(T^{2})$~\cite[Lemma 3.1]{2014arXiv1410.0859M} and obey the following commutation relations~\cite[Theorem 3.2]{2014arXiv1410.0859M}
\be\label{eq:commutator}
	[\sfP_{i,j} ,\sfP_{k,l}] = (q^{(i,j)\wedge (k,l)} - q^{(k,l)\wedge (i,j)}) \sfP_{i+k,j+l}
\ee
where $(i,j)\wedge (k,l) = il-jk$ is the algebraic intersection number of corresponding homology classes.

\subsection{The skein module of the solid torus}\label{ssec:skeinsolidtorus}
It was shown in~\cite{morton2002homfly, hadji2006basis} that $\Sk(S^{1}\times \R^{2})$ admits a basis as an $R$-module given by elements $\sfW_{\lambda,\overline\mu}$, where $\lambda$ and $\mu$ ranges over all partitions, that are all eigenvectors of the meridian operators $\sfP_{\pm 1,0}\in \Sk(T^{2})$. The positive subalgebra $\Sk(S^{1}\times\R^{2})^+$ is generated by $\sfW_{\lambda,\emptyset}$ and is isomorphic to the ring of symmetric functions $\boldsymbol{\Lambda}$ in infinitely many variables.  
Through this isomorphism, $\sfW_{\lambda,\emptyset}$ maps to the Schur function $s_\lambda$. Similarly, the negative subalgebra $\Sk(S^{1}\times\R^{2})^-$, which is generated by $\sfW_{\emptyset,\overline{\mu}}$, is also isomorphic to $\boldsymbol{\Lambda}$. The whole module $\Sk(S^{1}\times\R^{2})$ is then isomorphic to
\be
\Sk(S^{1}\times\R^{2}) \approx \boldsymbol{\Lambda}\otimes_R \boldsymbol{\Lambda}\,.
\ee
Reversing the orientation of all links in $L$ maps $\sfW_{\lambda,\overline{\mu}}$ into $\sfW_{\mu,\overline{\lambda}}$.

The skein elements $\sfW_{\lambda,\overline{\mu}}$ are $R$-linear combinations of links in the solid torus that can be found by direct calculation, diagonalizing the skein algebra operators $\sfP_{\pm1,0}$.
Alternatively, one can use symmetric functions as follows. Results from  
\cite{aiston1998idempotents, lukac2001homfly, lukac2005idempotents, morton2002power} 
show that the sum of closed braids shown in Figure~\ref{fig:powersum} form a generating set for the positive skein corresponding to power sums $p_k$ in the ring of symmetric functions, where the elements $\sfW_{\lambda,\emptyset}$ correspond to Schur functions.
Then changing basis from power sums to Schur functions one gets linear combinations of links for $\sfW_{\lambda,\emptyset}$.
Orientation reversal gives $\sfW_{\emptyset,\overline{\lambda}}$, and $\sfW_{\lambda,\overline{\mu}}$ can be expressed as a linear combination of products of the form $\sfW_{\nu,\emptyset}\cdot \sfW_{\emptyset,\overline{\rho}}$ with integer coefficient. 
We refer to~\cite{hadji2006basis} for a systematic description in terms of systems of curves corresponding to complete symmetric functions.

The first few basis elements are as follows:
\be\label{eq:W-basis}
\begin{split}
	&\sfW_{\ydiagram{1},\emptyset}  = \CIo \\
	&\sfW_{\ydiagram{2},\emptyset}
	 = \frac{1}{q+q^{-1}} \(q^{-1} \ \CIIo\  +\  \CIIi\)\\
	&\sfW_{\ydiagram{1,1},\emptyset}
	 = \frac{1}{q+q^{-1}} \(q\ \CIIo\ -\ \CIIi\)\\
	&\sfW_{\ydiagram{3},\emptyset}
	 = \frac{1}{(q+q^{-1})(q^2+1+q^{-2})} \(q^{-3} \ \CIIIo\  + (1+2 q^{-2})\, \ \CIIIi\  + (q+q^{-1}) \, \ \CIIIii\  \)\\
	&\sfW_{\ydiagram{2,1},\emptyset}
	 = \frac{1}{(q^2+1+q^{-2})} \( \CIIIo\  + (q-q^{-1})\, \ \CIIIi\  -  \CIIIii\  \)\\
	&\sfW_{\ydiagram{1,1,1},\emptyset}
	 = \frac{1}{(q+q^{-1})(q^2+1+q^{-2})} \(q^{3} \ \CIIIo\  - (1+2 q^{2})\, \ \CIIIi\  + (q+q^{-1}) \, \ \CIIIii\  \)
\end{split}
\ee
Here the left and right endpoints of the curves drawn should be identified respecting the order, after which they give link diagrams in a neighborhood of the central curve $S^{1}\times \{0\}$ (which is $\sfW_{\ydiagram{1}}$ above) in the solid torus $S^{1}\times\R^{2}$, where the diagram refers to the projection to $S^{1}\times\R$.

As mentioned, the generators $\sfP_{i,j}$ of the skein algebra $\Sk(T^{2})$ act diagonally on the skein module $\Sk(S^{1}\times \R^{2})$ in the $\sfW_{\lambda,\overline{\mu}}$ basis. For the primitive meridian operators: 
\be\label{eq:P10-eigenvalues}
\begin{split}
	\sfP_{1,0} \, \sfW_{\lambda,\overline{\mu}} 
	& = \left(\frac{a_L-a_L^{-1}}{q-q^{-1}} 
	+ (q-q^{-1}) \left[
	a_L C_\lambda(q^2)-a_L^{-1} C_{\overline{\mu}}(q^{-2})
	\right]\right)\, \sfW_{\lambda,\overline{\mu}} \,,
	\\
	\sfP_{-1,0} \, \sfW_{\lambda,\overline{\mu}} 
	& =\left( \frac{a_L-a_L^{-1}}{q-q^{-1}} 
	+ (q-q^{-1}) \left[
	a_L C_{\overline{\mu}}(q^2)-a_L^{-1} C_\lambda(q^{-2})
	\right]\right) \, \sfW_{\lambda,\overline{\mu}} \,,
	\\
\end{split}	
\ee
where $C_\lambda(s)=\sum_{\ydiagram{1}\in\lambda} s^{c(\ydiagram{1})}
$ is the content polynomial of the partition with $c(\ydiagram{1}) = j-i$, the content of the box in row $i$ and column $j$.

The primitive longitudinal generators act by adding and subtracting boxes as follows.
\be
\begin{split}
	\sfP_{0,1} \, \sfW_{\lambda,\overline{\mu}} & = \sum_{\alpha \in \lambda+\ydiagram{1}} \sfW_{\alpha,\overline{\mu}} + \sum_{\overline\beta \in \overline{\mu}-\ydiagram{1}} \sfW_{\lambda,\overline\beta}\,,
	\\
	\sfP_{0,-1} \, \sfW_{\lambda,\overline{\mu}} & = \sum_{\alpha \in \lambda-\ydiagram{1}} \sfW_{\alpha,\overline{\mu}} + \sum_{\overline\beta \in \overline{\mu}+\ydiagram{1}} \sfW_{\lambda,\overline\beta}\,.
	\\
\end{split}
\ee

The action of other elements of the algebra can be obtained by combined applications of the power sum construction and the commutator~\eqref{eq:Pij-powersum}--\eqref{eq:commutator}, see~\cite[Theorem 4.6]{2014arXiv1410.0859M} for details and further explanation. Explicitly,
for nonzero integers $m,n$:
\be\label{eq:generalformulaePW}
\begin{split}
	\sfP_{0,0} \, \sfW_{\lambda,\overline{\mu}} 
	& = 
	\frac{a_L-a_L^{-1}}{q-q^{-1}} \,  \sfW_{\lambda,\overline{\mu}} 
	\\
	\sfP_{m,0} \, \sfW_{\lambda,\overline{\mu}} 
	& = 
	\left(\frac{a_L^m-a_L^{-m}}{q^m-q^{-m}} 
	+ (q^m-q^{-m}) \left[
	a_L^m C_\lambda(q^{2m})
	-a_L^{-m} C_{\overline{\mu}}(q^{-2m})
	\right]\right)\, \sfW_{\lambda,\overline{\mu}} \,,
	\\
	\sfP_{0,n} \, \sfW_{\lambda,\overline{\mu}} 
	& = 
	\sum_{\alpha \in \lambda+[n]} (-1)^{\hgt(\alpha-\lambda)}\sfW_{\alpha,\overline{\mu}} 
	+ \sum_{\overline\beta \in \overline{\mu}-[n]} (-1)^{\hgt(\overline{\mu}-\beta)} \sfW_{\lambda,\overline\beta}\,,
	\\
	\sfP_{m,n} \, \sfW_{\lambda,\overline{\mu}} 
	& = 
	\frac{q^m-q^{-m}}{q^{mn}-q^{-mn}}
	\Bigg[ 
	\sum_{\alpha \in \lambda+[n]} (-1)^{\hgt(\alpha-\lambda)} a_L^m C_{\alpha-\lambda}(q^{2m})\, \sfW_{\alpha,\overline{\mu}} 
	\\
	&\hspace*{70pt}  \quad+ \sum_{\overline\beta \in \overline{\mu}-[n]} (-1)^{\hgt(\overline{\mu}-\beta)} a_L^{-m} C_{\overline\mu-\beta}(q^{-2m})\, \sfW_{\lambda,\overline\beta}
	\bigg]\,,
\end{split}
\ee
where $\lambda+[n]$ denotes all partitions obtained from $\lambda$ by adding a ``border strip'' of length $n$, and similarly $\overline{\mu}-[n]$ denotes all partitions obtained from $\overline\mu$ by removing a ``border strip'' of length $n$. A border strip $\gamma$ is a skew partition that is connected and has no two-by-two squares, see Figure~\ref{fig:borderstrips}.
\begin{figure}[h!]
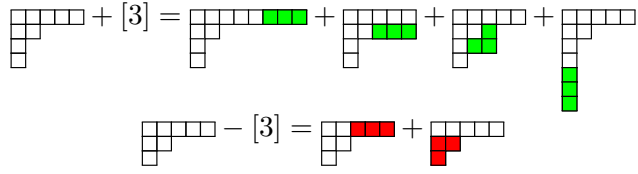

	\begin{center}
		$
		\ydiagram{5,2,1,1} + [3]
		=
		\ydiagram[*(green)]{5+3}*[*(white)]{8,2,1,1}
		+
		\ydiagram[*(green)]{0, 2+3}*[*(white)]{5,5,1,1}
		+
		\ydiagram[*(green)]{0,2+1,1+2}*[*(white)]{5,3,3,1}
		+
		\ydiagram[*(green)]{0,0,0,0,1,1,1}*[*(white)]{5,2,1,1,1,1,1} 
		$
		\\
		$
		\ydiagram{5,2,1} - [3]
		=
		\ydiagram[*(red)]{2+3}*[*(white)]{5,2,1}
		+
		\ydiagram[*(red)]{0, 2,1}*[*(white)]{5,2,1}
		$
		\caption{Addition and subtraction of border strips.}
		\label{fig:borderstrips}
	\end{center}
\end{figure}
The height $\hgt(\gamma)$ of a border strip is the number of rows plus one. 
The formulas extend to $n<0$ where 
\be
\lambda \pm [n] \quad\mapsto \lambda \mp [-n]\,,
\ee
and   
\be
C_{\alpha-\lambda}(q^{2m}) = C_{\alpha}(q^{2m}) - C_{\lambda}(q^{2m}).
\ee

Finally, we discuss the effect of framing change. A change of framing is a diffeomorphism $\Phi^{f}\colon S^{1}\times\R^{2}\to S^{1}\times\R^{2}$ of the form 
\be
\Phi^{f}(e^{i\theta}, u, v) \ = \ (e^{i\theta}, u\cos f\theta , v\sin f\theta ),
\ee
where $f\in\IZ$. It acts on the torus at infinity as
\be
\Phi^{f}(e^{i\theta}, e^{i\phi}) \ = \ (e^{i\theta}, e^{i (f \theta +\phi)} ).
\ee

Given an element $\sfZ$ in the completed positive skein $\widehat{\Sk}_{+}(S^{1}\times\R^{2})$,
\be
	\sfZ \ = \ \sum_{\lambda} H_\lambda(a,q)\, \sfW_{\lambda,\emptyset},
\ee
a change of framing acts as follows
\be\label{eq:framed-Z}
	\Phi^f(\sfZ) = \sum_{\lambda} ((-1)^{|\lambda|} q^{\kappa(\lambda)})^f\, H_\lambda(a,q)\, \sfW_{\lambda,\emptyset}\,,
\ee
where  
\be
\kappa(\lambda) = \sum_{i} (\lambda_i)^2 -  \sum_{j} (\lambda^{t}_j)^2
=
2\sum_{\ydiagram{1}\in\lambda} c(\ydiagram{1})\,.
\ee
The change of framing acts on the skein algebra $\Sk(T^2)$ simply as a coordinate change 
\be\label{eq:framed-Pij}
	\Phi^{f}(\sfP_{i,j})
	\ = \ (-1)^{j\, f\, \sigma}\sfP_{i+j\cdot  f,j},
\ee
where $\sigma=1$ if the spin structure on $(0,1)$ is bounding and $ \sigma=0$ if it is the Lie group spin structure.
Here we require that the 4-chain at infinity remains unchanged. This requirement leads to an intersection between the 4-chain in the new framing along a curve in the class $(f,0)$ which then effectively scales the action of $\sfP_{i,j}$ on $\sfZ$ by $a_L^{-j\cdot f}$,
\be\label{eq:aL-Pij-framed}
	\sfP_{i,j} \to a_L^{-j\cdot f}\, \sfP_{i,j} \,.
\ee
Recall that in the context of \eqref{eq:framed-Z} the solid torus $S^1 \times \R^2$ fills the meridian of the torus at infinity, so we are required to use the bounding spin structure when applying \eqref{eq:framed-Pij}.

\subsection{The skein module of the solid torus with two boundary points}\label{sec:skein of torus with two boundary points} 
In this section we generalize the skein module in Section \ref{ssec:skeinsolidtorus} and include also curves in the skein that are allowed to end at two fixed points in the boundary of the solid torus. In our applications these points will be Reeb chord endpoints. 

We denote this skein module by $\Sk^2(S^1\times \R^2)$. It consists of formal linear combinations over $R$ of framed links with open ends at two fixed boundary points of the solid torus up to framed isotopy and the skein relations \eqref{eq:skein-rel-intro}. We fix an incoming and outgoing boundary point where the open component begins and ends. The precise position of the boundary points does not matter up to isomorphism as the boundary is connected. 

The $R$-module $\Sk^2(S^1\times \R^2)$ has a natural structure of an $R$-algebra as follows. Represent the solid torus as $A\times [0,1]$, where $A$ is an annulus. Let the incoming and outgoing boundary points be $(p,0)$ and $(p,1)$, respectively. Then the multiplication corresponds to gluing $A\times [0,1]$ to $A\times[0,1]$ by identifying $A\times\{1\}$ in the first copy with $A\times\{0\}$ in the second and rescaling the interval factor. The resulting algebra is commutative, see \cite{morton2002homfly}.
Similarly, $\Sk^2(S^1\times \R^2)$ also has the structure of an $\Sk(S^1\times \R^2)$-bimodule compatible with the product structures on $\Sk^2(S^1\times \R^2)$ and $\Sk(S^1\times \R^2)$ as follows. The maps
\begin{align}\label{eq:landr}
	&l \colon \Sk(S^1\times \R^2) \otimes \Sk^2(S^1\times \R^2) \ \to \ \Sk^2(S^1\times \R^2),\\\notag
	&r  \colon  \Sk^2(S^1\times \R^2) \otimes \Sk(S^1\times \R^2) \ \to \ \Sk^2(S^1\times \R^2),
\end{align}
defined by placing a solid torus to the left, respectively to the right, of the solid torus with the boundary points are clearly homomorphisms of $R$-algebras. 

For $n \in \IZ$, let $\sfC'_{n}$ denote the arc from the boundary point at the top of the solid torus to the boundary point at the bottom which goes $n$ times around the solid the torus.
Then $\sfC'_0$ is the unit of the algebra $\Sk^2(S^1\times \R^2)$ and $\sfC'_m \sfC'_n = \sfC'_{n+m}$. Then, as an algebra over $\Sk(S^1\times \R^2)$ with respect to either $l$ or $r$, $\Sk^2(S^1\times \R^2)$ is the algebra $\Sk(S^1\times \R^2)[\sfC'_{\pm1}]$ of Laurent polynomials in $\sfC'_1$ over $\Sk(S^1\times \R^2)$, see \cite{lukac2001homfly}.

Any element $A \in \Sk^2(S^1\times \R^2)$ induces maps 
\be\label{eq:landr2} 
l(\cdot,A), r(A,\cdot)  \colon   \Sk(S^1\times \R^2) \ \to \ \Sk^2(S^1\times \R^2).
\ee 
In the special case that $A=\sfC_0$ or $A=\sfC_1$, we will write 
\be
\begin{split}
&\sfP'_{1,0} \coloneqq l(\cdot,\sfC_0) \qquad \sfP'_{0,0} \coloneqq r(\sfC_0,\cdot)\\
&\sfP'_{1,1} \coloneqq l(\cdot,\sfC_1) \qquad \sfP'_{0,1} \coloneqq a_L^{-1} r(\sfC_1,\cdot).
\end{split}
\ee
After choosing appropriate capping paths, $\sfP'_{i,j}$ is mapped to $\sfP_{i,j}$ viewed as an operator on the skein of the solid torus.

\section{Skein valued recursions -- first examples }\label{sec:unknot}
In this section we discuss skein valued recursions for the solid torus, the toric brane, the unknot conormal, and a neighborhood of a basic holomorphic annulus. The latter three were worked out in~\cite{Ekholm:2020csl, Ekholm:2021colored}. We mostly follow the description there but also add a new version of the recursion for the basic annulus that involves a certain power series.

\subsection{The empty solid torus}
In this section we consider the simplest case of the theory on the cotangent bundle of $T^{\ast}(S^{1}\times\R^{2})$ with no non-constant holomorphic curves. Let $L=S^{1}\times\R^{2}$ in the ambient space which is simply the cotangent bundle $T^{\ast}L$ itself. 
In order to obtain a recursion relation, we consider a version $W$ of $T^{\ast} L$ that is convex at infinity, with $L\subset X$ being a Lagrangian with asymptotic Legendrian boundary torus $\Lambda\subset \partial W$. 
Here $W$ is $T^{\ast} S^{1}\times \C^{2}$ with $L$ equal to $S^{1}\times \R^{2}$. 
As in~\cite{Ekholm:2021colored} we see that $\partial L$ is $S^{1}\times U$, where $U$ is the standard Legendrian unknot in the 3-ball. 
This means that there is an $S^{1}$ Bott-family of degree $1$ Reeb chords. Morsifying this configuration, we find exactly one Reeb chord of degree $1$ corresponding to the minimum in the Bott family. There are two holomorphic disks of the unknot that are asymptotic to this Reeb chord. 
The corresponding recursion operator in the worldsheet skein with two boundary arcs is 
\be
\sfP_{0,0}^{\ws}-\sfP_{1,0}^{\ws}, 
\ee
which for an appropriate $4$-chain and after choosing one of the disks itself as capping surface gives the skein operator 
\be
\sfP_{0,0}-\sfP_{1,0} \ \in \ \Sk(L).
\ee
We obtain the skein recursion 
\be
(\sfP_{0,0}-\sfP_{1,0})\sfZ_{L} = (\sfP_{0,0}-\sfP_{1,0})1=0.
\ee
Here we used $W_{\emptyset,\emptyset} \equiv 1$ to denote the empty link as an element in the skein. 
The $U(1)$-recursion $1-\hat y=0$ is obtained by specializing the above to the $U(1)$-skein where $\sfP_{0,0}\mapsto 1$ and $\sfP_{1,0}\mapsto \hat y$. We point out that $\sfZ=1$ in this case since there are no (bare) holomorphic curves.

\subsection{The toric brane}\label{sec:recSktoricbrane}
Our next example is the toric brane in $\C^{3}$. This is a Lagrangian $L\approx S^{1}\times\R^{2}$ that supports a single holomorphic disk (and its branched covers), see~\cite{Aganagic:2000gs}. The skein recursion operator for the toric brane was found in~\cite{Ekholm:2020csl}, reading off the curves in the world sheet skein gives 
\be\label{eq:relrecSktoricbrane}
\sfP_{0,0}^{\ws} - \sfP_{1,0}^{\ws} - \sfP_{0,1}^{\ws} \,.
\ee
which specializes to 
\be\label{eq:relrecSktoricbrane}
\sfP'_{0,0} - \sfP'_{1,0} - a_{L}\sfP'_{0,1} \,.
\ee
in the open skein of $L$ which for suitable capping disks gives
\be\label{eq:recSktoricbrane}
\sfP_{0,0} - \sfP_{1,0} - a_{L}\sfP_{0,1} \,.
\ee
Reducing to $U(1)$ we get $1-\hat y-q\hat x$, which after the substitution $x\mapsto qx$ gives the standard $1-\hat y-\hat x$. 

We next discuss signs and framing factors in~\eqref{eq:recSktoricbrane}. As explained in~\cite{ekholm2013knot}, signs on holomorphic curves depend on the choice of a spin structure on the Lagrangian. Here $L\approx S^{1}\times\R^{2}$ supports two distinct spin structures and (assuming the longitude generates $H_{1}(L)$), changing spin structure changes the sign of the last term in~\eqref{eq:recSktoricbrane}. In order to determine the remaining sign we use the representation of $\Lambda$ as the front shown in~\cite[Figure 9]{DRGo} and observe, see~\cite{ekholm2013knot}, that the two disks in question are copies of the disks on the standard 1-dimensional unknot with a stabilized boundary condition. Since for the standard unknot filled by a Lagrangian 2-disk these two disk form the boundary of a 1-dimensional moduli space it follows that their relative sign is negative for the bounding spin structure on the circle (here identified with the meridian). We conclude that also the first sign is as stated for spin structures on $\partial L$ that extend to $L$. Finally, we consider framing data. We take a capping path close to the unknot disk. Then the second disk looks like a meridian and the third like a splicing of a longitude and a meridian in $\Lambda$ which becomes an unknot in $L$. The $a_{L}$-factor comes from the kink in the last curve when viewed as a curve in the solid torus $L$. 

We finally discuss this from the point of view of skein module homomorphisms over the augmentation variety, see Remark \ref{rmk:elimination}. The augmentation variety has three asymptotic regions, $x\to 0$, $y\to 0$, and $x,y\to \infty$ with constant ratio. 
Each of these corresponds to a Lagrangian filling $L$ of topology $S^{1}\times\R^{2}$ that fill in $\sfP_{0,1}$, $\sfP_{1,0}$, and $\sfP_{1,1}$, respectively and each gives homomorphisms 
\[
H_{\Lambda;\text{BRST}} \ \to \ \Sk(L),
\]
where the homomorphisms take the empty word to $\sfZ_{L}$ given by the partition function in~\cite{Ekholm:2020csl}.
 
The recursion relation in generic framing can be obtained by transforming~\eqref{eq:recSktoricbrane} according to~\eqref{eq:framed-Pij} and \eqref{eq:aL-Pij-framed} 
\be\label{eq:eq:skein-rec-toric-framed}
	(\sfP_{0,0} - \sfP_{1,0} +(-1)^{1+f} a_L   \, \sfP_{f,1})\, \sfPsi_{\mathrm{di}}^{(f)} =0\,,
\ee
where the specialization $a_\Lambda = a_L$ was taken into account, since we take the operator to act on $\Sk(L)$.
This uniquely fixes the skein valued partition function of a framed toric brane $\sfPsi_{\mathrm{di}}^{(f)}$ 
with initial conditions zero for all negative curves
\be
	\sfPsi_{\mathrm{di}}^{(f)} 
	= \sum_\lambda  ((-1)^{|\lambda|} q^{\kappa(\lambda)})^f  \left[\prod_{\ydiagram{1}\in\lambda} \frac{- q^{-c(\ydiagram{1})}}{q^{h(\ydiagram{1})}-q^{-h(\ydiagram{1})}}\right]  \sfW_{\lambda,\emptyset}\,.
\ee
We point out that the framing change formulas assumes that the 4-chain are unaffected by the framing change at infinity. This gives an intersection of the 4-chain and the Lagrangian along a curve in class $(0,f)$ that shifts the action of the operator $\sfP_{0,1}$ by multiplication by~$a_{L}^{f}$.

\subsection{The unknot}\label{sec:unknot}
In this section we consider the skein recursion for various different Lagrangian fillings of the unknot Legendrian torus. In the limit $a\to 0$ each of these give different fillings for the toric brane discussed earlier.
The augmentation variety of the unknot is
\be
	1-y-x+a^2 xy = 0.
\ee
We think of this as a plane curve and take the tropical limit, which means we consider the curve in $(\C^\ast)^{2}\approx (\R\times S^1)^{2}$ and take the limit as the radius of $S^{1}$ shrinks to $0$ or equivalently draw $(\log_{t}|x|,\log_{t}|y|)$ as $t\to \infty$. The resulting tropical curve is then subdivided into five regions that we label by the solid torus Lagrangian filling that gives the corresponding augmentation.
\begin{itemize}
	\item $L_{l^{+}}$: $\{|x|<1,\  |y|= 1\}$
	\item $L_{m^{+}}$: $\{|x|= 1 ,\ |y|<1\}$
	\item $L_{l^{-}}$: $\{ |x|>a^2 ,\  |y|= a^2\}$
	\item $L_{m^{-}}$: $\{|x|= a^2,\ |y|>a^2\}$
	\item $L_{d}$:  $\{1<|x|=|y|<a^2\}$.
\end{itemize}	
We discuss skein valued curve counts and recursions for three of them, since the remaining two are related by a simple change of variables.

We start with the generator of the world sheet skein $D$-module.
The ideal boundary  of the unknot conormal Lagrangian $L_U$ is a Legendrian torus $\Lambda_U \subset U^{\ast}S^3\approx S^{3}\times S^{2}$  which can be thought of as the ideal contact boundary of either $T^{\ast}S^{3}$ or the resolved conifold $\CO_{\IP^1}(-1)\oplus \CO_{\IP^1}(-1)$. The Legendrian $\Lambda_{U}$ has a single Reeb chord $c$ of degree one and no Reeb chords of degree zero or lower. This means that the recursion relation in the worldsheet skein can be read off directly from the curves at infinity that asymptotes to the Reeb chord $c$. These curves were determined using Morse flow trees in~\cite{ekholm2013knot} and the resulting curves are shown in Figure~\ref{fig:unknot-legendrian}: there are four rigid disks with a positive puncture at $c$. We get the ideal generator in the worldsheet skein
\be
\sfA_U = \sfP_{0,0}^{\ws} - \sfP_{1,0}^{\ws} -  \sfP_{0,1}^{\ws} +   \sfP_{1,1}^{\ws}.
\ee

\begin{figure}[h!]
	\begin{center}
		\includegraphics[width=0.8\textwidth]{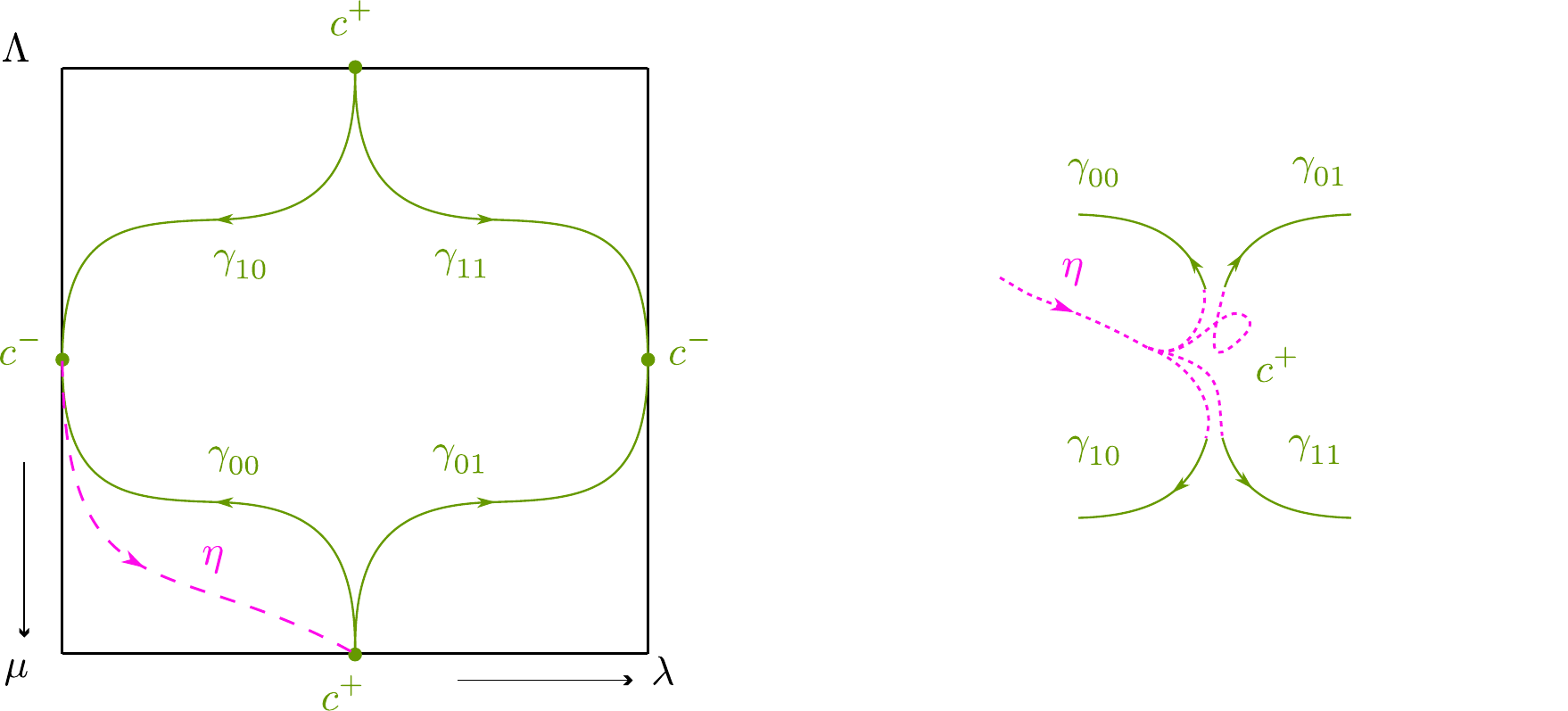}
		\caption{Left: The Legendrian torus of the unknot. Four disks boundaries are marked in solid green, a choice of capping path by a dashed line. Homology classes of boundary curves with this choice of capping are indicated in units of meridian and longitude windings $(\mu,\lambda)$. Right: joining the capping path to each of the disk boundaries. Joining to $\gamma_{01}$ introduces a positive curl, which is responsible for the power of $a_L$ in~\eqref{eq:skein-rec-unknot}.}
		\label{fig:unknot-legendrian}
	\end{center}
\end{figure}

\subsubsection{$L_{l^{\pm}}$ -- the conormal filling}\label{sec:unknot-conormal}
The skein valued recursion relation and its solution for the unknot conormal filling $L_{l^{\pm}}$ were first worked out in~\cite{Ekholm:2020csl}. Note that the sign in $l^{\pm}$ corresponds to the direction of the shift of the conormal. The orientation of the basic annulus changes according to the shift, the unknot in $S^{3}$ does not change but the direction of the curves in the conormal does. It is easy to find the relevant change of variables and we therefore focus on on case $L_{l^{+}}$, say.

In order to write a skein recursion with operators in $\Sk(\Lambda_{U})$ we pick a capping disk to close the disks off. With the boundary of that capping disk $\eta$ in Figure~\ref{fig:unknot-legendrian} we get the following recursion: 
It was shown in~\cite{Ekholm:2020csl} that the count of curves at infinity in the skein uplifts to the following operator in the skein of $\Lambda_{U}$ for a certain choice of 4-chains and spin structures:
\be\label{eq:skein-rec-unknot}
	\sfA_U = \sfA_{U;\lambda}= \sfP_{0,0} - \sfP_{1,0} - a^{-1} a_{L_{\lambda}} \sfP_{0,1} +  a \sfP_{1,1}.
\ee
The classes of the curves are evident from the picture. We discuss signs and framing factors. 

The signs in the formula depend on a spin structure. The particular signs here correspond to the spin structure on the boundary torus that restrict to the bounding spin structure on both the meridian and the longitude circles. The sign follows from index bundle considerations, see~\cite[Theorem 6.4]{ekholm2013knot} (where the calculation is carried out in the Lie group spin structure on $\Lambda_{U}$ is used). 

Consider next powers of $a$. In~\cite{ekholm2013knot} intersections of the four disks where determined relative a specific dual $4$-chain in $T^{\ast} S^{3}$ of the fiber sphere in $U^{\ast}S^{3}$. For the skein valued count we use that 4-chain together with its image under fiber reflection in $T^{\ast}S^{3}$. The result there would then correspond, in our current notation to coefficient $a^{0}$ for $\sfP_{0,0}$, $\sfP_{1,0}$, $\sfP_{0,1}$, and $a^{2}$ for $\sfP_{1,1}$. Here we must consider also intersections with the negative 4-chain near the south pole of the fiber sphere. For the curves bounded by $\sfP_{0,0}$ and $\sfP_{1,0}$ there are no intersections, the curve $\sfP_{1,1}$ has one intersection which now gives $a^{1}$ (formerly this intersection gave $a^{2}$), and the curve $\sfP_{0,1}$ has a negative intersection near the south pole which gives the coefficient $a^{-1}$. (To see the negative  sign, note that the linear combination of curves with boundary $\sfP_{1,1}-\sfP_{1,0}-\sfP_{0,1}$ represents a cycle that wraps once around the generator of $H^{2}(U^{\ast} S^{3})$ corresponding to $a^{2}$.) 

Finally, the power of $a_{L_{l^{+}}}$ in front of $\sfP_{0,1}$ arises as the corresponding factor in~\eqref{eq:recSktoricbrane} (note that $\sfP_{1,1}$ need not be related to any curve in the solid torus).

The skein recursion operator~\eqref{eq:skein-rec-unknot} reduces to the $U(1)$ recursion operator when specializing $N=1$, which corresponds to the following replacements 
\be
	a_{\Lambda_U} \to  q\,,\qquad
	\sfP_{0,0} \to 1\,,\qquad
	\sfP_{1,0} \to \hat y\,, \qquad
	\sfP_{0,1} \to \hat x\,, \qquad
	\sfP_{1,1} \to -q \hat x\hat y\,.
\ee

Also, the recursion operator in generic framing can be obtained by transforming ~\eqref{eq:skein-rec-unknot} according to~\eqref{eq:framed-Pij} 
\be\label{eq:eq:skein-rec-unknot-framed}
	\sfA_U^{(f)} = 
	\sfP_{0,0} - \sfP_{1,0} - (-1)^{f} a^{-1}  a_{L_{l^{+}}} \sfP_{f,1} + (-1)^{f} a \sfP_{f+1,1}\,.
\ee
with the proviso that, an additional rescaling of $\sfP_{i,j}$ by appropriate powers of the specialized framing variable $a_{\Lambda_U}=a_L$ should be taken into account, according to \eqref{eq:aL-Pij-framed}.

The skein valued partition function annihilated by this operator is then
related to the generating series of HOMFLYPT polynomials considered in~\cite{Ekholm:2020csl} by a change of framing~\eqref{eq:framed-Z}, namely
\be\label{eq:Z-unknot-explicit}
	\sfZ^{(f)}_U = \sfZ^{(f)}_{U;\lambda} = \sum_\lambda  ((-1)^{|\lambda|} q^{\kappa(\lambda)})^f  \left[\prod_{\ydiagram{1}\in\lambda} \frac{a q^{c(\ydiagram{1})}-a^{-1} q^{-c(\ydiagram{1})}}{q^{h(\ydiagram{1})}-q^{-h(\ydiagram{1})}}\right]  \sfW_{\lambda,\emptyset}\,.
\ee
In framing zero this factorizes into a product of two disks, see \eqref{eq: unknot skein quiver} below.

We finally discuss the dependence on the capping path $\eta$. The role of the capping path is simply to write the recursion relation in the skein of the torus $\Sk(T^{2})$, which is natural from the point of view of Chern Simons theory. From the geometric view point of curve counting the recursion relation rather lives in the the skein of the torus minus the two Reeb chord endpoints (one needs to perform an oriented blow up near the endpoints to keep track of framing data). In any case the relative recursion is more refined than the recursion in $\Sk(T^{2})$ and gives rise to many recursions there. Geometrically, that corresponds to the choice of capping path $\eta$. As the capping path is a dummy curve that lies above all actual holomorphic curve it does not interfere with the actual dynamics of holomorphic curves and any capping path gives rise to a recursion.   

For an example of this we consider the change of capping path shown in Figure~\ref{fig:unknot-legendrian-recap}.
\begin{figure}[h!]
\begin{center}
\includegraphics[width=0.5\textwidth]{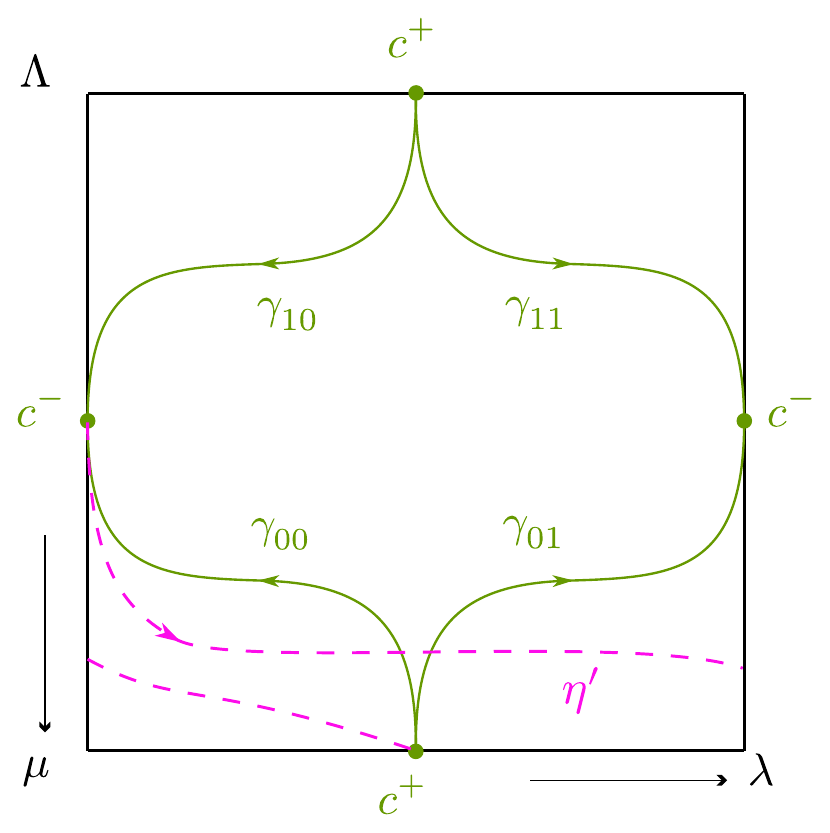}
\caption{The Legendrian torus of the unknot with a different choice of capping path.}
\label{fig:unknot-legendrian-recap}
\end{center}
\end{figure}
Consider first the effect of capping path change on the original framing zero recursion. We get the new recursion (see below for details)
\be\label{unknot-A-recap}
	\sfA_{U;{l^{+}}}^{[0,1]} = 
	a_{L_{l^{+}}} \sfP_{0,1} 
		- \sfP_{1,1}
		-a_{L_{l^{+}}} a^{-1} \frac{1}{2} ((q+q^{-1}) \sfP_{0,2} +(q-q^{-1})(\sfP_{0,1})^2)
		+ a \sfP_{1,2}\,,
\ee
which then as explained above annihilates the original zero-framing partition function $Z$
\be
	\sfA_{U;{l^{+}}}^{[0,1]} \cdot \sfZ_{U;{l^{+}}} = 0.
\ee
We point out that the operator $\sfA_U^{[0,1]}$ is not related to $\sfA_U$ in any algebraically simple way. 
To see the geometric origin of the operator in~\eqref{unknot-A-recap}, we consider how curves at infinity change between Figures~\ref{fig:unknot-legendrian} and~\ref{fig:unknot-legendrian-recap}.
Observe that, in addition to shifting homology classes of curve boundaries by the same amount $(0,1)$, there is an overall rescaling of $\sfP_{0,0}$ (which becomes $\sfP_{0,1}$) by a power of  $a_{\Lambda_U}$. 
The reason for this is evident from Figure \ref{fig:unknot-legendrian-recap} where the concatenation of $\eta'\star\gamma_{00}$ has a self-intersection. 
On the other hand, while a similar self-intersection is also present for the concatenation of $\eta'\star\gamma_{01}$, in this case no rescaling by $a_{L_{{l^{+}}}}$ is present. 
This is because the combination of terms in homology class $(0,2)$ from \eqref{unknot-A-recap} corresponds to a connected twice-around curve, which does not have any such a self-crossing (recall \eqref{eq:W-basis}).

\subsubsection{$L_{m^{\pm}}$ -- the complement filling}\label{sec:unknot-complement}
In the complement filling $L_{m^{\pm}}$ it is the $(0,1)$ curve that bounds a disk in the Lagrangian, while the $(1,0)$ curve at infinity plays the role of the longitude. To take this into account we would have to re-define how the algebra of $P_{i,j}$ operators acts on the $\sfW_{\lambda,\bar\mu}$ basis of this new skein module. In the new definition we would have that $\sfP_{1,0}$ would add a box and $\sfP_{0,1}$ would now act diagonally.

We take instead an alternative route, and change basis on the homology $H_1(\Lambda)$ of the torus at infinity by the $SL(2,\mathbb{Z})$ transformation $S^{-1}$ whose action is defined as
\be
	\sfP_{i,j} \to a^{j}\, \sfP_{-j,i}\,.
\ee
Here we also implemented a change of 4-chains, which is accounted for by the area factor $a^j$, see Section \ref{sec:conormal-complement-filling} for an explanation of the geometric origin of this shift.
With this change of labeling, the recursion for the conormal filling \eqref{eq:skein-rec-unknot} becomes
\be\label{eq:skein-rec-unknot-complement}
	\sfA_{U;\mu} = 
	\sfP_{0,0}
	- a_{L_{m^{+}}}^{-1} \sfP_{0,1}
	- \sfP_{-1,0}
	+ a^2 \sfP_{-1,1}
\ee
The solution to the recursion relation
\be
	\sfA_{U;{m^{+}}} \cdot \sfZ_{U,{m^{+}}} = 0
\ee
is the complement filling skein-valued partition function
\be
\sfZ^{c}_{U,{m^{+}}}  
= 
\sfPsi_{\mathrm{di}}^{(1)}\,  \cdot \, \sfPsi_{\mathrm{di}}^{(0)}[a^2]
= \sum_{\lambda} \prod_{\ydiagram{1}\in\lambda} \frac{q^{c(\ydiagram{1})} - a^2 q^{-c(\ydiagram{1})}}{q^{h(\ydiagram{1})}-q^{-h(\ydiagram{1})}}\, \sfW_\lambda \,.
\ee

\subsubsection{$L_{d}$ -- the middle leg filling}\label{sec:unknot-middle-leg}
For the middle leg filling $L_{d}$ the contractible curve is $(1,-1)$. 
As in the case of the complement filling we change basis for $H_1(\Lambda)$ to take this into account.
In this case we use $ST$ which generates an $\IZ_3$ subgroup of $SL(2,\IZ)$, i.e. $(ST)^3=1$
\be
	\sfP_{i,j} \to  (-1)^i\sfP_{j,-i-j}\,.
\ee
Again we implement an change of the 4-chain. The sign change can be understood as a result of a change in the spin structure.
The skein-valued recursion operator becomes 
\be\label{eq:skein-rec-unknot-middle}
	\sfA_{U;d} = 
	\sfP_{0,0} +a_{L_{d}}^{-1} \sfP_{0,-1} -  \sfP_{1,-1} - a^2 \sfP_{1,-2}
\ee
The solution to the recursion relation is
\be
\begin{split}
	\sfZ_{U;d}
	& = 
	\sum_{\lambda,\overline\mu} 
	\(\prod_{\ydiagram{1}\in\lambda}\frac{q^{-c(\ydiagram{1})}}{q^{h(\ydiagram{1})}-q^{-h(\ydiagram{1})}} \)
	\( \prod_{\ydiagram{1}\in\overline\mu}\frac{q^{-c(\ydiagram{1})}}{q^{h(\ydiagram{1})}-q^{-h(\ydiagram{1})}}\)
	a^{2|\overline\mu|}
	\sum_{\sigma,\overline \tau,\kappa} c_{\kappa\lambda}^\sigma c_{\kappa\mu}^\tau W_{\sigma,\overline\tau}
\end{split}
\ee
where $c_{\kappa\lambda}^\sigma$ are Littlewood-Richardson coefficients. Note that $\sfZ_{U;d}$ is the product of two disc partition functions whose boundaries go along the longitude of $L_d$ in opposite directions.

\subsection{The unlink}
We consider two unlinked unknots. Here, holomorphic curves can be read off from big disks and flow lines at infinity as in \cite{Ekholm:2018eee}. Alternatively, one can shrink the unknot components and move them far apart. After that an action argument shows that the unknot moduli spaces of the degree one Reeb chord of each component cannot contain any curves with negative Reeb chord asymptotics. This means that the worldsheet skein ideal contains both unknot worldsheet skein operator polynomials $\sfA_{U_{1}}$ and $\sfA_{U_{2}}$.

Specializing to the skein we then have the generators of the ideal given by 
\[
\sfA_{U_1}\otimes 1, \qquad 1\otimes \sfA_{U_2}.  
\]
We have Lagrangian fillings given by combinations of the conormal, complement, and middle leg fillings of the unknot components and the ideal generators imply that in all cases the partition functions are products and there are no contributions from curves stretching between the branes: 
\[
\sfZ_{L_{U_1;s}\cup L_{U_2;t}} = 
	\sfZ_{L_{1;s}}\otimes  \sfZ_{L_{2;t}}
,\quad
s,t\in\{l^{\pm},m^{\pm},d\}.
\]   

We point out that this contains non-trivial information for the conormal-complement filling. The complement filling $L_{U_2;m}$ say intersects the conormal of $L_{U_1;l}$ in a circle before shifting that after shifting gives a basic holomorphic annulus stretching between $L_{U_1;l^{+}}$ and $L_{U_2;m^{+}}$. Now stretching first around the complement $L_{U_2;m^{+}}$ and then around $S^{3}$ we find four disks appearing from the two basic annuli. If the framing variable for $S^{3}$ is $a$ and that for the complement of $U_{2}$ is $a_{L_{2}}$ then these disks have homology classes $a_{L_{2}}a$, $a_{L_{2}}a^{-1}$, $a_{L_{2}}a^{-1}$, and $a_{L_{2}}^{-1}a^{-1}$. Here $a_{L_{2}}a$ and one of the $a_{L_{2}}a^{-1}$ are anti-disks whereas $a_{L_{2}}^{-1}a^{-1}$ and one of the $a_{L_{2}}a^{-1}$ are disks and we find that the total contribution of the four disks are that of two disks $a_{L_{2}}a$ and $a_{L_{2}}^{-1}a^{-1}$ in line with the split ideal provided framing variables are specialized correctly.

\subsection{Basic holomorphic annuli}\label{sec:skein-rec-annulus}
In~\cite{Ekholm:2021colored} an induced skein relation near a basic holomorphic annulus was found.
Here we consider the ambient space $T^{\ast} S^{1}\times \C^{2}$ with Lagrangians $S^{1}\times\R^{2}$ and $S^{1}\times P$, where $P$ is a Lagrangian plane close to $\R^{2}$. We then shift the second Lagrangian along a 1-form and find a basic holomorphic annuls. At infinity the Lagrangian has ideal boundary $\Lambda_{1}\cup\Lambda_{2}$, where $\Lambda_{1}\cup\Lambda_{2}$ intersects each fiber $S^{3}$ in a Legendrian Hopf link. As explained in~\cite{Ekholm:2021colored}, after small perturbation $\Lambda_{1}\cup\Lambda_{2}$ has two degree $0$ Reeb chords $p$ and $q$ connecting different components of $\Lambda$, and two degree $1$ Reeb chords $c_{1}$ and $c_{2}$, starting and ending on the same component and we can find all holomorphic curves. This gives the operator equation in the worldsheet skein $\Sk(T^{\ast} S^{1}\times \C^{2},S^{1}\times\R^{2}\sqcup S^{1}\times P)$
\be
{\sfP}_{0,0}^{\ws,(1)} - {\sfP}_{1,0}^{\ws,(1)} = {\sfP}_{0,0}^{\ws,(2)} - {\sfP}_{1,0}^{\ws,(2)}\,.
\ee
This then leads to an operator relation in the relative skein of $S^{1}\times\R^{2}\sqcup S^{1}\times P$,
\be
	({\sfP'}_{0,0}^{(1)} - {\sfP'}_{1,0}^{(1)})\otimes a_{\Lambda_2} = a_{\Lambda_1}\otimes({\sfP'}_{0,0}^{(2)} - {\sfP'}_{1,0}^{(2)})\,,
\ee
or choosing capping paths
\be
(\sfP_{0,0}^{(1)} - \sfP_{1,0}^{(1)})\otimes a_{\Lambda_2} = a_{\Lambda_1}\otimes(\sfP_{0,0}^{(2)} - \sfP_{1,0}^{(2)}).
\ee 
This recursion relation was found combining two moduli spaces of curves with positive puncture at $c_{1}$ and $c_{2}$  to cancel curves with negative punctures at $p$ and $q$. Here we take direct a approach and simply compute the contributions from the canceled terms. The curves at infinity with positive puncture at $c_{1}$ are (except for $\sfP_{0.0}$ and $\sfP_{1,0}$) a triangle with negative punctures at the mixed chords $p$ and $q$. We compute the curves on the filling $L_{1}\cup L_{2}$ with the corresponding positive asymptotics.
\begin{prp}\label{prp:flowdisks}
For each integer $m$ there is exactly one disk with two positive punctures with boundary corresponding to the curve $C'_{m}$ in Figure~\ref{fig:annulus1} in each Lagrangian and no other holomorphic curves. 	
\end{prp}	

\begin{proof}
We use the correspondence between holomorphic curves and Morse flow trees~\cite{Ekholmflowtrees}.
As $P$ degenerates to $\R^{2}$ and the $1$-form used for the shift goes to zero, holomorphic curves limit to flow trees. The flow trees in question must be a flow line since there are locally only two sheets. This flow line is attached to the boundary of a holomorphic disk with boundary on $\R^{2}$ and positive puncture at the limit chord corresponding to $c_{1}$. See Figure~\ref{fig:annulus2}.	
\end{proof}	

\begin{figure}[ht]
	\begin{center}
		\includegraphics[width=0.4\textwidth]{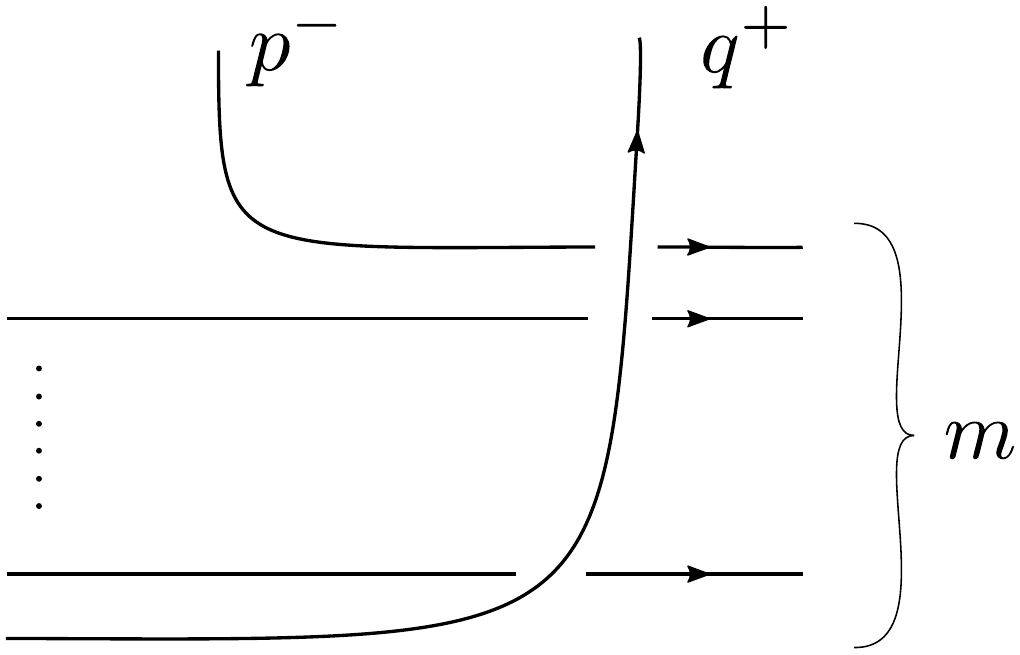}
		\caption{The curve $\sfC'_{m}$ connecting Reeb chord endpoints.} 
		\label{fig:annulus1}
	\end{center}
\end{figure}

\begin{figure}[ht]
	\begin{center}
		\includegraphics[width=0.4\textwidth]{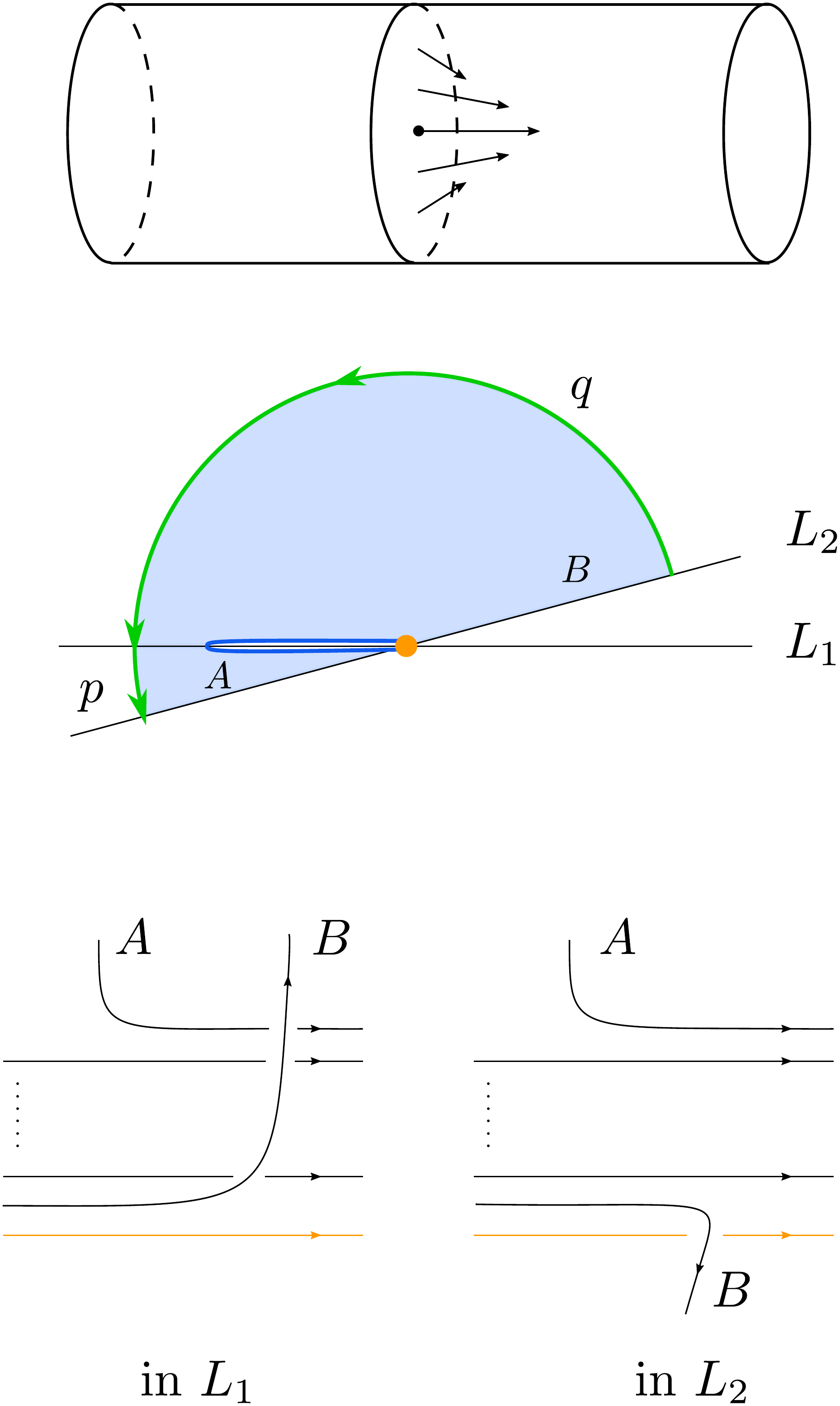}
		\caption{Top: the vector field corresponding to the difference of the two Lagrangians. Middle: the projection of the holomorphic disk with flow line attached to a slice for $m=0$ near the limit. Bottom: the boundary of the curve corresponding to a disk with flow line attached.}
		\label{fig:annulus2}
	\end{center}
\end{figure}
	
We denote by $\sfC_m$ the closed curve in the solid torus in homology class $m$ which has exactly $m-1$ positive crossings. Proposition~\ref{prp:flowdisks} then leads to the following infinite recursion. 
\begin{thm}\label{thm: rec basic annulus}
The skein operator
\be\label{eq:relrecursionAnnulus}
	{\sfP'}_{0,0}^{(1)}-{\sfP'}_{1,0}^{(1)} + (q-q^{-1}) \sum_{k>0} r\left(\sfC'_{k},\cdot\right)\otimes \sfC_{k}
\ee
annihilates the annulus skein valued partition function $\sfPsi_{\mathrm{an}}^{(0, 0)} $ where
\be\label{eq:psi-ann}
	\sfPsi_{\mathrm{an}}^{(f_1, f_2)} := \sum_{\lambda}  
	\left[(-1)^{|\lambda|} q^{\kappa(\lambda)}\right]^{f_1+f_2} \
	\sfW^{(1)}_{\lambda,\emptyset}\otimes \sfW^{(2)}_{\lambda,\emptyset}
\ee 
This is the only solution with initial condition $0$ in all negative classes.
\end{thm} 	

\begin{proof}
The fact that the skein valued partition function is annihilated follows from the usual boundary of moduli space argument once we observe that the disks in Proposition~\ref{prp:flowdisks} lie outside the central annulus. Uniqueness follows by expanding in the $\sfW_{\lambda,\overline{\mu}}$ basis which is an eigenbasis for $\sfP_{0,0}-\sfP_{1,0}$. 	

See also \cite{2024arXiv240110730N} for a combinatorial proof of this relation.
\end{proof}	

\begin{rmk}\label{r: formal fillings}
We point out the following general application of the worldsheet skein $D$-modules that generalizes well-known uses of augmentation varieties. Consider a Lagrangian $L$ with negative end along a Legendrian $\Lambda$. Then we can treat $L$ like a closed Lagrangian when it comes to holomorphic disk counting by counting disks with negative asymptotics at Reeb chords of $\Lambda$ and then applying the augmentation at the negative ends. Similarly, we can use skein recursions to fill in negative ends of Lagrangians. We use this construction to define a skein module for two $3$-dimensional Lagrangians that intersect cleanly along a circle. In the notation above for unshifted $S^{1}\times\R^{2}$ and $S^{1}\times P$ we define $\Sk(S^{1}\times\R^{2}\cup S^{1}\times P)$ as generated by isotopy classes of curves that are allowed to switch Lagrangian along the intersection and by imposing the relative recursion relation for an annulus at crossings. A more detailed study of this immersed skein module will appear elsewhere.  
\end{rmk}

\subsection{Basic holomorphic three holed spheres}\label{ssec:threeholedsphere}
Similarly to basic holomorphic disks and annuli, there exist basic holomorphic three holed spheres. From the corresponding formulas for disks and annuli, we conjecture that the local contribution of a basic holomorphic three holed sphere to the Gromov-Witten partition function is given by 
\be\label{eq:definition of Z_{(0,3)}}
	\sfPsi_{(0,3)} \coloneqq 1 + \sum_{\lambda \vdash n \in \IN} (-1)^n \prod_{\ydiagram{1} \in \lambda} q^{c(\ydiagram{1})}\left(q^{h(\ydiagram{1})} - q^{-h(\ydiagram{1})}\right) \sfW_{\lambda,\emptyset} \otimes \sfW_{\lambda,\emptyset} \otimes \sfW_{\lambda,\emptyset} \ \in \ \widehat{\Sk}(S^{1}\times\R^{2})^{\otimes 3}
\ee
for some choice of almost complex neighborhood of the underlying simple holomorphic curve.

Theorem \ref{thm:alternative quiver formula for Hopf} gives an expression for the HOMFLYPT partition function of the Hopf link in terms of the partition functions of a three holed sphere with one boundary component mapped to the unknot in $S^3$ and four basic disks. For this, recall that the map $\mathcal{U}\colon \textnormal{Sk}_+(\R^2 \times S^1) \to \textnormal{Sk}(S^3) \approx R$ defined by mapping the solid torus onto an unframed unknot is given on the basis elements $\sfW_{\lambda,\emptyset}$ by 
\be
\mathcal{U}\left( \sfW_{\lambda,\emptyset}\right) = \left[\prod_{\ydiagram{1}\in\lambda} \frac{a q^{c(\ydiagram{1})}-a^{-1} q^{-c(\ydiagram{1})}}{q^{h(\ydiagram{1})}-q^{-h(\ydiagram{1})}}\right] 
\ee

It follows that the partition function of a three holed sphere with one boundary component mapped to the unknot in $S^3$ is given by
\begin{align*}
\widetilde{\sfPsi}_{(0,3)} \ &= \ \left(1 \otimes 1 \otimes \mathcal{U}\right)\left(\sfPsi_{(0,3)}\right) \\
&= \ 1 + \sum_{\lambda \vdash n \in \IN} (-1)^n \prod_{\ydiagram{1} \in \lambda} \left(a q^{2c(\ydiagram{1})}-a^{-1}\right) \sfW_{\lambda,\emptyset} \otimes \sfW_{\lambda,\emptyset} \ \in \ \widehat{\Sk}(S^{1}\times\R^{2})^{\otimes 2}.
\end{align*}

The partition function $\widetilde{\sfPsi}_{(0,3)}$ satisfies a recursion relation. We state it for a framed version of $\widetilde{\sfPsi}_{(0,3)}$ (framing $(-1,-1)$) with shifted $a$ powers to make it more directly applicable in the proof of Theorem \ref{thm:alternative quiver formula for Hopf} below.

\begin{prp}\label{prop:relation for Z_{(0,3)}}
	The partition function $\sfPsi_{(0,3)}$ of a three holed sphere satisfies the relations 
	\begin{equation}\label{eq:relation for Z_{(0,3)}),first identity}
		\left( \left(- \sfP_{-1,0}^{(1)} + \sfP_{0,0}^{(1)} \right) \otimes a_{L_2} - a_{L_1} \otimes \left( \sfP_{0,0}^{(2)} - \sfP_{-1,0}^{(2)}\right) \right) \left(\widetilde{\sfPsi}_{(0,3)}^{(-1,-1)}[a]\right) \ = \ 0
	\end{equation}
	and
	\begin{align}\label{eq:relation for Z_{(0,3)}),second identity}
		&\left( \left(a_{L_1} \sfP_{1,-1}^{(1)} - \sfP_{0,-1}^{(1)} + a_{L_1} \left(\sfP_{0,0}^{(1)} - \sfP_{-1,0}^{(1)} \right) \right) \otimes 1 \right) \left(\widetilde{\sfPsi}_{(0,3)}^{(-1,-1)}[a]\right)\\\notag
		&= \left(1 \otimes \left(a_{L_2} \left(\sfP_{0,0}^{(2)} - \sfP_{-1,0}^{(2)}\right) + a_{L_2} \left(a^{2}+a_{L_1}^2\right) \sfP_{-1,1}^{(2)} - a_{L_2}^{2} \sfP_{-2,1}^{(2)} - a^2a_{L_1}^2 \sfP_{0,1}^{(2)} \right) \right) \left(\widetilde{\sfPsi}_{(0,3)}^{(-1,-1)}[a]\right)
	\end{align}
\end{prp}

\begin{proof}
	The partition function $\widetilde{\sfPsi}_{(0,3)}^{(-1,-1)}[a]$ is given by
	\begin{equation*}
		\widetilde{\sfPsi}_{(0,3)}^{(-1,-1)}[a] \coloneqq 1 + \sum_{\lambda \vdash n \in \IN} (-1)^n \prod_{\ydiagram{1} \in \lambda} \left(a^2 q^{-2c(\ydiagram{1})} - q^{-4c(\ydiagram{1})}\right) \sfW_{\lambda,\emptyset} \otimes \sfW_{\lambda,\emptyset}.
	\end{equation*}
	It is immediate that the operator 
	\begin{equation*}
		a_{L_1} \left(- \sfP_{-1,0}^{(1)} + \sfP_{0,0}^{(1)} \right) \otimes 1 - 1 \otimes a_{L_2} \left( \sfP_{0,0}^{(2)} - \sfP_{-1,0}^{(2)}\right)
	\end{equation*}
	annihilates $\widetilde{\sfPsi}_{(0,3)}^{(-1,-1)}[a]$ since $a_L(- \sfP_{-1,0} + \sfP_{0,0})$ has the basis elements $\sfW_{\lambda,\emptyset}$ as eigenvectors with corresponding eigenvalues which are independent of $a_L$. This implies the relation \eqref{eq:relation for Z_{(0,3)}),first identity}. It remains to prove that 
	\begin{equation*}
		\begin{gathered}
			\left( \left( a_{L_1} \sfP_{1,-1}^{(1)} - \sfP_{0,-1}^{(1)}\right) \otimes 1 \right) \, \widetilde{\sfPsi}_{(0,3)}^{(-1,-1)}[a] \\
			= 
			\left(  \left(a^{2}+a_{L_1}^2\right) \otimes a_{L_2} \sfP_{-1,1}^{(2)} - 1\otimes a_{L_2}^{2} \sfP_{-2,1}^{(2)} - a^2a_{L_1}^2  \otimes \sfP_{0,1}^{(2)} \right) \, \widetilde{\sfPsi}_{(0,3)}^{(-1,-1)}[a].
		\end{gathered}
	\end{equation*}
	To show this, recall that
	\begin{equation*}
		\left(\left(a_{L_1} \sfP_{1,-1}^{(1)} - \sfP_{0,-1}^{(1)}\right)\otimes 1\right) \, \sfW_{\lambda,\emptyset}\otimes \sfW_{\emptyset,\emptyset} = \sum_{\lambda \in \mu + \ydiagram{1}} \left(a_{L_1}^2 q^{2 c(\ydiagram{1})} - 1 \right) \sfW_{\mu,\emptyset}\otimes \sfW_{\emptyset,\emptyset}
	\end{equation*}
	and that 
	\begin{align*}
		\Bigl( \left(a^{2}+a_{L_1}^2\right) \otimes a_{L_2} \sfP_{-1,1}^{(2)} &- 1\otimes a_{L_2}^{2} \sfP_{-2,1}^{(2)} - a^2a_{L_1}^2 \otimes \sfP_{0,1}^{(2)} \Bigr) \, \sfW_{\emptyset,\emptyset}\otimes  \sfW_{\mu,\emptyset} \\
		&= \sum_{\lambda \in \mu + \ydiagram{1}} \left(\left(a^2 + a_{L_1}^2\right) q^{-2c(\ydiagram{1})} - q^{-4c(\ydiagram{1})} - a^2 a_{L_1}^2 \right) \sfW_{\emptyset,\emptyset}\otimes \sfW_{\lambda,\emptyset}.
	\end{align*}
	We compute
	\begin{align*}
			&\left( \left(  a_{L_1} \sfP_{1,-1}^{(1)} - \sfP_{0,-1}^{(1)}\right) \otimes 1 \right)\left(\widetilde{\sfPsi}_{(0,3)}^{(-1,-1)}[a]\right)\\
			&= \sum_{\mu \vdash n \in \IN_0} (-1)^n \prod_{\ydiagram{1} \in \mu}\left(a^2 q^{-2c(\ydiagram{1})} - q^{-4c(\ydiagram{1})}\right)\\
			 & \qquad\qquad  \cdot\left(- \sum_{\lambda \in \mu + \ydiagram{1}} \left(a^2 q^{-2c(\ydiagram{1})} - q^{-4c(\ydiagram{1})} \right) \left(a_{L_1}^2 q^{2 c(\ydiagram{1})} - 1 \right)  \right) \sfW_{\mu,\emptyset} \otimes \sfW_{\lambda,\emptyset}\\
			&=  \sum_{\mu \vdash n \in \IN_0} (-1)^n \prod_{\ydiagram{1} \in \mu}\left(a^2 q^{-2c(\ydiagram{1})} - q^{-4c(\ydiagram{1})}\right) \\
			& \qquad\qquad \cdot \left(-\sum_{\lambda \in \mu + \ydiagram{1}} \left( a^2 a_{L_1}^2 - \left( a^2 + a_{L_1}^2 \right) q^{-2c(\ydiagram{1})} +  q^{-4c(\ydiagram{1})}\right) \right) \sfW_{\mu,\emptyset} \otimes \sfW_{\lambda,\emptyset}\\
			&= \left( \left(a^2 + a_{L_1}^2\right)\otimes a_{L_2} \sfP_{-1,1}^{(2)} - 1\otimes a_{L_2}^2 \sfP_{-2,1}^{(2)} - a^2 a_{L_1}^2\otimes \sfP^{(2)}_{0,1} \right)\left(\widetilde{\sfPsi}_{(0,3)}^{(-1,-1)}[a]\right).
	\end{align*}
	The proposition follows.
\end{proof}

\section{Skein valued recursion for the Hopf link}\label{sec:hopf}
In this section we derive the skein recursion for the Hopf link from knowledge of holomorphic curves with Reeb chord asymptotics at its ideal contact boundary. 
\begin{figure}[h!]
	\begin{center}
		\includegraphics[width=0.75\textwidth]{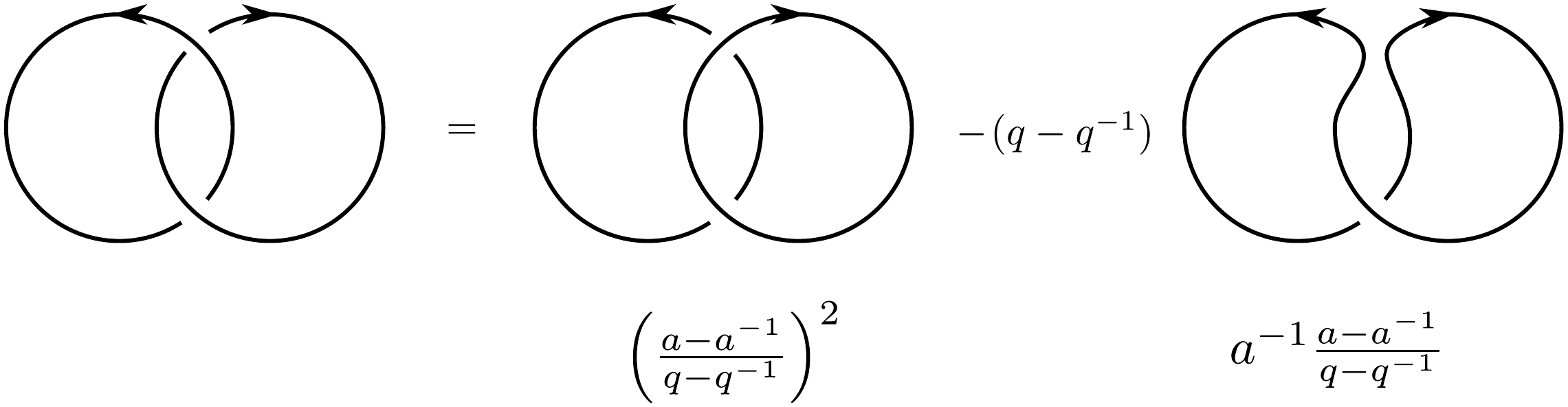}
		\caption{The Hopf link with negative crossings.}
		\label{fig:negative-hopf}
	\end{center}
\end{figure}
The main new ingredient compared to earlier examples are the appearance of degree zero Reeb chords. When there are no such chords recursion relations can be read off directly from the moduli spaces of holomorphic curves with asymptotics at degree one Reeb chords. Here such moduli spaces have broken curves in their boundary and the recursion relation appears when the inside pieces of such curves are eliminated algebraically. We describe the curves in Section~\ref{ssec:curves on Leg Hopf}, carry out the elimination theory and prove Theorem~\ref{t:Hopf} in Section~\ref{ssec:proofThmHopf}. In Section~\ref{ssec:auxthmHopf} we discuss auxiliary aspects of the recursion relation.

\subsection{Holomorphic curves on the Legendrian conormal of the Hopf link}\label{ssec:curves on Leg Hopf}
As for any knot or link, holomorphic curves on the Legendrian conormal can be computed in the limit when the conormal approaches a multiple of the unknot conormal where the holomorphic curves on the conormal correspond to holomorphic curves on the unknot conormal with flow trees attached. Here the flow trees are described by the link conormal front over the unknot conormal.

For the Hopf link conormal Legendrian, the Morse flow, the unknot disks and resulting holomorphic curves were determined in~\cite{Ekholm:2018iso}. We use a more symmetric Morse flow here that results from isotoping the Legendrian. The resulting Morse flow is shown in Figure~\ref{fig:Hopf-Legendrian}. To see the isotopy, recall that the Hopf link  results from a perturbation of a Bott-Morse Legendrian where the Bott manifold of Reeb chords are the stable manifolds of $a_{ij}$. The isotopy that takes us from~\cite{Ekholm:2018iso} to Figure~\ref{fig:Hopf-Legendrian} simply slides $a_{21}$ and $b_{21}$ along its Bott-manifold.

\begin{figure}[h!]
\begin{center}
\includegraphics[width=0.95\textwidth]{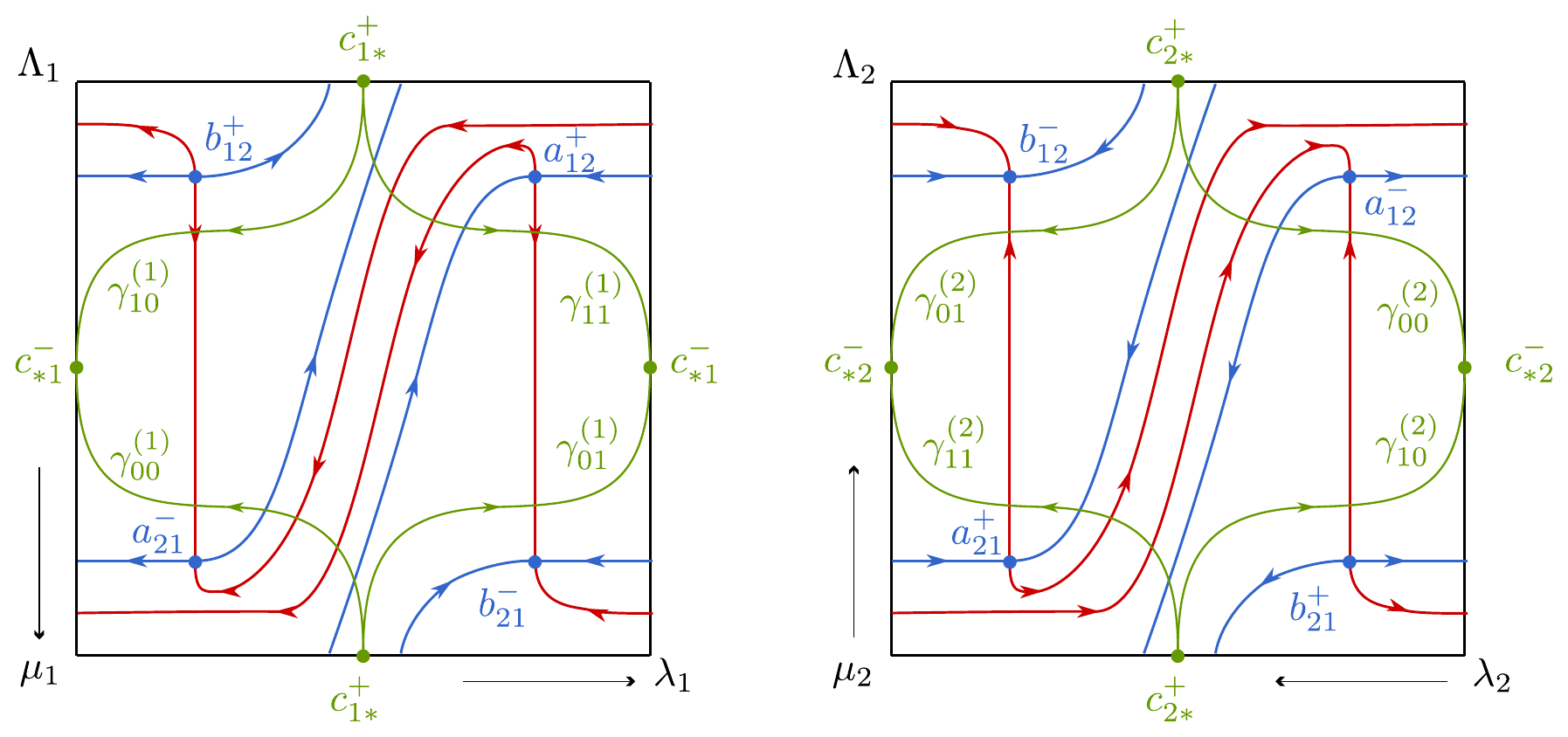}
\caption{Legendrian tori of the Hopf link conormals. The oriented basis $(\lambda_j,\mu_j)$ are induced by the natural orientation on the complement filling.}
\label{fig:Hopf-Legendrian}
\end{center}
\end{figure}

As mentioned above, we follow the SFT-terminology of~\cite{Ekholm:2018iso}, and call the curve count at infinity the Hamiltonian. We write $\sfH^{\textnormal{ws}}(c,a_{1},\dots,a_{m})$ for the worldsheet skein valued count of holomorphic curves with positive punctures at the degree one chord $c$ and degree 0 chords $a_{1},\dots,a_{m}$. We compute relevant parts of the Hamiltonian directly from Figure~\ref{fig:Hopf-Legendrian}.

\subsubsection{Reeb chords of the Hopf-link Legendrian}
The Reeb chords are
\be
c_{11},\;c_{12},\; c_{21},\; c_{22}, \; b_{12}, \; b_{21}
\ee
of degree $1$ and
\be
a_{12},\; a_{21}
\ee
of degree $0$.

\subsubsection{Curves with positive puncture at $c_{11}$}\label{ssec:c11}
There are the following five curves with positive puncture at $c_{11}$, 
\be\label{eq:Hc11}
\sfH^{\textnormal{ws}}(c_{11}) = 
\ \picLabel{{$\text{D}_{00}$}{70}{130}} {Hc11_00} {.350}
\ -\  \picLabel{{$\text{D}_{10}$}{70}{130}} {Hc11_10} {.350}
\ -\  \picLabel{{$\text{D}_{01}$}{70}{130}} {Hc11_01} {.350}
\ +\  \picLabel{{$\text{D}_{11}$}{70}{130}} {Hc11_11} {.350}
\ +\  \picLabel{{$\text{D}_{00}$}{130}{120}} {Hc11_a12_a21_new} {.350}
\ee
where the curves $\alpha_1,\alpha_2,\alpha_3$ arise from attaching Morse flow lines to $\text{D}_{00}$ as shown in Figure~\ref{fig:alpha123}. Here, $\text{D}_{i,j}$ is the disk that has boundary $\gamma_{i,j}^{(1)}$ on $\Lambda_1$ and thus boundary $\gamma_{1-i,1-j}^{(2)}$ on $\Lambda_2$.
The signs in this and similar formulas below are determined by signs of the big disks and intersection signs of (oriented) boundaries of big disks with flow lines attached to them as in \cite{ekholm2013knot}.
\begin{figure}[h!]
\begin{center}
\includegraphics[width=0.75\textwidth]{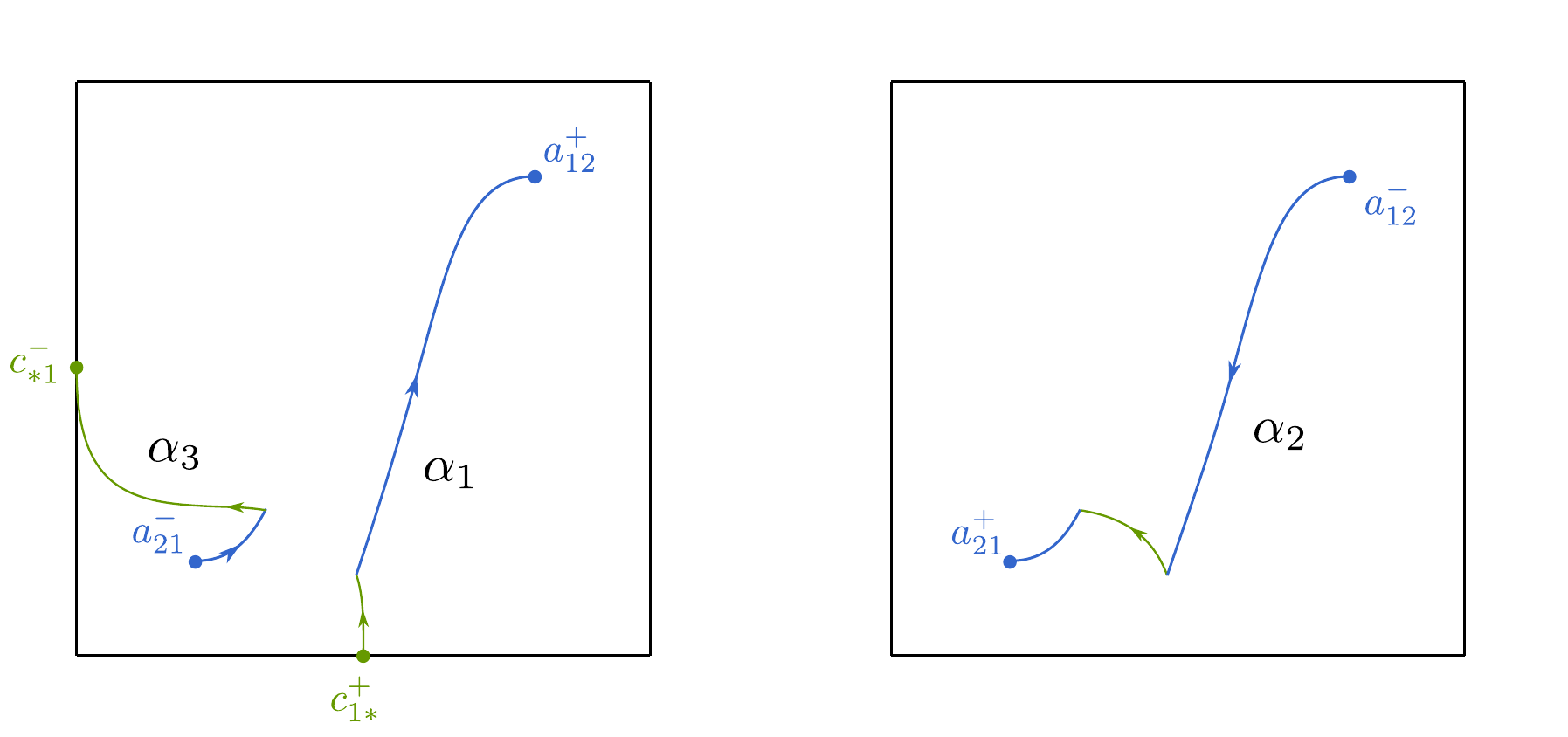}
\caption{Boundary components of the last disk in~\eqref{eq:Hc11}}
\label{fig:alpha123}
\end{center}
\end{figure}

\subsubsection{Curves with positive puncture at $c_{22}$}\label{ssec:c22}
There are the following five curves with positive puncture at $c_{22}$ (symmetric to $c_{11}$).
\be\label{eq:Hc22}
\sfH^{\textnormal{ws}}(c_{22}) = 
\ \picLabel{{$\text{D}_{11}$}{70}{130}} {Hc22_00} {.350}
\ -\  \picLabel{{$\text{D}_{01}$}{70}{130}} {Hc22_10} {.350}
\ -\  \picLabel{{$\text{D}_{10}$}{70}{130}} {Hc22_01} {.350}
\ +\  \picLabel{{$\text{D}_{00}$}{70}{130}} {Hc22_11} {.350}
\ +\  \picLabel{{$\text{D}_{11}$}{130}{120}} {Hc22_a21_a12_new} {.350}
\ee
where the curves $\beta_1,\beta_2,\beta_3$ are shown in Figure~\ref{fig:beta123}.

\begin{figure}[h!]
\begin{center}
\includegraphics[width=0.75\textwidth]{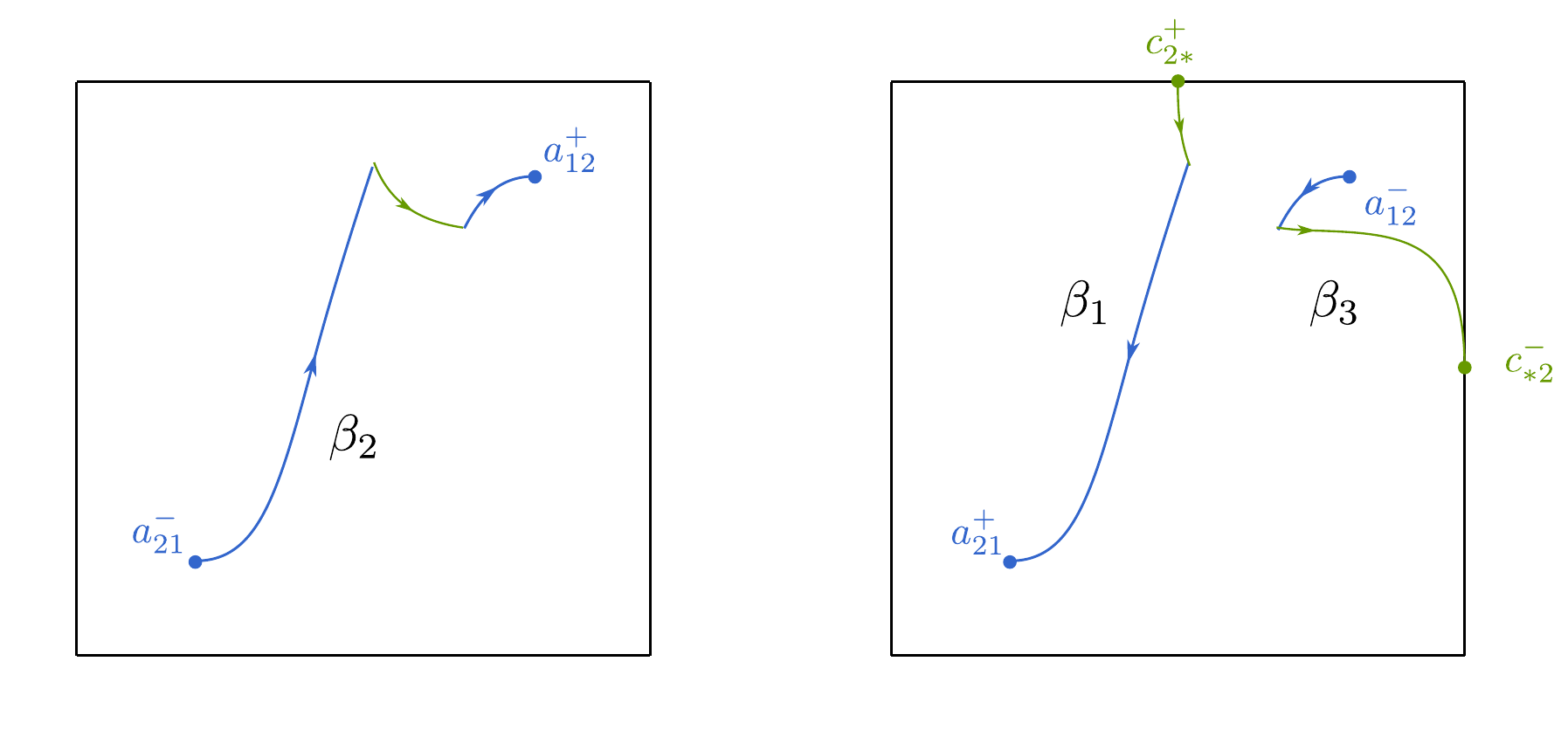}
\caption{Boundary components of the last disk in~\eqref{eq:Hc22}}
\label{fig:beta123}
\end{center}
\end{figure}

\subsubsection{Curves with positive punctures at $c_{12}$ and $a_{21}$}\label{ssec:c12a21}
There are the following four curves with positive punctures at $c_{12}$ and $a_{21}$.  
\be\label{eq:Hc12}
\begin{split}
	\sfH^{\textnormal{ws}}(c_{12},a_{21}) 
	& = 
	\ \picLabel{{$\text{D}_{00}$}{105}{115} {$\text{T}_{a_{21}}$}{305}{100}} {Hc12_a12_new} {.350}
	\qquad-\ \  \picLabel{{$\text{D}_{11}$}{105}{115} {$\text{T}_{a_{21}}$}{305}{100}} {Hc12_a12_bis_new} {.350}
	\\
	& -\  \picLabel{{$\text{D}_{11}$}{70}{100}} {Hc12_a21_new} {.350}
	\ -\  \picLabel{{$\text{D}_{10}$}{70}{100}}{Hc12_a21_bis_new} {.350}
\end{split}
\ee
where the curves arise from attaching Morse flow lines to the respective disks as shown in Figure~\ref{fig:c12-boundaries}, and $\text{T}_{a_{21}}$ denotes the trivial strip over the Reeb chord $a_{21}$. 

\begin{figure}
     \centering
     \begin{subfigure}[b]{0.49\textwidth}
         \centering
         \includegraphics[width=\textwidth]{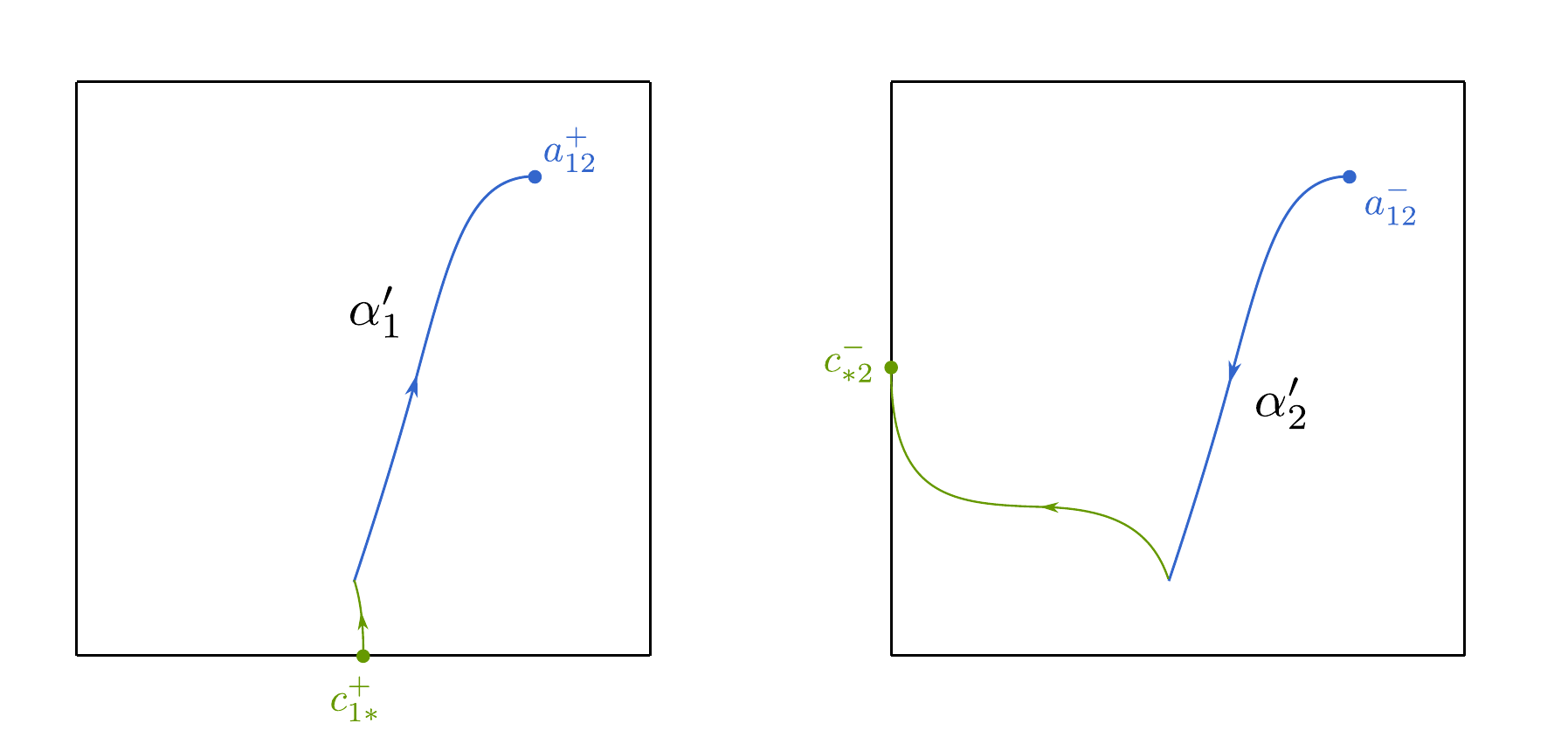}
         \caption{}
         \label{fig:}
     \end{subfigure}
     \hfill
     \begin{subfigure}[b]{0.49\textwidth}
         \centering
         \includegraphics[width=\textwidth]{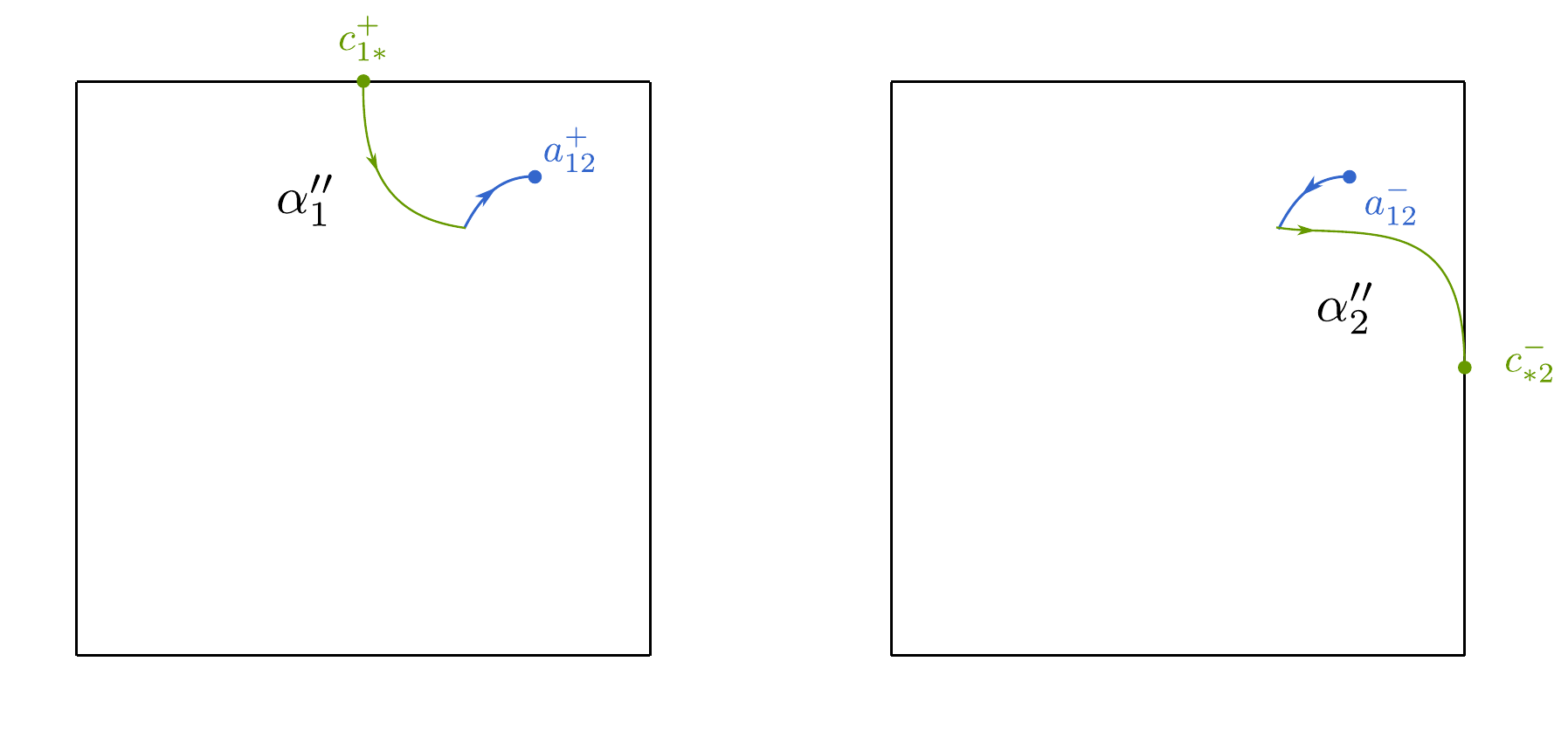}
         \caption{}
         \label{fig:}
     \end{subfigure}
     %%%%
     \begin{subfigure}[b]{0.49\textwidth}
         \centering
         \includegraphics[width=\textwidth]{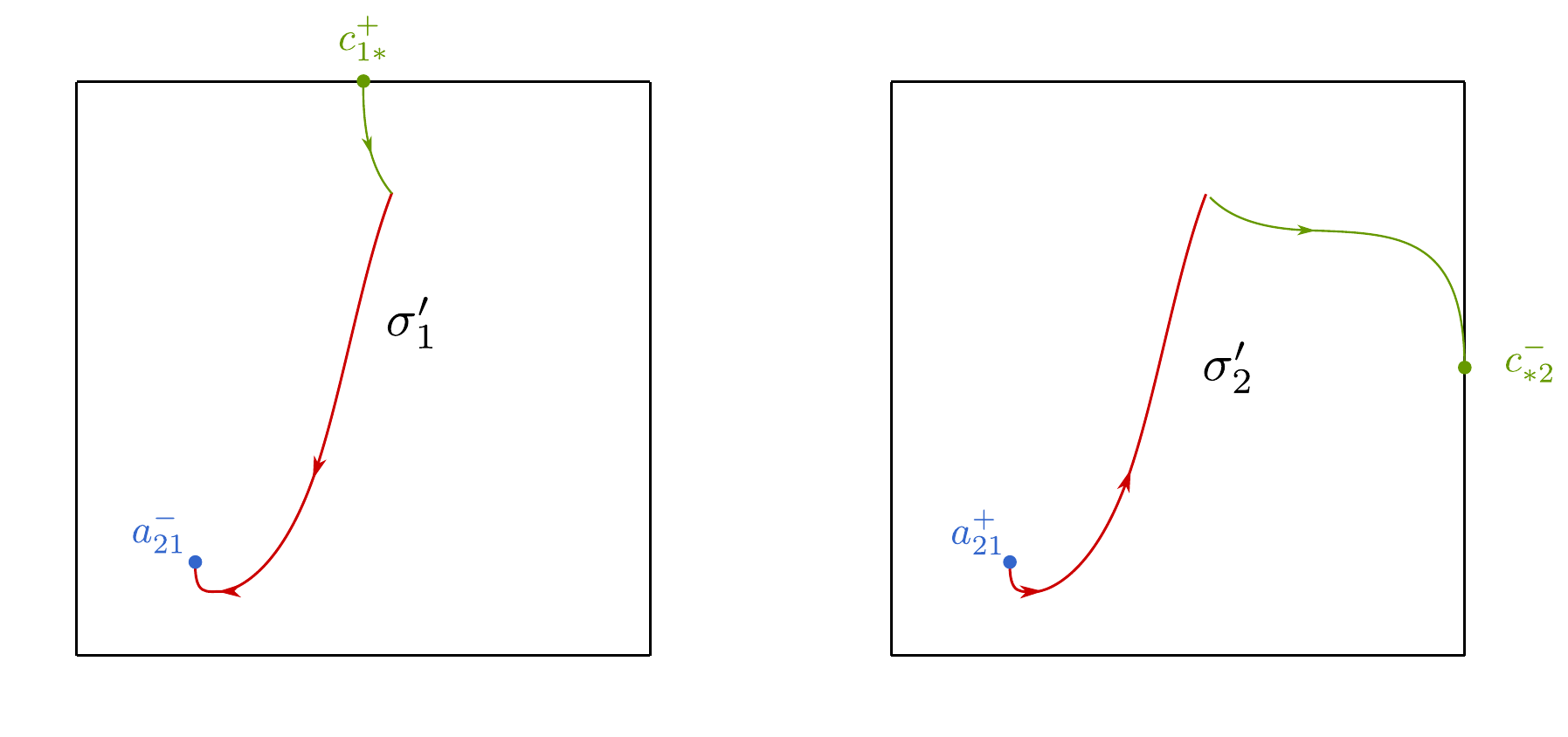}
         \caption{}
         \label{fig:}
     \end{subfigure}
     \hfill
     \begin{subfigure}[b]{0.49\textwidth}
         \centering
         \includegraphics[width=\textwidth]{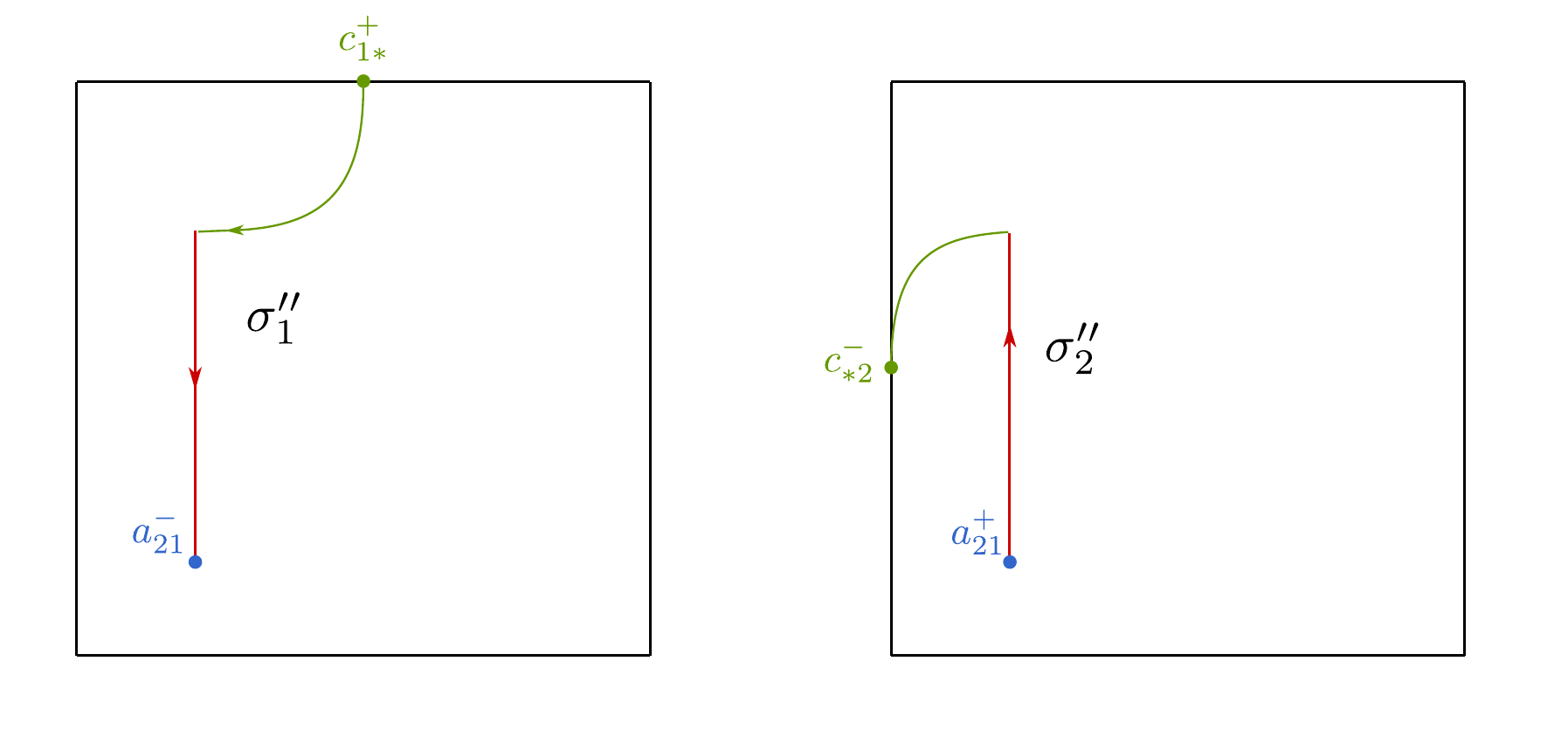}
         \caption{}
         \label{fig:}
     \end{subfigure}
        \caption{Boundaries of curves in~\eqref{eq:Hc12}. Trivial strips from $a_{21}$ to itself are omitted.}
        \label{fig:c12-boundaries}
\end{figure}

\subsubsection{Curves with positive punctures at $c_{21}$ and $a_{12}$}\label{ssec:c21a12}
There are the following four curves with positive punctures at $c_{21}$ and $a_{12}$ (symmetric to $c_{12}$ and $a_{21}$). 
\be\label{eq:Hc21}
\begin{split}
	\sfH^{\textnormal{ws}}(c_{21},a_{12}) & = 
	\ \picLabel{{$\text{D}_{11}$}{105}{115} {$\text{T}_{a_{12}}$}{305}{100}} {Hc21_a21_new} {.350}
	\qquad +\ \ \picLabel{{$\text{D}_{00}$}{105}{115} {$\text{T}_{a_{12}}$}{305}{100}} {Hc21_a21_bis_new} {.350}
	\\
	& +\  \picLabel{{$\text{D}_{00}$}{70}{100}}{Hc21_a12_new}{.350}
	\ +\  \picLabel{{$\text{D}_{01}$}{70}{100}}{Hc21_a12_bis_new}{.350}
\end{split}
\ee
where the curves arise from detours along Morse flow lines to the respective disks as shown in Figure~\ref{fig:c21-boundaries},

\begin{figure}
     \centering
     \begin{subfigure}[b]{0.49\textwidth}
         \centering
         \includegraphics[width=\textwidth]{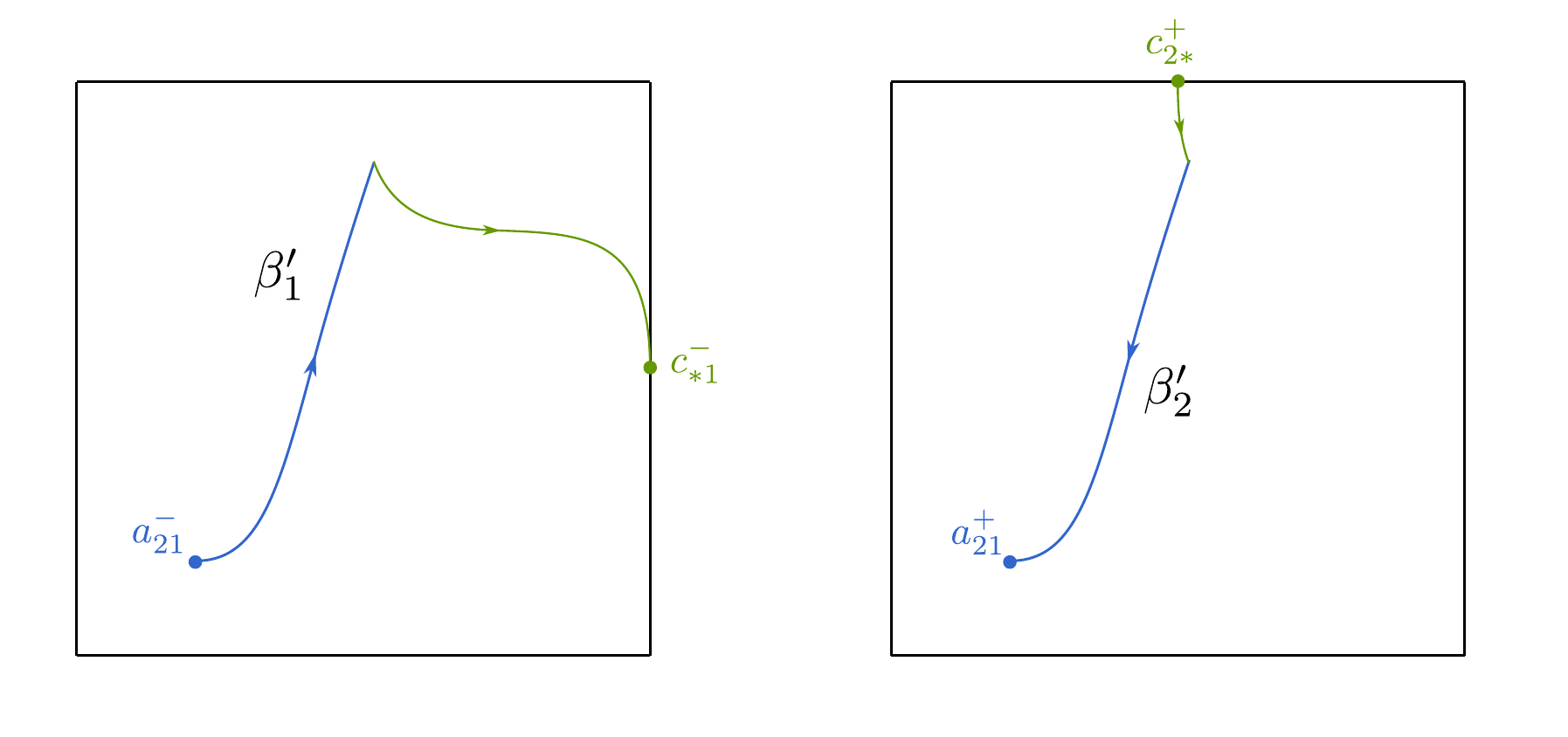}
         \caption{}
         \label{fig:}
     \end{subfigure}
     \hfill
     \begin{subfigure}[b]{0.49\textwidth}
         \centering
         \includegraphics[width=\textwidth]{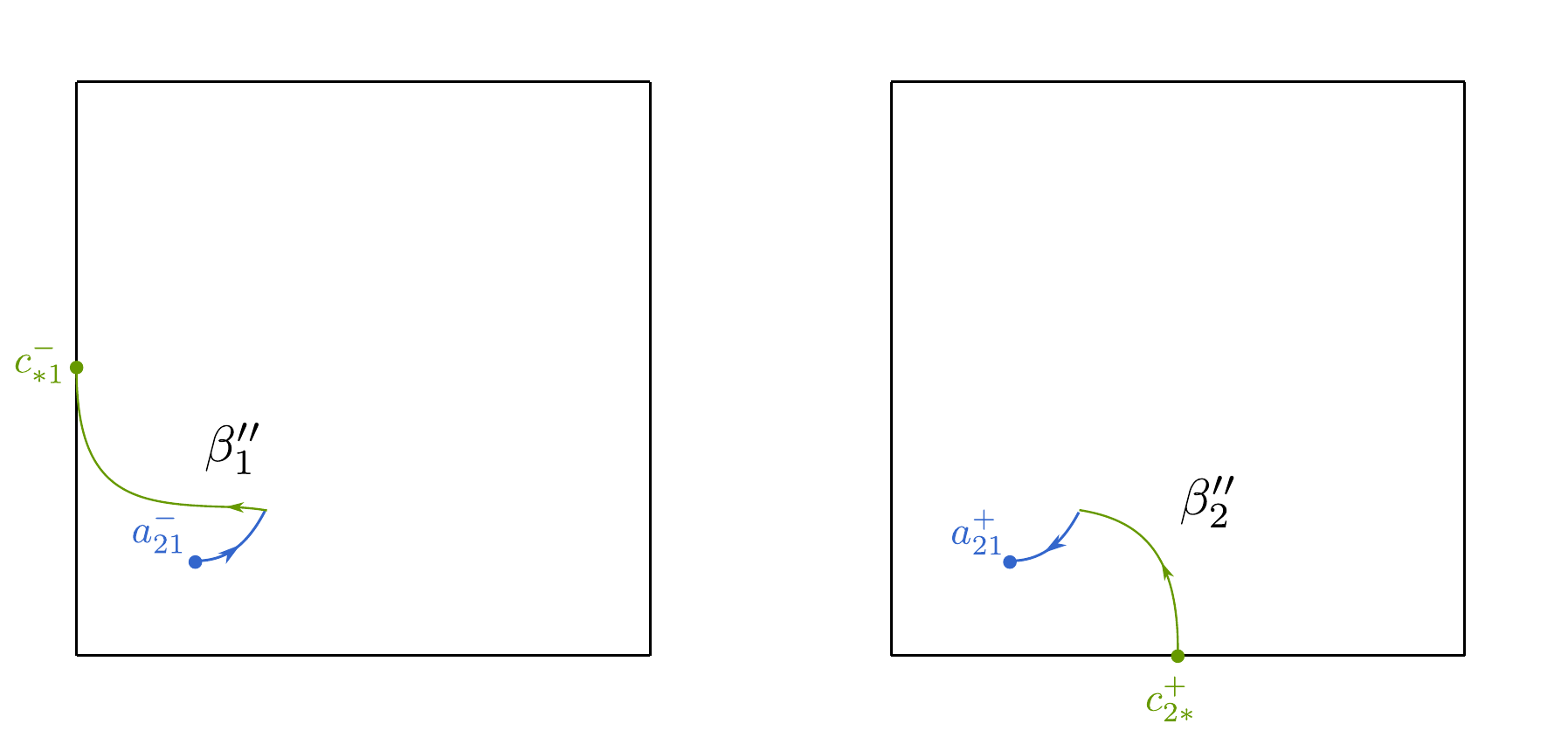}
         \caption{}
         \label{fig:}
     \end{subfigure}
     %%%%
     \begin{subfigure}[b]{0.49\textwidth}
         \centering
         \includegraphics[width=\textwidth]{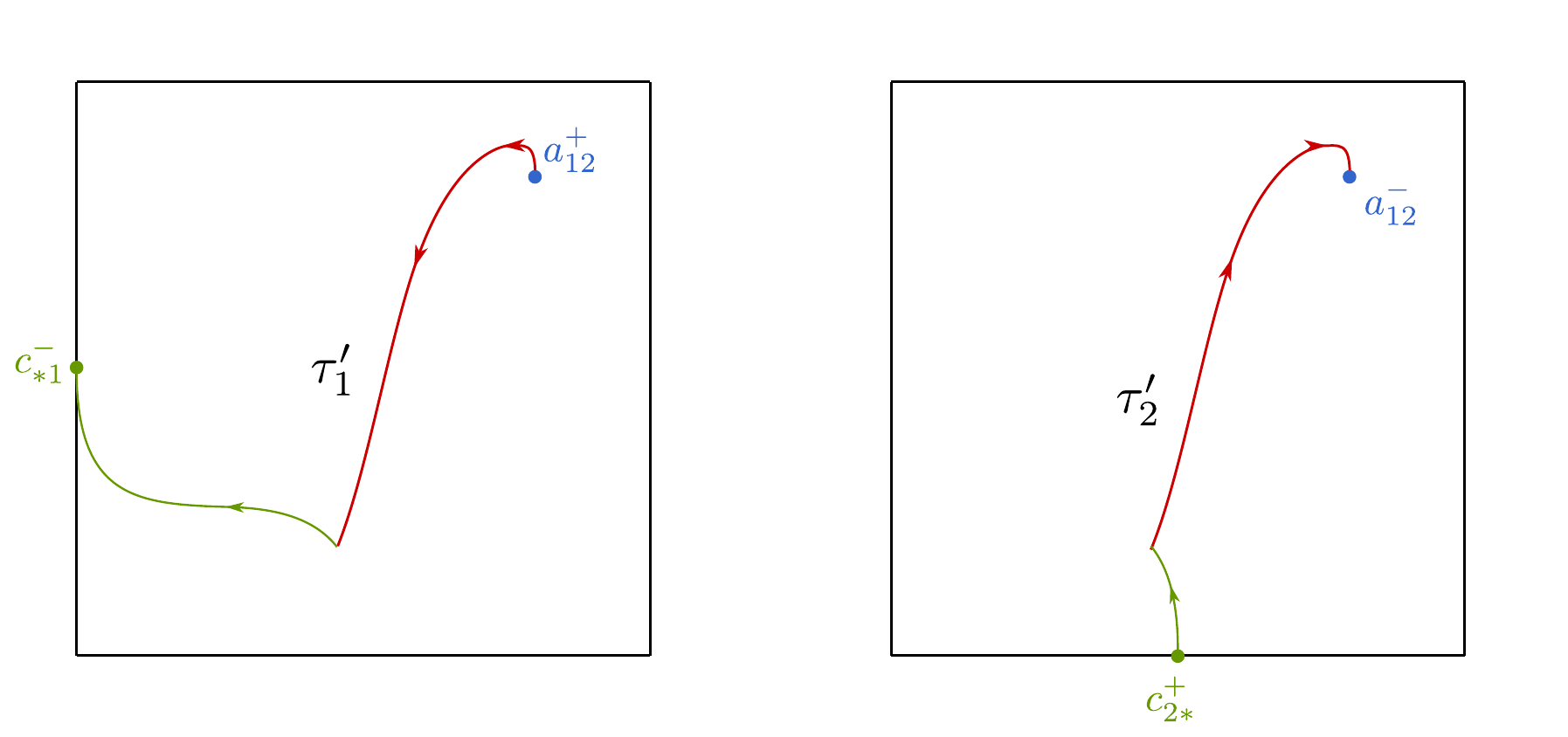}
         \caption{}
         \label{fig:}
     \end{subfigure}
     \hfill
     \begin{subfigure}[b]{0.49\textwidth}
         \centering
         \includegraphics[width=\textwidth]{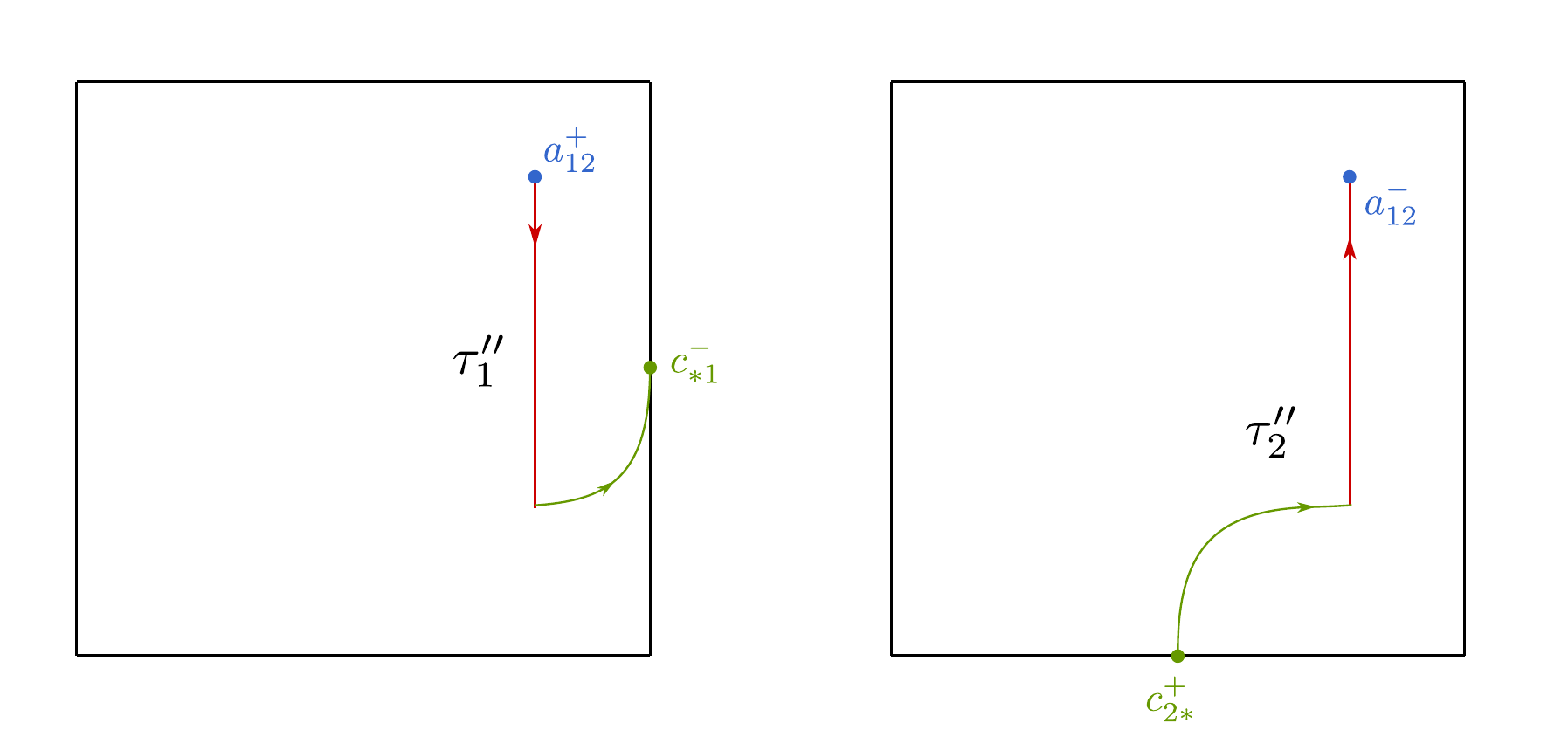}
         \caption{}
         \label{fig:}
     \end{subfigure}
        \caption{Boundaries of curves in~\eqref{eq:Hc21}. Trivial strips from $a_{12}$ to itself are omitted.}
        \label{fig:c21-boundaries}
\end{figure}

\subsection{Proof of Theorem~\ref{t:Hopf} -- worldsheet skein $D$-module generators}\label{ssec:proofThmHopf}
In this section we prove the first part of Theorem~\ref{t:Hopf}. We first show that the augmentation variety is generated by augmentations of Lagrangian fillings, then we derive the three recursion relations. We then show in Section \ref{ssec:specialization-and-uniqueness}   that these determine skein valued partition functions uniquely.

\subsubsection{Augmentations and Lagrangian fillings}
The augmentation variety of the Hopf link is cut out by the following three equations in $(\C^{\ast})^{4}$, see~\cite{Aganagic:2013jpa},
\be
\begin{split}\label{eq:classical-aug-var-factorized}
	A_1  \ &= \ (1 - y_1 - x_1 + a^2 x_1 y_1) - (1 - y_2- x_2 + a^2 x_2 y_2) \ = \ 0,\\
	A_2' \ &= \ (1 - y_1 - x_1 + a^2 x_1 y_1)(y_1- x_2) \ = \ 0, \\
	A_3' \ &= \ (1 - y_1 - x_1 + a^2 x_1 y_1)(x_1 - y_2) \ = \ 0.
\end{split}
\ee
There are two main branches in the augmentation variety, each of dimension two,
\be\label{eq:V-Hopf-branches}
\begin{split}
	V_{\Gamma}^{(1)} \ &=  \ \left\{x_1 = y_2 ,\quad x_2 = y_1\right\}, \\
	V_{\Gamma}^{(2)} \ &= \  \left\{1 - y_1 - x_1 + a^2 x_1 y_1 = 0, \quad 1 - y_2- x_2 + a^2 x_2 y_2=0\right\}. 
\end{split}
\ee
The branches intersect along a curve (a `diagonal unknot curve'):
\begin{equation} \label{eq:V-Hopf-branches-intersection}
	V_{\Gamma}^{(1)}\cap V_{\Gamma}^{(2)} \ = \ \left\{x_1 = y_2, \; x_2 = y_1, \; 1 - y_1 - x_1 + a^2 x_1 y_1 = 0\right\}. 
\end{equation}

Along $V_\Gamma^{(1)}$ augmentations are induced by the Lagrangian filling which is the complement of the Hopf link, with the topology $T^{2}\times\R$. Thinking of the Hopf link as two fibers of the Hopf fibration it is easy to see that the complement of the Hopf link is a Lagrangian that admits an exact deformation which is disjoint from the 0-section in $T^{\ast} S^{3}$. Exactness then implies that the Lagrangian does not support any holomorphic curve and the associated partition function is simply $1$ and the corresponding    
homomorphism
\[
\sfH_{\text{BRST};\Gamma} \to \Sk(T^{2}\times \R),
\]
then simply takes $1$ to $1$.

The branch $V_{\Gamma}^{(2)}$ is simply the product of the augmentation varieties of the unknot components of the Hopf link, $V^{(2)}=V_{U_1}\times V_{U_2}$ and has augmentations induced by combinations of the fillings of the unknot components.

Noting that the complement of an unknot component is Hamiltonian isotopic to the conormal of the other component we find that along $V_{\Gamma}^{(2)}$ we have skein homomorphisms
\[
\sfH_{\Gamma;\text{BRST}} \to \Sk(L_{1}\cup L_{2}),
\]
where $L_{j}\approx S^{1}\times\R^{2}$ and where $L_{1}$ and $L_{2}$ are either conormal or complement fillings of components  according to the asymptotics. Counting curves on these Lagrangian fillings gives wave functions. Below we obtain recursion relations by elimination in the skein and therefore the curve counting wave function satisfies the recursion relation by skein invariance and the description of the boundary of 1-dimensional moduli spaces. The first statement of the theorem follows.

\subsubsection{Skein valued elimination theory}
We next turn to the world sheet skein $D$-module generators. We use Proposition \ref{prp: chain map eqn} and elimination.

\subsubsection*{The operator $\sfA^{\textnormal{ws}}_1$ from $\sfH^{\textnormal{ws}}(c_{11})$ and $\sfH^{\textnormal{ws}}(c_{22})$}
We consider the Hamiltonians of Reeb chords $c_{11}, c_{22}$, see \eqref{eq:Hc11} --~\eqref{eq:Hc22}.
In each case there is only one disk with negative punctures, the last one which is a triangle.
Note that both have the \emph{same} negative punctures and that the corresponding terms in the boundary of the 1-dimensional moduli spaces that arise by gluing rigid curves from the inside with corresponding positive Reeb chord asymptotics can then be identified provided that we find capping surfaces for $c_{11}$ and $c_{22}$ so that the two triangles at infinity can be identified.

The desired identification is obtained by capping $c_{11}$ with a path $\gamma_1$ such that
\be
	\alpha_1\star\gamma_1\star\alpha_3 = \beta_2\,,
\ee
and capping $c_{22}$ with a path $\gamma_2$ such that
\be
	\beta_3\star\gamma_2\star\beta_1 = \alpha_2\,.
\ee
We do not draw the capping paths, but it is easy to see that they are as follows
\be
\begin{split}	
	\gamma_1 & \simeq \alpha_1^{-1}\star\beta_2\star \alpha_3^{-1} \simeq (\gamma^{(1)}_{00})^{-1}\,,
	\\
	\gamma_2 & \simeq \beta_3^{-1}\star\alpha_2\star \beta_1^{-1} \simeq (\gamma^{(2)}_{00})^{-1}\,.
	\\
\end{split}
\ee
If we then choose the capping surfaces
\be
\begin{split}
	\text{D}_1 \simeq (\text{D}^{(1)}_{00})^{-1}\,,
	\\	
	\text{D}_2 \simeq (\text{D}^{(2)}_{11})^{-1}\,,
\end{split}
\ee
where inversion denotes orientation reversal, we have that
\vspace{15pt}
\be
\picLabel{{$\text{D}_{00}$}{130}{110}
	{$\gamma_1$}{120}{220}
	{$\text{D}_1$}{125}{178}
	}{wsHc11_a12_a21_capped_new} {.350}  
\ = \ 
\picLabel{{$\text{D}_{11}$}{130}{110}
	{$\gamma_2$}{120}{220}
	{$\text{D}_2$}{125}{178}
	}{wsHc22_a21_a12_capped_new} {.350}  
\ .
\ee

Thus capping $\sfH^{\textnormal{ws}}(c_{ii})$ with $\text{D}_i$ and eliminating triangles by taking the difference of the corresponding operator equations we get from~\eqref{eq:Hc11} --~\eqref{eq:Hc22} 
\be\label{eq:skein-A1}
\begin{split}
	\sfA_1^{\textnormal{ws}} &= \sfH^{\textnormal{ws}}(c_{11};\gamma_1) - \sfH^{\textnormal{ws}}(c_{22};\gamma_2) \\
	& \simeq 
	\( \text{D}_{00}^{(1)} \)^{-1} \cdot \left(\text{D}_{00}^{(1)} - \text{D}_{10}^{(1)} - \text{D}_{01}^{(1)} + \text{D}_{11}^{(1)} \right)
	- 
	\(\text{D}_{11}^{(2)}\)^{-1} \cdot \left(\text{D}_{11}^{(2)} - \text{D}_{01}^{(2)} - \text{D}_{10}^{(2)} + \text{D}_{00}^{(2)} \right)\,.
\end{split}
\ee

\subsubsection*{The operator $\sfA^{\textnormal{ws}}_2$ from $\sfH^{\textnormal{ws}}(c_{12},a_{21})$, $\sfH^{\textnormal{ws}}(c_{11})$ and $\sfH^{\textnormal{ws}}(c_{22})$}\label{ssec:A2}

To obtain a recursion relation from $\sfH^{\textnormal{ws}}(c_{12},a_{21})$ we need to cancel the curves with negative punctures. These disks have negative punctures at $a_{12}$ and $a_{21}$ and can be canceled by the corresponding terms in $\sfH^{\textnormal{ws}}_{c_{11}}$ and $\sfH^{\textnormal{ws}}_{c_{22}}$, respectively. We use the following system of capping surfaces
\be\label{eq:Hc12-cappings-gamma12p}
	\picLabel{{$\text{D}_3$}{130}{85}} {wsgamma12_capping_new} {.350}  
\ee
where $\gamma_1'$ runs from $a_{21}^-$ to $c_{1*}^+$, while $\gamma_2'$ runs from $c_{2*}^-$ to $a_{21}^+$.
If these capping paths for $\sfH^{\textnormal{ws}}(c_{12},a_{21})$ satisfy the following
\be
\begin{split}
	\alpha_1'\star\gamma_1'  &= \alpha_1\star\delta_1'\star\alpha_3 \\
	\gamma_2'\star\alpha_2'  &= \alpha_2
\end{split}	
\qquad
\begin{split}
	\alpha_1''\star\gamma_1'  &= \beta_2\\
	\gamma_2'\star\alpha_2''  &= \beta_1\star\delta_2'\star\beta_3 
\end{split}	
\ee
where $\delta_{1}',\delta_2'$ are capping paths for $c_{11}$ and $c_{22}$, respectively, then
we get a system of four equations in four variables which admits the following solution
\be
\begin{split}
	\gamma_1' &= (\alpha_1'')^{-1}\star\beta_2,\\
	\delta_1' &= (\alpha_1)^{-1}\star \alpha_1'\star(\alpha_1'')^{-1}\star\beta_2 \star(\alpha_3)^{-1}\\
	& \simeq (\gamma^{(1)}_{10})^{-1},
\end{split}
\qquad
\begin{split}
	\gamma_2' &= \alpha_2\star(\alpha_2')^{-1},\\
	\delta_2' &= (\beta_1)^{-1} \star \alpha_2\star(\alpha_2')^{-1} \star \alpha_2'' \star \beta_3^{-1}\\
	& \simeq (\gamma^{(2)}_{01})^{-1}.
\end{split}
\ee
From this we see that we can choose $\text{D}_3$ as $(\text{D}_{10})^{-1}$ with a flow line to $\text{a}_{21}$ attached and 
\be
\text{D}_4 \simeq (\text{D}_{10}^{(1)})^{-1} \quad \text{D}_5 \simeq (\text{D}_{10}^{(2)})^{-1}.
\ee
With capping surfaces fixed in this way we have the following identities
\be
	\picLabel{{$\text{D}_{00}$}{105}{115} {$\text{T}_{a_{21}}$}{280}{100}
	{$\text{D}_3$}{165}{210}} {wsHc12_a12_capped_new} {.350}  
	\quad = \ 
	\picLabel{{$\text{D}_{00}$}{130}{110}
		{$\delta'_1$}{120}{230}
		{$\text{D}_4$}{125}{178}
	}{wsHc11_a12_a21_capped_new} {.350}  
	\ ,
	\qquad
	\picLabel{{$\text{D}_{11}$}{105}{115} {$\text{T}_{a_{21}}$}{280}{100}
		{$\text{D}_3$}{165}{210}} {wsHc12_a12_bis_capped_new} {.350}  
	\quad = \ 
	\picLabel{{$\text{D}_{11}$}{130}{110}
		{$\delta'_2$}{120}{230}
		{$\text{D}_5$}{125}{178}
	}{wsHc22_a21_a12_capped_new} {.350}  
	\ .
\ee

This implies that we can eliminate all curves with negative punctures using linear combination of $\sfH^{\textnormal{ws}}(c_{12},a_{21})$, $\sfH^{\textnormal{ws}}(c_{11})$ and $\sfH^{\textnormal{ws}}(c_{22})$. We get 
\be\label{eq:skein-A2}
\begin{split}
	\sfA^{\textnormal{ws}}_2 &= \sfH^{\textnormal{ws}}(c_{12},a_{21};\text{D}_3) - \sfH^{\textnormal{ws}}(c_{11},\text{D}_4) + \sfH^{\textnormal{ws}}(c_{22},\text{D}_5) \\
	& \simeq -\( \text{D}_{10}^{(1)} \)^{-1} \cdot \left(\text{D}_{00}^{(1)} - \text{D}_{10}^{(1)} - \text{D}_{01}^{(1)} + \text{D}_{11}^{(1)} \right)
	+ 
	\(\text{D}_{10}^{(2)}\)^{-1} \cdot \left(\text{D}_{11}^{(2)} - \text{D}_{01}^{(2)} - \text{D}_{10}^{(2)} + \text{D}_{00}^{(2)} \right)
	\\
	& - \text{A}_1 - \text{A}_2\,,
\end{split}
\ee
where the annuli $\text{A}_1$ and $\text{A}_2$ are up to isotopy obtained from $(\text{D}_{10})^{-1} \cdot \text{D}_{11}$ and $(\text{D}_{10})^{-1} \cdot \text{D}_{10}$ respectively by attaching a flow line starting and ending on the disk:
\be
\text{A}_1 = \picLabel{{$\text{D}_{11}$}{70}{100}
				{$\text{D}_3$}{135}{200}} {wsHc12_a21_capped_new} {.350} \,,
\qquad
\text{A}_2 = \picLabel{{$\text{D}_{10}$}{70}{100}
				{$\text{D}_3$}{135}{200}} {wsHc12_a21_bis_capped_new} {.350} \,.
\ee
Both $\text{A}_1$ and $\text{A}_2$ have a contractible component, so they can be written as the difference of two disks in the worldsheet skein, one of which cancels the disks $(\text{D}_{10}^{(2)})^{-1} \cdot \text{D}_{11}^{(2)}$ and $(\text{D}_{10}^{(2)})^{-1} \cdot \text{D}_{10}^{(2)}$ in $\sfA^{\textnormal{ws}}_2$, respectively.

\subsubsection*{The operator $\sfA^{\textnormal{ws}}_3$ from $\sfH^{\textnormal{ws}}(c_{21},a_{12})$, $\sfH^{\textnormal{ws}}(c_{11})$ and $\sfH^{\textnormal{ws}}(c_{22})$}
As in Section~\ref{ssec:A2}, we obtain a recursion relation from $\sfH^{\textnormal{ws}}(c_{21},a_{12})$ by eliminating negative puncture curves.
By symmetry between Morse flows on the two Legendrian tori $\Lambda_1$ and $\Lambda_2$, the argument is directly analogous and we get
\be
\begin{split}\label{eq:skein-A3}
	\sfA^{\textnormal{ws}}_3
	& \simeq -\( \text{D}_{10}^{(2)} \)^{-1} \cdot \left(\text{D}_{00}^{(2)} - \text{D}_{10}^{(2)} - \text{D}_{01}^{(2)} + \text{D}_{11}^{(2)} \right)\\
	& + 
	\(\text{D}_{10}^{(1)}\)^{-1} \cdot \left(\text{D}_{11}^{(1)} - \text{D}_{01}^{(1)} - \text{D}_{10}^{(1)} + \text{D}_{00}^{(1)} \right)
	- \text{A}_1' - \text{A}_2'\,,
\end{split}
\ee
for two annuli $\text{A}_1'$ and $\text{A}_2'$.

Equations~\eqref{eq:skein-A1},~\eqref{eq:skein-A2},~\eqref{eq:skein-A3} shows that $\sfA_{i}^{\ws}$, $i=1,2,3$ lie in the world sheet skein $D$-module.

\subsection{Proof of Theorem \ref{t:Hopf} -- specialization of $\sfA_i^{\ws}$ and uniqueness} \label{ssec:specialization-and-uniqueness}

We describe the homomorphism $\rho_L$ that maps the (universal) worldsheet skein $D$-module $\sfD^{\ws}_\Gamma$ to the skein $D$-module $\sfD_L = \Sk(\partial L) / \sfI_L$ for Lagrangian fillings over corresponding to points in the augmentation variety $V_\Gamma$ and prove uniqueness of solutions to the skein recursions.

\subsubsection{Lagrangian fillings of the Legendrian Hopf link conormal}
Recall that the augmentation variety $V_\Gamma$ can be considered as the moduli space of (equivalence classes of Floer theoretically unobstructed) Lagrangian fillings of the Legendrian tori $\Lambda_\Gamma$ at infinity. Recall further that the variety has two branches $V_{\Gamma}=V_{\Gamma}^{(1)}\cup V_{\Gamma}^{(2)}$, see \eqref{eq:V-Hopf-branches}. The branch $V^{(1)}$ is the surface parameterized by the two periods along the exact filling $L_{0}$ with topology $T^{2}\times\R$, we treat this case in Section \ref{sssec: uniquess link complement}. The branch $V^{(2)}$ is the product of the augmentation curves of the two unknot components of $\Gamma$. We use notation $L_{st}$, $s,t\in\{l^{\pm},m^{\pm},d\}$ to denote these fillings. This notation is explained pictorially in Figure \ref{fig:Hopf-phases}, where $L_{st}$ refers to a Lagrangian filling with $L_{1}$ on the branch labeled by $s$ and $L_{2}$ labeled by $t$. Note that the roles of $l^{\pm}$ and $m^{\pm}$ are reverted between the components reflecting the fact that for the Hopf link the meridian of one component is the longitude of the other. 
\begin{figure}[h!]
\begin{center}
\includegraphics[width=0.5\textwidth]{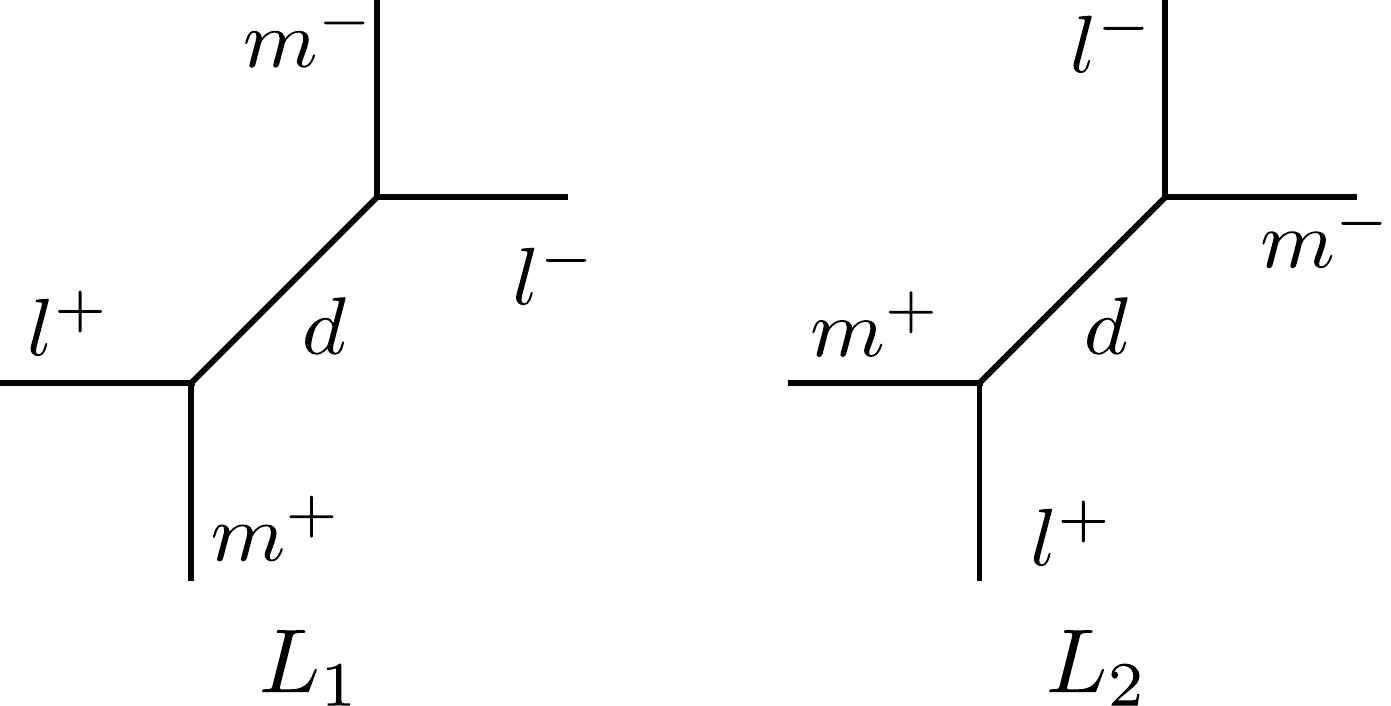}
\caption{Lagrangian fillings of the Hopf link Legendrian over points in $V_{\Gamma}^{(2)}$.}
\label{fig:Hopf-phases}
\end{center}
\end{figure}

We will prove that the partition function $\sfZ_{L_{st}}$ is uniquely determined by $\rho_{L_{st}}(\sfA_{i}^{\ws})$ for $(s,t)\in\{(l^{+},l^{+}),(l^{+},m^{+}), (l^{+},d),(d,d)\}$. Other cases follow using various symmetries. Consider for example $L_{l^{+},l^{-}}$. Here we shift the conormal $L_{2;l}$ in the negative rather than the positive direction. This means that the second of the two basic annuli changes orientation both in the conormal and in $S^{3}$. Thus, the component in $S^{3}$ is the positive Hopf link. The same effect can be obtained by starting from $L_{l^{+}l^{+}}$, reverting the orientation of $S^{3}$ and changing basis in the conormal of the second component. Uniqueness of $L_{l^{+}l^{+}}$ then implies uniqueness for $L_{l^{+}l^{-}}$. Next observing that the complement Lagrangian of one component is Lagrangian isotopic to the conormal Lagrangian of the other we relate $L_{l^{\pm},d}$ and $L_{m^{\pm},d}$ and relabeling components we get further combinations. When the Lagrangians lie on the same leg we must consider their relative position here again relabeling reduces to one case. By these symmetries it follows that it is sufficient to consider the cases listed above. To simplify notation we will denote them simply $L_{ll}$, $L_{lm}$, $L_{ld}$, and $L_{dd}$.

\subsubsection{Two unknot conormals -- $L_{ll}$}\label{ss:RecursionConConFilling}
With $\sfA_i = \rho_\Lambda(\sfA_i^\textnormal{ws})$, the recursion operators at infinity in the case of the conormal-conormal filling read 
\begin{align}\label{eq:A1-ConormalInfinity}
	\sfA_1
	& = 
	\(
	\sfP^{(1)}_{0,0}
	- a^{-1} \sfP^{(1)}_{1,0}
	- a^{-1} a_{\Lambda_1} \sfP^{(1)}_{0,1}
	+ \sfP^{(1)}_{1,1}
	\) \otimes a_{\Lambda_2}\\
	&\quad - a_{\Lambda_1} \otimes  
	\(
	\sfP^{(2)}_{0,0}
	- a^{-1} \sfP^{(2)}_{1,0}
	- a^{-1} a_{\Lambda_2} \sfP^{(2)}_{0,1}
	+ \sfP^{(2)}_{1,1}
	\)\,\\ \label{eq:A2-ConormalInfinity}
	\sfA_2 
	&=
	a_{\Lambda_1}
	\(
	- a \sfP^{(1)}_{-1,0}
	+ \sfP^{(1)}_{0,0}
	\)\otimes 1
	+
	\(
	a_{\Lambda_1} \sfP^{(1)}_{-1,1}
	- a \sfP^{(1)}_{0,1}
	\)\otimes (a_{\Lambda_2})^2
	\\
	&+ 1\otimes a_{\Lambda_2}
	\(
	a (a_{\Lambda_2})^{-1} \sfP^{(2)}_{0,-1}
	- \sfP^{(2)}_{1,-1}
	- \sfP^{(2)}_{0,0}
	+ a \sfP^{(2)}_{1,0}
	\) \,\\ \label{eq:A3-ConormalInfinity}
	\sfA_3
	& = 
	1\otimes a_{\Lambda_2}
	\(
	- a \sfP^{(2)}_{-1,0}
	+ \sfP^{(2)}_{0,0}
	\)
	+
	(a_{\Lambda_1})^2 \otimes 
	\(
	a_{\Lambda_2} \sfP^{(2)}_{-1,1}
	- a \sfP^{(2)}_{0,1}
	\)
	\\
	& + a_{\Lambda_1}
	\(
	a (a_{\Lambda_1})^{-1} \sfP^{(1)}_{0,-1}
	- \sfP^{(1)}_{1,-1}
	- \sfP^{(1)}_{0,0}
	+ a \sfP^{(1)}_{1,0}
	\)\otimes 1\,.
\end{align}

Here, $\rho_\Lambda$ is defined via the choices of framing vector fields and 4-chains described below. The correct powers of $a$ and $a_{\Lambda_i}$ can be determined by counting intersections of the present holomorphic curves with the 4-chains of the Lagrangians. Alternatively, these are the unique powers up to an overall scaling so that the operators \eqref{eq:A1-conormal}, \eqref{eq:A2-conormal}, and \eqref{eq:A3-conormal} determined below annihilate the HOMFLYPT generating function of the Hopf link.

Denote the two components of the filling $L_{l;i}\approx S^{1}\times\R^{2}$, $i=1,2$. 
We use coordinates on $L_{l;1}$ and $L_{l;2}$ so that the operators $\sfP^{(i)}_{k,l}$ on $\Lambda_i$ act on the skein of $L_{l;i}$ by the formulas described in Section \ref{ssec:skeinsolidtorus}, and we substitute $a_{\Lambda_i} \to a_{L_i}$, where $a_{L_i}$ denotes the $a$-variable corresponding to $L_{l;i}$. We choose the framing vector field along the central $S^1 \subseteq L_{l;i}$ and the 4-chains $C_i$ of $L_{l;i}$ in such a way that the central $S^1$ has vanishing self-linking in these coordinates and that the $C_i$ do not intersect any of the present Lagrangians. The framing vector field of $S^3$ is chosen so that the two components of the Hopf link, framed by this vector field, have self-linking number $+1$. As a consequence, the 4-chain $C_{S^3}$ will intersect $L_{l;i}$ in a $(-1,1)$-curve (note that the orientation on $L_{l;i}$ we use is opposite to the orientation induced by $S^3$). Therefore, when acting with the operators $\sfP^{(i)}_{k,l}$ on the partition function with boundary near the central $S^1$, we need to rescale the operators according to 
\be
\sfP^{(i)}_{k,l} \mapsto a^{k+l} \sfP^{(i)}_{k,l}
\ee
to obtain the correct action. The recursion relation thus reads:
\begin{align}
	\label{eq:A1-conormal}
	\sfA_1 =
	\ &= \
	\(
	\sfP^{(1)}_{0,0}
	- \sfP^{(1)}_{1,0}
	-a_{L_1} \sfP^{(1)}_{0,1}
	+a^2 \sfP^{(1)}_{1,1}
	\)\otimes a_{L_2}\\ \notag
	&\quad- a_{L_1}\otimes
	\(
	\sfP^{(2)}_{0,0}
	- \sfP^{(2)}_{1,0}
	-a_{L_2} \sfP^{(2)}_{0,1}
	+ a^2 \sfP^{(2)}_{1,1}
	\),\\ 
	\label{eq:A2-conormal}
	\sfA_2 
	\ &= \
	a_{L_1}
	\(
	-\sfP^{(1)}_{-1,0}
	+ \sfP^{(1)}_{0,0}\)\otimes 1
	+  \(a_{L_{1}}\sfP^{(1)}_{-1,1}
	- a^2 \sfP^{(1)}_{0,1}
	\)\otimes (a_{L_2})^2\\\notag
	&\quad+ 
	1\otimes a_{L_{2}}\((a_{L_{2}})^{-1}\sfP^{(2)}_{0,-1}
	-\sfP^{(2)}_{1,-1}
	-\sfP^{(2)}_{0,0}
	+a^2 \sfP^{(2)}_{1,0}
	\),\\ 
	\label{eq:A3-conormal}
	\sfA_3
	\ &= \
	1\otimes a_{L_2}
	\(
	-\sfP^{(2)}_{-1,0}
	+ \sfP^{(2)}_{0,0}\)
	+ (a_{L_1})^2\otimes \(a_{L_{2}}\sfP^{(2)}_{-1,1}
	- a^2 \sfP^{(2)}_{0,1}
	\) \\\notag
	&\quad+a_{L_1}
	\(
	(a_{L_1})^{-1} \sfP^{(1)}_{0,-1}
	- \sfP^{(1)}_{1,-1}
	-\sfP^{(1)}_{0,0}
	+a^2 \sfP^{(1)}_{1,0}
	\)\otimes 1\,.
\end{align}  

We show next that these operators determine a partition function with initial conditions corresponding to the conormal filling uniquely. The partition function takes values only in the non-negative part $\widehat{\Sk}_{+}(L_{l;1})\otimes \widehat{\Sk}_{+}(L_{l;2})$ of the skein modules of $L_{l;1}$ and $L_{l;2}$ (i.e., curves in non-negative homology classes). 
Let $\sfZ_j$, $j=0,1$ be two partition functions with degree $(0,0)$ parts equal to $1$ and consider $R$-linear maps of homogeneous degree $(i,j)$,
\be
A_{i,j},B_{i,j} \ \colon \ \Sk_+(L_1) \otimes \Sk_+(L_2) \ \to \ \Sk_+(L_1)  \otimes \Sk_+(L_2)
\ee
Assume that $\sfZ_k$, $k=0,1$ satisfy
	\begin{equation}\label{eq:(P_{1,0} - P_{0,0}) otimes 1 - 1 otimes (P_{1,0} - P_{0,0}) = A_(i,j)}
		\left(\left(\sfP_{1,0}^{(1)} - \sfP_{0,0}^{(1)}\right) \otimes a_{L_1}^{-1} - a_{L_2}^{-1} \otimes \left(\sfP_{1,0}^{(2)} - \sfP_{0,0}^{(2)}\right)\right) \sfZ_k =  \sum_{i,j \geq 0, (i,j) \neq (0,0)} A_{i,j}(\sfZ_k)
	\end{equation}
	and
	\begin{equation}\label{eq:(P_{1,-1} - P_{0,-1}) otimes 1) = B_(i,j)}
		\left(\left(a_{L_1} \sfP_{1,-1}^{(1)} - \sfP_{0,-1}^{(1)}\right) \otimes 1\right)\sfZ_k = \sum_{i,j \geq 0} B_{i,j}(\sfZ_k).
	\end{equation}
We show below that then $\sfZ_0 = \sfZ_1$, which implies the desired uniqueness.

We induct on $(i,j)$. The degree $(0,0)$ parts of $\sfZ_0$ and $\sfZ_1$ are both equal to 1 by assumption. Fix $(i,j)$ and assume that we know that the degree $(k,l)$ parts of $\Phi_0$ and $\Phi_1$ are equal for all $(k,l)$ with $k \leq i, l \leq j, (k,l) \neq (i,j)$.
Note that
\begin{align*}
	&\left(\left(\sfP_{1,0}^{(1)} - \sfP_{0,0}^{(1)}\right) \otimes a_{L_1}^{-1} - a_{L_2}^{-1} \otimes \left(\sfP_{1,0}^{(2)} - \sfP_{0,0}^{(2)}\right)\right) \sfW_{\mu,\emptyset} \otimes \sfW_{\lambda,\emptyset} \\
	&\qquad = z \left(\sum_{\ydiagram{1} \in \mu} q^{2 c(\ydiagram{1})} - \sum_{\ydiagram{1} \in \lambda} q^{2 c(\ydiagram{1})} \right)\sfW_{\mu,\emptyset} \otimes \sfW_{\lambda,\emptyset}
\end{align*}
vanishes exactly when $\mu = \lambda$. Since the degree $(i,j)$ part of $ \textnormal{Sk}_+(L_1) \otimes \textnormal{Sk}_+(L_2)$ is a free module over the integral domain $R$, it follows from the induction hypothesis and \eqref{eq:(P_{1,0} - P_{0,0}) otimes 1 - 1 otimes (P_{1,0} - P_{0,0}) = A_(i,j)} that the degree $(i,j)$ part of $\sfZ_1 - \sfZ_0$ is equal to zero if $i \neq j$ and equal to
	\begin{equation}\label{eq:diagonal partition function in degree i}
		\sum_{\lambda \vdash i} b_\lambda \sfW_{\lambda,\emptyset} \otimes \sfW_{\lambda,\emptyset}
	\end{equation}
for some $b_\lambda \in R$ if $i = j$.
	
Let $i = j$ and recall that  
	\begin{equation*}
		(a_{L} \sfP_{1,-1} - \sfP_{0,-1})(\sfW_{\lambda,\emptyset}) = \sum_{\lambda \in \mu + \ydiagram{1}} \left(a_{L}^{2} q^{2 c(\ydiagram{1})} - 1 \right) \sfW_{\mu,\emptyset}.
	\end{equation*}
Using the induction hypothesis and \eqref{eq:(P_{1,-1} - P_{0,-1}) otimes 1) = B_(i,j)}, it follows that the degree $(i-1,i)$ part of 
$$
\left(\left(a_{L_1} \sfP_{1,-1}^{(1)} - \sfP_{0,-1}^{(1)}\right) \otimes 1\right)(\sfZ_1 - \sfZ_0)
$$ 
vanishes, and we find
\begin{equation*}
\sum_{\lambda \vdash i} \sum_{\lambda \in \mu + \ydiagram{1}} \left(a_{L_1}^{2} q^{2 c(\ydiagram{1})} - 1 \right) b_\lambda \sfW_{\mu,\emptyset} \otimes \sfW_{\lambda,\emptyset} = 0,
\end{equation*}
	which implies that $b_\lambda = 0$ for all $\lambda \vdash i$ and consequently $\sfZ_{0}=\sfZ_{1}$. \\

\subsubsection{One unknot conormal and one unknot complement -- $L_{lm}$}\label{sec:conormal-complement-filling}
As the curves relevant to the skein recursion lie at infinity they are unaffected by the particular choice of filling. However, the $4$-chains do depend on the filling and we need to trace how they affect the recursion.
 
We choose the 4-chain $C_{2}'$ of $L_{2;m}$ by gluing the 4-chains $C_{S^3}$ of $S^{3}$ and the previous 4-chain $C_2$ of $L_{2;l}$. We modify $C_2$ so that it intersects $S^3$ in a parallel (unlinked) copy of the second component of the Hopf link. The intersection curve then has linking number $-1$ with the first Hopf link component and the effect on the partition function is the substitution
\be
\sfW^{(1)}_{\lambda,\overline{\mu}} \mapsto a_{L_2}^{-|\lambda| + |\mu|} \sfW^{(1)}_{\lambda,\overline{\mu}}, 
\ee
which changes the operators in \eqref{eq:A1-ConormalInfinity}, \eqref{eq:A2-ConormalInfinity}, and \eqref{eq:A3-ConormalInfinity} after the substitution $a_{\Lambda_i} \to a_{L_i}$ according to
\be
\sfP^{(1)}_{i,j} \mapsto a_{L_{2}}^{-j} \sfP^{(1)}_{i,j}.
\ee

Here we use the same symbol $a_{L_2}$ to denote the $a$-variable corresponding to both $L_{2;l}$ and $L_{2;m}$. Now we glue $S^3$ to $L_{2;l}$ and the 4-chain $C_{S^3}$ of $S^3$ to the 4-chain $C_2$ of $L_{2;l}$ to obtain the complement filling $L_{2;m}$ with 4-chain $C_{2}'$. Since at infinity, $C_{2}' = C_2 \cup C_{S^3}$, we get the induced change of variables $a \mapsto a a_{\Lambda_2}$ in relation to Section \ref{ss:RecursionConConFilling}.

Since the cycle $(0,1)$ in $\Lambda_2$ contracts in $L_{2;m}$ and $(1,0)$ does not, it is convenient to change basis on $H_1(\Lambda_2)$ by
\be\label{eq:conormal-complement-H1-basis}
\sfP^{(2)}_{i,j} \mapsto \sfP^{(2)}_{-j,i}\,.
\ee
This relabeling then preserves the orientation of $\Lambda_2$ and the skein algebra \eqref{eq:commutator}. In the new basis, the operators at infinity become
\begin{align}
	\label{eq:skein-A1,con-comp4chain}
	\sfA_1
	& = 
	\(
	\sfP^{(1)}_{0,0}
	- a^{-1} a_{\Lambda_2}^{-1} \sfP^{(1)}_{1,0}
	- a^{-1} a_{\Lambda_1} a_{\Lambda_2}^{-2} \sfP^{(1)}_{0,1}
	+ a_{\Lambda_2}^{-1} \sfP^{(1)}_{1,1}
	\) \otimes a_{\Lambda_2}\\
	&\quad - a_{\Lambda_1} \otimes  
	\(
	\sfP^{(2)}_{0,0}
	- a^{-1} a_{\Lambda_2}^{-1} \sfP^{(2)}_{0,1}
	- a^{-1} \sfP^{(2)}_{-1,0}
	+ \sfP^{(2)}_{-1,1}
	\) \, ,\\
	\label{eq:skein-A2,con-comp4chain}
	\sfA_2 
	&=
	a_{\Lambda_1}
	\(
	- a a_{\Lambda_2} \sfP^{(1)}_{-1,0}
	+ \sfP^{(1)}_{0,0}
	\)\otimes 1
	+
	\(
	a_{\Lambda_1} a_{\Lambda_2}^{-1} \sfP^{(1)}_{-1,1}
	- a \sfP^{(1)}_{0,1}
	\)\otimes (a_{\Lambda_2})^2
	\\
	&+ 1\otimes a_{\Lambda_2}
	\(
	a \sfP^{(2)}_{1,0}
	- \sfP^{(2)}_{1,1}
	- \sfP^{(2)}_{0,0}
	+ a a_{\Lambda_2} \sfP^{(2)}_{0,1}
	\)\,,\\
	\label{eq:skein-A3,con-comp4chain}
	\sfA_3
	& = 
	1\otimes a_{\Lambda_2}
	\(
	- a a_{\Lambda_2} \sfP^{(2)}_{0,-1}
	+ \sfP^{(2)}_{0,0}
	\)
	+
	(a_{\Lambda_1})^2 \otimes 
	\(
	a_{\Lambda_2} \sfP^{(2)}_{-1,-1}
	- a a_{\Lambda_2} \sfP^{(2)}_{-1,0}
	\)
	\\
	& + a_{\Lambda_1}
	\(
	a (a_{\Lambda_1})^{-1} (a_{\Lambda_2})^2 \sfP^{(1)}_{0,-1}
	- a_{\Lambda_2} \sfP^{(1)}_{1,-1}
	- \sfP^{(1)}_{0,0}
	+ a a_{\Lambda_2} \sfP^{(1)}_{1,0}
	\)\otimes 1\,.
\end{align}
Finally, we choose the framing vector field of $S^3$ along the intersection $S^3 \cap L_{2;m}$ before shifting $L_{2;m}$ so that this circle has self-linking number $+1$, and we choose framing vector field of $L_{2;m}$ and perform the shifts of $L_1$ and $L_{2;m}$ in such a way that near the intersection of $L_1 \cap L_{2;m}$, the framing vector field of $L_{2;m}$ induces trivial self-linking for this intersection and $L_1$ is shifted relative to $L_{2;m}$ in the direction of a 1-form dual to the tangent vector of the first component of the Hopf link. This induces the rescaling $\sfP^{(2)}_{i,j} \mapsto a^{-i+j} \sfP^{(2)}_{i,j}$ since $C_{S^3}$ now intersects $L_{2;m}$ in a $(-1,-1)$-curve, and $\sfP^{(1)}_{i,j} \mapsto a^{i+j} a_{L_{2}}^{i} \sfP^{(1)}_{i,j}$ since $C_{2;m}$ intersects $L_1$ in a $(0,1)$-curve and $C_{S^3}$ intersects $L_1$ in a $(-1,1)$-curve. We then get the final form of the operators
\begin{align}
	\label{eq:A1-conormal-complement}
	\sfA_1
	\ &= \
	\(
	\sfP^{(1)}_{0,0}
	- \sfP^{(1)}_{1,0}
	- a_{L_1} a_{L_{2}}^{-2} \sfP^{(1)}_{0,1}
	+ a^2 \sfP^{(1)}_{1,1}
	\)\otimes a_{L_{2}}\\ \notag
	&\quad- a_{L_1}\otimes
	\(
	\sfP^{(2)}_{0,0}
	- a_{L_{2}}^{-1} \sfP^{(2)}_{0,1}
	- \sfP^{(2)}_{-1,0}
	+ a^2 \sfP^{(2)}_{-1,1}
	\),\\ 
	\label{eq:A2-conormal-complement}
	\sfA_2 
	\ &= \
	a_{L_1}
	\(
	-\sfP^{(1)}_{-1,0}
	+ \sfP^{(1)}_{0,0}\)\otimes 1
	+  \(
	a_{L_{1}} a_{L_{2}}^{-2} \sfP^{(1)}_{-1,1}
	- a^2 \sfP^{(1)}_{0,1}
	\)\otimes (a_{L_{2}})^2\\\notag
	&\quad+ 
	1 \otimes a_{L_{2}}\(\sfP^{(2)}_{1,0}
	- \sfP^{(2)}_{1,1}
	- \sfP^{(2)}_{0,0}
	+ a^2 a_{L_{2}} \sfP^{(2)}_{0,1}
	\),\\ 
	\label{eq:A3-conormal-complement}
	\sfA_3
	\ &= \
	1\otimes a_{L_{2}}
	\(
	- a_{L_{2}}^{-1} \sfP^{(2)}_{0,-1}
	+ a_{L_{2}}^{-2} \sfP^{(2)}_{0,0}\)
	+ (a_{L_1})^2 \otimes \( 
	a_{L_{2}}^{-1} \sfP^{(2)}_{-1,-1}
	- a^2 a_{L_{2}}^{-1} \sfP^{(2)}_{-1,0}
	\) \\\notag
	&\quad+a_{L_1}
	\(
	(a_{L_1})^{-1} \sfP^{(1)}_{0,-1}
	- \sfP^{(1)}_{1,-1}
	- a_{L_{2}}^{-2} \sfP^{(1)}_{0,0}
	+ a^2 \sfP^{(1)}_{1,0}
	\)\otimes 1\,,
\end{align} 
(where $\sfA_3$ was multiplied by $a_{L_{2}}^{-2}$) that are the image of $\sfA_i^{\ws}$ under $\rho_{L_{lm}}$ that generate the HOMFLYPT skein $D$-module for $\partial L_{lm}$. 

The skein valued recursion operators \eqref{eq:A1-conormal-complement}, \eqref{eq:A2-conormal-complement}, and \eqref{eq:A3-conormal-complement} uniquely determine the skein valued partition function of $L_{lm}$. With coordinates as above, the partition function is supported in degrees $(i,j)$ such that $i \geq 0$ and $i + j\geq 0$, and its degree $(0,0)$ part equals $1$. We again use induction on $(i,j)$. Assuming that the degree $(k,l)$ part of a solution $\sfZ$ to \eqref{eq:A1-conormal-complement}, \eqref{eq:A2-conormal-complement}, and \eqref{eq:A3-conormal-complement} is uniquely determined by the conditions on $\sfZ$ for all $(k,l) \neq (i,j)$ with $k \leq i, l \leq j$, then, as before, the degree $(i,j)$ part of $\sfZ$ is uniquely determined by \eqref{eq:A1-conormal-complement} unless $i = -j$ in which case it is determined up to addition of a term of the form  
\begin{equation}
	\sum_{\lambda \vdash i} b_\lambda \sfW_{\lambda,\emptyset} \otimes \sfW_{\emptyset,\overline{\lambda}}.
\end{equation}
Here the coefficients $b_\lambda$ are uniquely determined by \eqref{eq:A3-conormal-complement}.

\subsubsection{A middle leg Lagrangian and an unknot conormal -- $L_{dl}$}\label{ssec:middle conormal}
We obtain the recursion in this case from the case of $L_{ll}$ by applying the coordinate change
\be
\sfP^{(1)}_{i,j} \mapsto (-1)^j \sfP^{(1)}_{j,-i-j}
\ee 
(moving the first brane to the middle leg) and adjusting signs and powers of $a_{L_1}$ with the result 
\begin{align}
	\sfA_1
	\ &= \
	\(
	\sfP^{(1)}_{0,0}
	- (a_{L_1})^{-1} \sfP^{(1)}_{0,-1}
	+ \sfP^{(1)}_{1,-1}
	- a^2 \sfP^{(1)}_{1,-2}
	\)\otimes a_{L_2}\\ \notag
	&\quad- a_{L_1}\otimes
	\(
	\sfP^{(2)}_{0,0}
	- \sfP^{(2)}_{1,0}
	- (a_{L_1})^{-2} a_{L_2} \sfP^{(2)}_{0,1}
	+ a^2 \sfP^{(2)}_{1,1}
	\),\\ 
	\sfA_2 
	\ &= \
	\(
	- a_{L_1} \sfP^{(1)}_{0,1}
	+ \sfP^{(1)}_{0,0}
	- (a_{L_2})^{2} \sfP^{(1)}_{1,0}
	+ a^2 (a_{L_2})^{2} \sfP^{(1)}_{1,-1}
	\)\otimes 1 \\\notag
	&\quad+ 
	a_{L_1} \otimes \(
	\sfP^{(2)}_{0,-1}
	- a_{L_2} \sfP^{(2)}_{1,-1}
	- (a_{L_1})^{-2} a_{L_2} \sfP^{(2)}_{0,0}
	+ a^2 a_{L_2} \sfP^{(2)}_{1,0}
	\),\\ 
	\sfA_3
	\ &= \
	a_{L_2} \otimes
	\(
	-\sfP^{(2)}_{-1,0}
	+ \sfP^{(2)}_{0,0}
	+ \sfP^{(2)}_{-1,1}
	- a^2 (a_{L_1})^2 (a_{L_2})^{-1} \sfP^{(2)}_{0,1}
	\) \\\notag
	&\quad+a_{L_1}
	\(
	- \sfP^{(1)}_{-1,1}
	+ \sfP^{(1)}_{-1,0}
	- \sfP^{(1)}_{0,0}
	+ a^2 a_{L_1} \sfP^{(1)}_{0,-1}
	\)\otimes 1\,.
\end{align} 
Here the support of the partition function involves also the $a$-grading. If $t$ denotes the $a$-degree and $(i,j)$ are as usual then the partition function is supported in degrees $(i,j,t)$, where $j\ge 0$, $i+2t\geq0$, and $t \geq 0$.  Uniqueness follows as before.

\subsubsection{Two middle leg conormals -- $L_{dd}$}\label{ssec:middle middle}
In this case we also move the second Lagrangian in $L_{dl}$ to the middle leg and accordingly change coordinates on $\Lambda_2$ according to $\sfP^{(2)}_{i,j} \mapsto (-1)^i \sfP^{(2)}_{i,-i-j}$. The recursion becomes
\begin{align}
	\sfA_1
	\ &= \
	\(
	\sfP^{(1)}_{0,0}
	- (a_{L_1})^{-1} \sfP^{(1)}_{0,-1}
	+ \sfP^{(1)}_{1,-1}
	- a^2 \sfP^{(1)}_{1,-2}
	\)\otimes a_{L_2}\\ \notag
	&\quad- a_{L_1}\otimes
	\(
	\sfP^{(2)}_{0,0}
	+ \sfP^{(2)}_{0,-1}
	- (a_{L_1})^{-2} a_{L_2} \sfP^{(2)}_{1,-1}
	- a^2 \sfP^{(2)}_{1,-2}
	\),\\ 
	\sfA_2 
	\ &= \
	\(
	- a_{L_1} (a_{L_2})^{-2} \sfP^{(1)}_{0,1}
	+ \sfP^{(1)}_{0,0}
	- \sfP^{(1)}_{1,0}
	+ a^2 \sfP^{(1)}_{1,-1}
	\)\otimes a_{L_2} \\\notag
	&\quad+ 
	a_{L_1} \otimes \(
	(a_{L_2})^{-1} \sfP^{(2)}_{-1,1}
	+ \sfP^{(2)}_{-1,0}
	- \sfP^{(2)}_{0,0}
	- a^2 \sfP^{(2)}_{0,-1}
	\),\\ 
	\sfA_3
	\ &= \
	a_{L_2} \otimes
	\(
	\sfP^{(2)}_{0,1}
	+ \sfP^{(2)}_{0,0}
	- \sfP^{(2)}_{1,0}
	- a^2 (a_{L_1})^2 (a_{L_2})^{-1} \sfP^{(2)}_{1,-1}
	\) \\\notag
	&\quad+a_{L_1}
	\(
	-\sfP^{(1)}_{-1,1}
	+ \sfP^{(1)}_{-1,0}
	- \sfP^{(1)}_{0,0}
	+ a^2 a_{L_1} \sfP^{(1)}_{0,-1}
	\)\otimes 1\,.
\end{align} 
Here, we rescaled $\sfA_2$ by $(a_{L_2})^{-1}$ and used that $(a_{L_2})^{-1} \sfP^{(1)}_{0,0} - (a_{L_1})^{-1} \sfP^{(2)}_{0,0} =  a_{L_2} \sfP^{(1)}_{0,0} - a_{L_1} \sfP^{(2)}_{0,0}$. 

To bring the operators into a more symmetric form, we change the framing on the second brane: $\sfP^{(2)}_{i,j} \mapsto \sfP^{(2)}_{i+j,j}$. The resulting operators read
\begin{align}
	\sfA_1
	\ &= \
	\(
	\sfP^{(1)}_{0,0}
	- (a_{L_1})^{-1} \sfP^{(1)}_{0,-1}
	+ \sfP^{(1)}_{1,-1}
	- a^2 \sfP^{(1)}_{1,-2}
	\)\otimes a_{L_2}\\ \notag
	&\quad- a_{L_1}\otimes
	\(
	\sfP^{(2)}_{0,0}
	+ \sfP^{(2)}_{-1,-1}
	- (a_{L_1})^{-2} a_{L_2} \sfP^{(2)}_{0,-1}
	- a^2 \sfP^{(2)}_{-1,-2}
	\),\\ 
	\sfA_2 
	\ &= \
	\(
	- a_{L_1} \(a_{L_2}\)^{-2} \sfP^{(1)}_{0,1}
	+ \sfP^{(1)}_{0,0}
	- \sfP^{(1)}_{1,0}
	+ a^2 \sfP^{(1)}_{1,-1}
	\)\otimes a_{L_2} \\\notag
	&\quad+ 
	a_{L_1} \otimes \(
	\(a_{L_2}\)^{-1} \sfP^{(2)}_{0,1}
	+ \sfP^{(2)}_{-1,0}
	- \sfP^{(2)}_{0,0}
	- a^2 \sfP^{(2)}_{-1,-1}
	\),\\ 
	\sfA_3
	\ &= \
	a_{L_2} \otimes
	\(
	\sfP^{(2)}_{1,1}
	+ \sfP^{(2)}_{0,0}
	- \sfP^{(2)}_{1,0}
	- a^2 (a_{L_1})^2 (a_{L_2})^{-1} \sfP^{(2)}_{0,-1}
	\) \\\notag
	&\quad+a_{L_1}
	\(
	- \sfP^{(1)}_{-1,1}
	+ \sfP^{(1)}_{-1,0}
	- \sfP^{(1)}_{0,0}
	+ a^2 a_{L_1} \sfP^{(1)}_{0,-1}
	\)\otimes 1\,.
\end{align} 
The partition function is supported on $(i,j,t)$ with $i+2t\ge 0$, $i+2t+j\ge 0$, and $t \geq 0$. Uniqueness follows as for $L_{ll}$.

\subsubsection{The Hopf link complement -- $L_0$}\label{sssec: uniquess link complement} 
Lastly, we consider the Hopf link complement filling $L_0 \approx T^2 \times \R$. To obtain the operators at infinity, we proceed as for $L_{lm}$, but instead of gluing $C_1$ to $C_{S^3}$, we glue $C_1$ to $C'_{2}$ to obtain the 4-chain $C$. At infinity, the 4-chain becomes the union $C = C_1 \cup C'_{2}$, and we set $a_{\Lambda_1} = a_{\Lambda_2}= a_\Lambda$. 

We then choose the coordinates on $L_0 \cong T^2 \times \R$ in such a way that $\sfP_{i,j}^{(2)}$ now acts as multiplication by $\sfP_{i,j}$ from above in the skein of $T^2$ and $\sfP_{i,j}^{(1)}$ acts as multiplication by $\sfP_{-i,j}$ from below. The operators acting on the skein of $L$ become
\begin{align}
	\label{eq:A1-complement}
	\sfA_1
	\ &= \
	a^{-1} [\sfP_{0,1},\,\cdot\,] + a^{-1} [\sfP_{-1,0},\,\cdot\,] - [\sfP_{-1,1},\,\cdot\,]
	\,,\\ 
	\label{eq:A2-complement}
	\sfA_2 
	\ &= \
	a [\sfP_{1,0},\,\cdot\,] - [\sfP_{1,1},\,\cdot\,] + a [\sfP_{0,1},\,\cdot\,]
	\,,\\ 
	\label{eq:A3-complement}
	\sfA_3
	\ &= \
	- a [\sfP_{0,-1},\,\cdot\,] + [\sfP_{-1,-1},\,\cdot\,] - a [\sfP_{-1,0},\,\cdot\,]
	\,,
\end{align} 
where in addition to the described changes we rescaled $\sfA_2$ and $\sfA_3$ by a factor of $a_L^{-2}$. 

Given the initial condition that the degree $(0,0)$-part of the partition function $\sfZ$ is equal to $1$ and $\sfZ$ is non-vanishing only in non-negative degrees, uniqueness follows again by induction on the degree $(i,j)$: assuming that we have already proven uniqueness in degrees $(k,l) \neq (i,j)$ with $k \leq i$ and $l \leq j$, then the degree $(i-1,j)$-part of $\sfA_1(\sfZ)$ implies uniqueness in degree $(i,j)$ if $j > 0$ since $\left[\sfP_{-1,0},\sfP_{i,j}\right] = (q^{-j} - q^{j})\sfP_{i-1,j}$. If $j=0$ and $i\neq 0$, then the degree $(i-1,-1)$-part of $\sfA_3(\sfZ)$ proves the desired uniqueness since $\left[\sfP_{-1,-1},\sfP_{i,0}\right] = (q^{i} - q^{-i})\sfP_{i-1,-1}$.

This completes the proof of Theorem \ref{t:Hopf}.\qed

\subsection{Framing and reduction to quantum curves}\label{ssec:auxthmHopf}
In this section we discuss tow further aspects of Theorem~\ref{t:Hopf}.

\subsubsection{Skein recursion identities in generic framing}
As before we obtain skein recursions in arbitrary framings by straightforward change of variables \eqref{eq:framed-Pij} applied to $\sfA_{j}$, $j=1,2,3$ for the filling $L_{ll}$ (with partition function given by all colored HOMFLYPT)
\begin{align}
	\begin{split} 				\label{eq:A1-skein-framed}
	\sfA_1^{(f_1,f_2)} 			
	& = 
	\(
		\sfP^{(1)}_{0,0}
		- \sfP^{(1)}_{1,0}
		-(-1)^{f_1}a_{\Lambda_1} \sfP^{(1)}_{f_1,1}
		+ (-1)^{f_1} a^2 \sfP^{(1)}_{1+f_1,1}
	\) \otimes a_{\Lambda_2}
	\\
	& - a_{\Lambda_1} \otimes  
	\(
		\sfP^{(2)}_{0,0}
		- \sfP^{(2)}_{1,0}
		-(-1)^{f_2}a_{\Lambda_2} \sfP^{(2)}_{f_2,1}
		+ (-1)^{f_2}a^2 \sfP^{(2)}_{1+f_2,1}
	\), 
	\end{split}
\\[4pt]
	\begin{split} 				\label{eq:A2-skein-framed}
	\sfA_2^{(f_1,f_2)} & =  			
	a_{\Lambda_1}
	\(
		-\sfP^{(1)}_{-1,0}
		+ \sfP^{(1)}_{0,0}
	\)\otimes 1
	+ (-1)^{f_1}
	\(
		a_{\Lambda_1} \sfP^{(1)}_{-1+f_1,1}
		- a^2 \sfP^{(1)}_{f_1,1}
	\)\otimes (a_{\Lambda_2})^2
	\\
	&+ 1\otimes a_{\Lambda_2} 
	\(
		(-1)^{f_2} (a_{\Lambda_2})^{-1} \sfP^{(2)}_{-f_2,-1}
		- (-1)^{f_2} \sfP^{(2)}_{1-f_2,-1}
		-\sfP^{(2)}_{0,0}
		+a^2 \sfP^{(2)}_{1,0}
	\) \,,
	\end{split}
	\\[4pt]
	\begin{split} 				\label{eq:A3-skein-framed}
	\sfA_3^{(f_1,f_2)}				
	& = 
	1\otimes a_{\Lambda_2}
	\(
		-\sfP^{(2)}_{-1,0}
		+ \sfP^{(2)}_{0,0}
	\)
	+
	(a_{\Lambda_1})^2 \otimes 
	(-1)^{f_2}\(
		a_{\Lambda_2} \sfP^{(2)}_{-1+f_2,1}
		- a^2 \sfP^{(2)}_{f_2,1}
	\)
	\\
	& + a_{\Lambda_1}
	\(
		(-1)^{f_1}(a_{\Lambda_1})^{-1} \sfP^{(1)}_{-f_1,-1}
		- (-1)^{f_1} \sfP^{(1)}_{1-f_1,-1}
		-\sfP^{(1)}_{0,0}
		+a^2 \sfP^{(1)}_{1,0}
	\)\otimes 1\,.
	\end{split}
\end{align}
These operators define the skein $D$-module of the (conormal-conormal filling of the) framed Hopf link.

\subsubsection{Specialization to the $U(1)$ skein}\label{sec:U1-spec}
The $U(1)$ limit on each of the branes $L_1, L_2$ is defined by
\[
	a_{\Lambda_i}\to q
	\qquad 
	\sfP^{(k)}_{\pm 1,0} \to \hat x_k^{\pm 1}
	\qquad
	\sfP^{(k)}_{0,\pm 1} \to \hat y_k^{\pm 1}
\]
with $\hat y_k \hat x_r = q^{2\delta_{kr}} \hat y_k \hat x_r $.
From the above relations it follows also that
\[
	\sfP^{(k)}_{1,1} \to - q  \hat x_k\hat y_k
	\qquad
	\sfP^{(k)}_{-1,1} \to -q^{-1} \hat x_k \hat y_k^{-1}
	\qquad
	\sfP^{(k)}_{1,-1} \to q^{-1} \hat x_k^{-1} \hat y_k\,.
\]
Using these in the skein valued recursion operators~\eqref{eq:skein-A1},~\eqref{eq:skein-A2} and~\eqref{eq:skein-A3} gives
\begin{align}\notag
	\hat A_1 
	& = 
	q \left[\(1 - \hat y_1 - q \hat x_1 + a^2 q \hat x_1\hat y_1\) 
	- \(1  - \hat y_2 - q \hat x_2  + a^2 q \hat x_2\hat y_2\)\right], 
	\\\notag
	%%%
	\hat A_2 
	& = 
	q \(-\hat y_1^{-1} + 1 + q \hat x_1 \hat y_1^{-1} - q \, a^2\, \hat x_1 \) 
	+ q \(q^{-1} \hat x_2^{-1} - q^{-1} \hat x_2^{-1}\hat y_2 - 1 + a^2\, \hat y_2 \), 
	\\\notag
	\hat A_3 
	& = 
	q \(-\hat y_2^{-1} + 1 + q \hat x_2 \hat y_2^{-1} - q \, a^2\, \hat x_2 \) 
	+ q \(q^{-1} \hat x_1^{-1} - q^{-1} \hat x_1^{-1}\hat y_1 - 1 + a^2\, \hat y_1 \) 
	\\\label{eq:U1-recursion}
\end{align}
In agreement with the quantum $U(1)$-relations in~\cite{Aganagic:2013jpa}

Taking the semi-classical limit of the three skein recursions~\eqref{eq:U1-recursion}, we get the augmentation variety of the Hopf link cut out by the three equations \eqref{eq:classical-aug-var-factorized}.
Recall that these equations define an augmentation variety with two branches \eqref{eq:V-Hopf-branches} which intersect along the one dimensional locus \eqref{eq:V-Hopf-branches-intersection}.
Viewing the augmentation variety as the characteristic variety of the $D$-module determined by the quantized versions $\hat A_{j}$, $j=1,2,3$, the phenomenon that branches of the augmentation variety intersect in codimension one is related to the irreducibility of the $D$-module.

\section{Generalized quiver structures of skein valued curve counts}\label{sec:generalized-quiver-structures}
The skein valued recursion operators determined above determines the partition functions of the Lagrangian fillings, for the conormal the HOMFLYPT polynomials colored by arbitrary partitions. In this section we describe these, and how they can be viewed as being generated by basic disks and annuli with linking boundaries, generalizing the quiver description of holomorphic curves for knots and links proposed in~\cite{Ekholm:2018eee, Ekholm:2019lmb}.

\subsection{Basic curves}\label{sec:basic-curves}
As a preliminary step towards a skein valued quiver description, we discuss the building blocks. We first discuss the building blocks themselves and how they interact through a natural algebraic structure in the skein.

\subsubsection{The disk}\label{ssec:disk}
The skein valued disk partition function in framing $f$ is, see Section \ref{sec:recSktoricbrane}: 
\be\label{eq:psi-disk}
	\sfPsi_{\mathrm{di}}^{(f)}
	:= 
	\sum_\lambda  ((-1)^{|\lambda|} q^{\kappa(\lambda)})^f  \left[\prod_{\ydiagram{1}\in\lambda} \frac{- q^{-c(\ydiagram{1})}}{q^{h(\ydiagram{1})}-q^{-h(\ydiagram{1})}}\right]  \sfW_{\lambda,\emptyset}
	\ \in \ \Sk(S^{1}\times\R^{2}),
\ee
which coincides with the partition function of a toric brane in $\IC^3$, and as such it satisfies the skein recursion obtained from \eqref{eq:recSktoricbrane} by specializing $a_\Lambda=a_L$ and by applying the framing transformations \eqref{eq:framed-Pij} and \eqref{eq:aL-Pij-framed}
\be\label{eq:toric-brane-recursion}
	(\sfP_{0,0} - \sfP_{1,0} + (-1)^{1-f}a_{L} \sfP_{f,1})\;\sfPsi_{\mathrm{di}}^{(f)} \ =  \ 0,
\ee
where the extra 4-chain intersection with $L$ gives the effective rescaling $\sfP_{f,1}\mapsto a_{L}^{-f} \sfP_{f,1}$.

\subsubsection{The annulus}\label{ssec:annulus}
The skein valued annulus partition function in framing $(f_{1},f_{2})$ is, see  Section \ref{sec:skein-rec-annulus}:
\be
	\sfPsi_{\mathrm{an}}^{(f_1, f_2)} := \sum_{\lambda}  
	\left[(-1)^{|\lambda|} q^{\kappa(\lambda)}\right]^{f_1+f_2} \
	\sfW_{\lambda,\emptyset}\otimes \sfW_{\lambda,\emptyset}
	\ \in \ \Sk(S^{1}\times\R^{2})\otimes \Sk(S^{1}\times\R^{2}).
\ee
It satisfies the following identity  
\be
\left(a_{L_1}^{-1}(\sfP^{(1)}_{0,0} - \sfP^{(1)}_{1,0})\otimes 1 
-1\otimes a_{L_2}^{-1}(\sfP^{(2)}_{0,0} - \sfP^{(2)}_{1,0})\right) \, \sfPsi_{\mathrm{an}}^{(f_1, f_2)}  \ = \ 0,
\ee
that follows from the (infinite) recursion
\be
	\left(
	a_{L_1}^{-1} (\sfP^{(1)}_{1,0} - \sfP^{(1)}_{0,0}) - (q-q^{-1})\sum_{n\geq 1} {\sfC_{n}}^{(1;f_{1})} \otimes {\sfC_{n}}^{(2;f_{2})} 
	\right)
	\,
	\sfPsi_{\mathrm{an}}^{(f_1, f_2)} 
	\ = \ 0
\ee
where $\sfC_{n}^{(a;f)}$ is a curve on brane $L_a$ winding $n$ times around the longitude with $n-1$ positive self-crossings framing changed to framing $f$.

\subsubsection{The twisted annulus}\label{ssec:twistedannulus}
An alternative annulus is obtained by taking one of the boundaries to wind along the anti-meridian of the two solid tori: 
\be\label{eq:twisted-ann}
	\sfPsi_{\widetilde{\mathrm{an}}}^{(f_1, f_2)} := \sum_{\lambda}  
	\left[(-1)^{|\lambda|} q^{\kappa(\lambda)}\right]^{f_1+f_2} \
	\sfW_{\lambda,\emptyset}\otimes \sfW_{\emptyset,\overline{\lambda}}
	\ \in \ \widehat{\Sk}(S^{1}\times\R^{2})\otimes \widehat{\Sk}´(S^{1}\times\R^{2})
\ee
which in framing $(0,0)$ is annihilated by the following operators: 
\begin{align}
\label{eq:twAnnulus-B1}
\sfB_1 \ &= \ \( \sfP^{(1)}_{0,0} - \sfP^{(1)}_{1,0} \) \otimes a_{L_2} - a_{L_1} \otimes \( \sfP^{(2)}_{0,0} - \sfP^{(2)}_{-1,0} \) \,,\\
\label{eq:twAnnulus-B2}
\sfB_2 \ &= \ a_{L_1} \( -\sfP^{(1)}_{-1,0} + \sfP^{(1)}_{0,0}\) \otimes 1 + 1 \otimes a_{L_{2}}\(\sfP^{(2)}_{1,0} - \sfP^{(2)}_{0,0} \)\,,\\
\label{eq:twAnnulus-B3}
\sfB_3 \ &= \ \( (a_{L_1})^{-i} \sfP^{(1)}_{i,1} - (a_{L_{2}})^{-2i} \sfP^{(1)}_{0,1} \) \otimes 1 - 1 \otimes \( (a_{L_2})^{-i} \sfP^{(2)}_{-i,1} - (a_{L_2})^{-2i} \sfP^{(2)}_{0,1}\) \,,\\
\label{eq:twAnnulus-B4}
\sfB_4 \ &= \ \( (a_{L_1})^{-i} \sfP^{(1)}_{i,-1} - (a_{L_{1}})^{-2i} \sfP^{(1)}_{0,-1} \) \otimes 1 - 1 \otimes \( (a_{L_2})^{-i} \sfP^{(2)}_{-i,-1} - (a_{L_1})^{-2i} \sfP^{(2)}_{0,-1}\) 
\end{align}
which follow directly from the action of $\sfP_{i,j}$ on $\sfW_{\lambda,\overline{\mu}}$ described in Section \ref{ssec:skeinsolidtorus}.

\subsubsection{The anti-annulus}\label{ssec:antiannulus}
The skein valued anti-annulus partition function is: 
\be
	\sfPsi_{\overline{\mathrm{an}}}^{(f_1, f_2)} 
	= \sum_\lambda 
	\left[(-1)^{|\lambda|} q^{\kappa(\lambda)}\right]^{f_1-f_2} 
	\sfW_{\lambda,\emptyset}\otimes \sfW_{\lambda^t,\emptyset}
\ee
where we used $\kappa(\lambda)=-\kappa(\lambda^t)$, which satisfies the following identity 
\be
\left(a_{L_1}^{-1}(\sfP^{(1)}_{0,0} - \sfP^{(1)}_{1,0}) \otimes 1
-1\otimes a_{L_2}(\sfP^{(2)}_{0,0} - \sfP^{(2)}_{-1,0})\right)\;
\sfPsi_{\overline{\mathrm{an}}}^{(f_1, f_2)} 
\ = \ 0,
\ee
that follows from the (infinite) recursion
\be
\left(
a_{L_1}^{-1} (\sfP^{(1)}_{1,0} - \sfP^{(1)}_{0,0}) - (q-q^{-1})\sum_{n\geq 1} {\sfC_{n}}^{(1;f_{1})} \otimes {\sfC_{n}^{\ast}}^{(2;f_{2})} 
\right)\,
\sfPsi_{\overline{\mathrm{an}}}^{(f_1, f_2)} 
\ = \ 0
\ee
where ${\sfC_{n}^{\ast}}^{f}$ is a curve winding $n$ times around the longitude of a brane with $n-1$ negative self-crossings framing changed to framing $f$.

We will use the basic curve partition functions in Sections \ref{ssec:disk}--\ref{ssec:antiannulus} with additional homology variables. We write $\sfPsi_{\mathrm{di}}^{(f)}[\xi]$ for the partition function from $\sfPsi_{\mathrm{di}}^{(f)}$ by replacing $\sfW_\lambda$ by $\xi^{|\lambda|} \sfW_\lambda $. We use similar notation for the annuli partition functions.

\subsection{The algebra structure on $\Sk(S^1\times \IR^2)$}\label{sec:algebra-Ws}
We will also use the algebra structure of the HOMFLYPT skein module $\Sk(L)$. We recall its definition and discuss an extension describing products of linked curves.

The product of (unlinked) elements of $\Sk(L)$ is obtained by viewing the solid torus as the product $A\times I$, where $A$ is an annulus and $I$ an interval.  
We multiply one skein element with another by gluing the outer boundary component of one annulus to the inner boundary component of the other. The product is obviously commutative.

In the basis $\sfW_{\lambda,\overline{\mu}}$ the algebra is that of universal characters of $GL(N)$ defined in \cite{koike1989decomposition}
\be\label{eq:Koike-product}
	\sfW_{\lambda,\overline{\mu}}\, \sfW_{\nu,\overline{\rho}}
	= 
	\sum_{\sigma,\overline{\tau}}
	M^{\sigma,\overline{\tau}}_{\lambda,\overline{\mu}; \nu,\overline{\rho}}
	\sfW_{\sigma,\overline{\tau}}
\ee
where the structure constants $M^{\sigma,\overline{\tau}}_{\lambda,\overline{\mu}; \nu,\overline{\rho}}$ have an explicit expression in terms of Littlewood-Richardson coefficients $c_{\lambda\mu}^{\nu}$ as follows
\be\label{eq:Koike-product-M}
	M^{\sigma,\overline{\tau}}_{\lambda,\overline{\mu}; \nu,\overline{\rho}}
	=
	\sum_{\alpha,\beta,\gamma,\delta}
	\left[\sum_{\kappa} c_{\kappa\alpha}^{\lambda} c_{\kappa\beta}^{\rho}\right]
	\left[\sum_{\epsilon} c_{\epsilon\gamma}^{\mu} c_{\epsilon\delta}^{\nu} \right]
	c^{\sigma}_{\alpha\delta}c^{\tau}_{\beta\gamma}\,.
\ee

Specializing to the $\sfW_{\lambda} := \sfW_{\lambda,\emptyset}$ gives in fact the 
Littlewood-Richardson product~\cite{morton2002homfly}
\be\label{eq:LRprod}
\sfW_{\lambda} \sfW_{\mu} = \sum_{\nu} c_{\lambda\mu}^{\nu} \sfW_\nu\,.
\ee
This result is natural from the physical point of view, as mentioned in Section \ref{ssec:introskeinDmodule}, $\sfW_{\lambda} = \sfW_{\lambda,\emptyset}$ can be viewed as line operators  
in (complexified) $U(N)_\kappa$ Chern-Simons theory on $S^{1}\times\R^{2}$ and, as such, give a generalization of the Verlinde basis \cite{Verlinde:1988sn, Witten:1988hf}. 
There is a natural algebraic structure on this basis, corresponding to the fusion algebra, which in the limits of large rank and large level becomes the Littlewood-Richardson product. 

Another specialization that will be useful below corresponds to the product of  $\sfW_{\lambda,\emptyset}$ and $\sfW_{\emptyset,\overline{\rho}}$. In this case \eqref{eq:Koike-product} simplifies to
\be
	\sfW_{\lambda,\emptyset}\, \sfW_{\emptyset,\overline{\rho}}
	= 
	\sum_{\sigma,\overline{\tau}}
	\left[\sum_{\kappa} c_{\kappa\sigma}^{\lambda} c_{\kappa\tau}^{\rho}\right]
	\sfW_{\sigma,\overline{\tau}}^{(0,0)}\,.
\ee

Finally we introduce a twisted product on  $\sfW_{\lambda} := \sfW_{\lambda,\emptyset}$ controlled by a discrete parameter $f\in\IZ$, and defined as follows
\be\label{eq:twisted-prod}
	{\sfW_{\lambda} \star_f \sfW_{\mu} = \sum_{\nu} c_{\lambda\mu}^{\nu} \,q^{f (\kappa(\nu)-\kappa(\lambda)-\kappa(\mu))}\, \sfW_\nu}.
\ee
The case $\star_0$ is just the usual product, and we omit that in notation.
The product \eqref{eq:Koike-product} correspond to the product in the skein of the annulus $\Sk(S^1\times[0,1])$.
The twisted product \eqref{eq:twisted-prod} correspond to its framed counterpart (specialized to positive-winding curves).

\subsection{Proof of Theorem \ref{t: skein quiver 1}}
We find solutions to the recursion relations for the various Lagrangian fillings. 

\subsubsection{The filling $L_{lm}$}\label{sec:conormal-complement-solution}\label{sssec: L_lm}  
We find the solution to the recursion relations in Section \ref{sec:conormal-complement-filling}.
Recall the unknot partition function \eqref{eq: unknot skein quiver} corresponding to the product of unlinked disks (in framing zero):
\be
\sfZ^{}_U  
= 
\sfPsi_{\mathrm{di}}^{(0)}\,  \cdot \, \sfPsi_{\mathrm{di}}^{(1)}[a^2]
= \sum_{\lambda} \prod_{\ydiagram{1}\in\lambda} \frac{a^2 q^{c(\ydiagram{1})} - q^{-c(\ydiagram{1})}}{q^{h(\ydiagram{1})}-q^{-h(\ydiagram{1})}}\, \sfW_\lambda 
\ee
The partition function for the complement of the unknot is obtained by replacing $\sfW_\lambda$ by $(-1)^{|\lambda|}\sfW_{\lambda^t}$ in the sum, which gives 
\be
\sfZ^{c}_U  
= 
\sfPsi_{\mathrm{di}}^{(1)}\,  \cdot \, \sfPsi_{\mathrm{di}}^{(0)}[a^2]
= \sum_{\lambda} \prod_{\ydiagram{1}\in\lambda} \frac{q^{c(\ydiagram{1})} - a^2 q^{-c(\ydiagram{1})}}{q^{h(\ydiagram{1})}-q^{-h(\ydiagram{1})}}\, \sfW_\lambda 
\ee

The unique solution for $L_{lm}$ is the product of an unknot on $L_1$, and anti-unknot on $L_{2}$ and a twisted annulus \eqref{eq:twisted-ann} in framing $(0,0)$ connecting them
\be\label{eq:conormal-complement-Z}
\sfZ_{L_{lm}} = 
\left[\sfZ^{}_U  \otimes \sfZ^{c}_U  \right]\star_0 \sfPsi_{\widetilde{\mathrm{an}}}
\ee
Using the algebra defined in \eqref{eq:Koike-product} this can be expressed as follows
\be
\begin{split}
	\sfZ_{L_{lm}}  
	& = 
	\sum_{\alpha,\beta,\lambda,\mu,\nu,\kappa}
	c^\alpha_{\lambda\nu} c_{\kappa\beta}^{\mu} c_{\kappa\gamma}^{\nu}
	\left(\prod_{\ydiagram{1}\in\lambda} \frac{a^2 q^{c(\ydiagram{1})} - q^{-c(\ydiagram{1})}}{q^{h(\ydiagram{1})}-q^{-h(\ydiagram{1})}}\right)
	\left(\prod_{\ydiagram{1}\in\mu} \frac{ q^{c(\ydiagram{1})} - a^2 q^{-c(\ydiagram{1})}}{q^{h(\ydiagram{1})}-q^{-h(\ydiagram{1})}}\right)
	\sfW_{\alpha\emptyset}\otimes \sfW_{\beta\overline{\gamma}}\,.
\end{split}
\ee

Next we show that \eqref{eq:conormal-complement-Z} is annihilated by \eqref{eq:A1-conormal-complement}, \eqref{eq:A2-conormal-complement} and \eqref{eq:A3-conormal-complement}. We recall the recursion operators:
\begin{align}
	\sfA_1
	\ &= \
	\(
	\sfP^{(1)}_{0,0}
	- \sfP^{(1)}_{1,0}
	- a_{L_1} (a_{L_{2}})^{-2} \sfP^{(1)}_{0,1}
	+ a^2 \sfP^{(1)}_{1,1}
	\)\otimes a_{L_{2}}\\ \notag
	&\quad- a_{L_1}\otimes
	\(
	\sfP^{(2)}_{0,0}
	- (a_{L_{2}})^{-1} \sfP^{(2)}_{0,1}
	- \sfP^{(2)}_{-1,0}
	+ a^2 \sfP^{(2)}_{-1,1}
	\),\\ 
	\sfA_2 
	\ &= \
	a_{L_1}
	\(
	-\sfP^{(1)}_{-1,0}
	+ \sfP^{(1)}_{0,0}\)\otimes 1
	+  \(
	a_{L_{1}} (a_{L_{2}})^{-2} \sfP^{(1)}_{-1,1}
	- a^2 \sfP^{(1)}_{0,1}
	\)\otimes (a_{L_{2}})^2\\\notag
	&\quad+ 
	1 \otimes a_{L_{2}}\(\sfP^{(2)}_{1,0}
	- \sfP^{(2)}_{1,1}
	- \sfP^{(2)}_{0,0}
	+ a^2 a_{L_{2}} \sfP^{(2)}_{0,1}
	\),\\ 
	\sfA_3
	\ &= \
	1\otimes
	\(
	- \sfP^{(2)}_{0,-1}
	+ (a_{L_{2}})^{-1} \sfP^{(2)}_{0,0}\)
	+ (a_{L_1})^2 \otimes (a_{L_{2}})^{-1} \( 
	\sfP^{(2)}_{-1,-1}
	- a^2 \sfP^{(2)}_{-1,0}
	\) \\\notag
	&\quad+a_{L_1}
	\(
	(a_{L_1})^{-1} \sfP^{(1)}_{0,-1}
	- \sfP^{(1)}_{1,-1}
	- (a_{L_{2}})^{-2} \sfP^{(1)}_{0,0}
	+ a^2 \sfP^{(1)}_{1,0}
	\)\otimes 1\,.
\end{align} 

We will use the open recursion relations
\begin{align}
	\(\sfP'_{0,0} - \sfP'_{1,0} + a^2 \sfP'_{1,1} - a_L \sfP'_{0,1}\)\(\sfZ^{}_U\)& = 0\,,\\
	\(\sfP'_{0,0} - \sfP'_{1,0} + \sfP'_{1,1} - a^2 a_L \sfP'_{0,1}\)\(\sfZ^{c}_U\) &= 0\,
\end{align}
to reduce the claim to identities about the partition function of the twisted annulus $\sfPsi_{\widetilde{\mathrm{an}}}$.

For any partition function $B$, we have from the open recursion relation that
\begin{align}
	\(\sfP_{1,0} - a^2 \sfP_{1,1}\)\(\sfZ_U \star_0 B\) &= \sfZ_U \star_0 \(\sfP_{1,0} - \sfP_{1,1} \)\(B\) \,,\\
	\(\sfP_{-1,0} - \sfP_{-1,1}\)\(\sfZ_U \star_0 B\) &= \sfZ_U \star_0 \(\sfP_{-1,0} - a^2 \sfP_{-1,1} \)\(B\)\,,\\
	\(\sfP_{1,-1} - a^2 \sfP_{1,0}\)\(\sfZ_U \star_0 B\) &= \sfZ_U \star_0 \(\sfP_{1,-1} - \sfP_{1,0} \)\(B\)\,,\\
	\(\sfP_{-1,0} - a^2 \sfP_{-1,1}\)\(\sfZ^{c}_U \star_0 B\) &= \sfZ^{c}_U \star_0 \(\sfP_{-1,0} - \sfP_{-1,1} \)\(B\)\,,\\
	\(\sfP_{1,0} - \sfP_{1,1}\)\(\sfZ^{c}_U \star_0 B\) &= \sfZ^{c}_U \star_0 \(\sfP_{1,0} - a^2 \sfP_{1,1} \)\(B\)\,,\\
	\(\sfP_{-1,-1} - a^2 \sfP_{-1,0}\)\(\sfZ^{c}_U \star_0 B\) &= \sfZ^{c}_U \star_0 \(\sfP_{-1,-1} - \sfP_{-1,0} \)\(B\)\,.
\end{align}
It follows that $\sfA_1$, $\sfA_2$, and $\sfA_3$ annihilates $B \star_{0}\(\sfZ_U \otimes \sfZ^{c}_U\)$ if and only if
\begin{align}
	\sfA_1'
	\ &= \
	\(
	\sfP^{(1)}_{0,0}
	- \sfP^{(1)}_{1,0}
	- a_{L_1} (a_{L_{2}})^{-2} \sfP^{(1)}_{0,1}
	+ \sfP^{(1)}_{1,1}
	\)\otimes a_{L_{2}}\\ \notag
	&\quad- a_{L_1}\otimes
	\(
	\sfP^{(2)}_{0,0}
	- (a_{L_{2}})^{-1} \sfP^{(2)}_{0,1}
	- \sfP^{(2)}_{-1,0}
	+ \sfP^{(2)}_{-1,1}
	\)\,,\\
	\sfA_2' 
	\ &= \
	a_{L_1}
	\(
	-\sfP^{(1)}_{-1,0}
	+ \sfP^{(1)}_{0,0}\)\otimes 1
	+ \( a^2 a_{L_{1}} (a_{L_{2}})^{-2} \sfP^{(1)}_{-1,1}
	- a^2 \sfP^{(1)}_{0,1}
	\)\otimes (a_{L_{2}})^2\\\notag
	&\quad+ 
	1\otimes a_{L_{2}}\(\sfP^{(2)}_{1,0}
	-a^2 \sfP^{(2)}_{1,1}
	- \sfP^{(2)}_{0,0}
	+a^2 a_{L_{2}} \sfP^{(2)}_{0,1}
	\)\,,\\
	\sfA_3'
	\ &= \
	1\otimes
	\(
	-\sfP^{(2)}_{0,-1}
	+ (a_{L_1})^2 (a_{L_{2}})^{-1} \sfP^{(2)}_{0,0}\)
	+ (a_{L_1})^2\otimes (a_{L_{2}})^{-1} \(\sfP^{(2)}_{-1,-1}
	- \sfP^{(2)}_{-1,0}
	\) \\\notag
	&\quad+a_{L_1}
	\(
	(a_{L_1})^{-1} \sfP^{(1)}_{0,-1}
	- \sfP^{(1)}_{1,-1}
	- \sfP^{(1)}_{0,0}
	+ \sfP^{(1)}_{1,0}
	\)\otimes 1
\end{align}
annihilate $B$, where we used the identity
\be
(a_{L_{2}})^{-1} \sfP^{(2)}_{0,0} - a_{L_1} (a_{L_{2}})^{-2} \sfP^{(1)}_{0,0} = (a_{L_1})^{2} (a_{L_{2}})^{-1}\sfP^{(2)}_{0,0} - a_{L_1} \sfP^{(1)}_{0,0}
\ee
to rewrite $\sfA_3'$. It remains to show that the operators $\sfA_1'$, $\sfA_2'$, and $\sfA_3'$ all annihilate $\sfPsi_{\widetilde{\mathrm{an}}}$, which is an immediate consequence of \eqref{eq:twAnnulus-B1}, \eqref{eq:twAnnulus-B2}, \eqref{eq:twAnnulus-B3}, and \eqref{eq:twAnnulus-B4}.

\subsubsection{The filling $L_{dl}$}\label{sssec: L_dl}
The recursion in Section \ref{ssec:middle conormal}  
has the solution
\begin{equation}
	\sfZ = \( \(\sfPsi^{(0)}_{\text{di}} \overline{\sfPsi}^{(0)}_{\text{di}}[a^2] \) \otimes \(\sfPsi^{(0)}_{\text{di}} \sfPsi^{(1)}_{\text{di}}[a^2]\) \) \sfPsi^{(0,0)}_{\text{an}},
\end{equation}
where 
\begin{equation}
	\overline{\sfPsi}^{(0)}_\text{di} = \sum_\lambda \prod_{\square \in \lambda} \frac{q^{-c(\square)}}{q^{h(\square)} - q^{-h(\square)}} \sfW_{\emptyset,\overline{\lambda}}.
\end{equation}

This is seen as follows. Using the relative recursion relations for the disks, it follows that for any partition function $B$, we have that
\begin{equation}
	\begin{split}
		\(\sfP_{1,-1} - a^2 \sfP_{1,-2}\) \(B \star \sfPsi^{(0)}_{\text{di}} \star \overline{\sfPsi}^{(0)}_{\text{di}}[a^2] \) &= \( \sfP_{1,-1} - \sfP_{1,0} \)(B) \star \sfPsi^{(0)}_{\text{di}} \star \overline{\sfPsi}^{(0)}_{\text{di}}[a^2] \,,\\
		\(\sfP_{1,0} - a^2 \sfP_{1,-1}\)\(B \star \sfPsi^{(0)}_{\text{di}} \star \overline{\sfPsi}^{(0)}_{\text{di}}[a^2] \) &= \( \sfP_{1,0} - \sfP_{1,1} \)(B) \star \sfPsi^{(0)}_{\text{di}} \star \overline{\sfPsi}^{(0)}_{\text{di}}[a^2] \,,\\
		\(\sfP_{-1,0} - \sfP_{-1,1}\) \(B \star \sfPsi^{(0)}_{\text{di}} \star \overline{\sfPsi}^{(0)}_{\text{di}}[a^2] \) &= \( \sfP_{-1,0} - a^2 \sfP_{-1,-1} \)(B) \star \sfPsi^{(0)}_{\text{di}} \star \overline{\sfPsi}^{(0)}_{\text{di}}[a^2] \,,\\
		\(\sfP_{1,0} - a^2 \sfP_{1,1}\)\(B \star \sfPsi^{(0)}_{\text{di}} \star \sfPsi^{(1)}_{\text{di}}[a^2] \) &= \( \sfP_{1,0} -  \sfP_{1,1} \)(B) \star \sfPsi^{(0)}_{\text{di}} \star \sfPsi^{(1)}_{\text{di}}[a^2] \,,\\
		\(\sfP_{1,-1} - a^2 \sfP_{1,0}\)\(B \star \sfPsi^{(0)}_{\text{di}} \star \sfPsi^{(1)}_{\text{di}}[a^2] \) &= \( \sfP_{1,-1} -  \sfP_{1,0} \)(B) \star \sfPsi^{(0)}_{\text{di}} \star \sfPsi^{(1)}_{\text{di}}[a^2] \,,\\
		\(\sfP_{-1,0} - \sfP_{-1,1}\)\(B \star \sfPsi^{(0)}_{\text{di}} \star \sfPsi^{(1)}_{\text{di}}[a^2] \) &= \( \sfP_{-1,0} -  a^2 \sfP_{-1,1} \)(B) \star \sfPsi^{(0)}_{\text{di}} \star \sfPsi^{(1)}_{\text{di}}[a^2] \,.\\
	\end{split}
\end{equation}
This implies that $\sfZ$ is annihilated by $\sfA_1,\sfA_2,$ and $\sfA_3$ if and only if $\sfPsi^{(0,0)}_{\text{an}} = \sum_\lambda \sfW_{\lambda,\emptyset} \otimes \sfW_{\lambda,\emptyset}$ is annihilated by the operators
\begin{align}
	\sfA'_1
	\ &= \
	\(
	\sfP^{(1)}_{0,0}
	- (a_{L_1})^{-1} \sfP^{(1)}_{0,-1}
	+ \sfP^{(1)}_{1,-1}
	- \sfP^{(1)}_{1,0}
	\)\otimes a_{L_2}\\ \notag
	&\quad- a_{L_1}\otimes
	\(
	\sfP^{(2)}_{0,0}
	- \sfP^{(2)}_{1,0}
	- (a_{L_1})^{-2} a_{L_2} \sfP^{(2)}_{0,1}
	+ \sfP^{(2)}_{1,1}
	\),\\ 
	\sfA'_2 
	\ &= \
	\(
	- a_{L_1} \sfP^{(1)}_{0,1}
	+ \sfP^{(1)}_{0,0}
	- (a_{L_2})^{2} \sfP^{(1)}_{1,0}
	+ (a_{L_2})^{2} \sfP^{(1)}_{1,1}
	\)\otimes 1 \\\notag
	&\quad+ 
	a_{L_1} \otimes \(
	\sfP^{(2)}_{0,-1}
	- a_{L_2} \sfP^{(2)}_{1,-1}
	- (a_{L_1})^{-2} a_{L_2} \sfP^{(2)}_{0,0}
	+ a_{L_2} \sfP^{(2)}_{1,0}
	\),\\ 
	\sfA'_3
	\ &= \
	a_{L_2} \otimes
	\(
	-\sfP^{(2)}_{-1,0}
	+ \sfP^{(2)}_{0,0}
	+ a^2 \sfP^{(2)}_{-1,1}
	- a^2 (a_{L_1})^2 (a_{L_2})^{-1} \sfP^{(2)}_{0,1}
	\) \\\notag
	&\quad+a_{L_1}
	\(
	- a^2 \sfP^{(1)}_{-1,-1}
	+ \sfP^{(1)}_{-1,0}
	- \sfP^{(1)}_{0,0}
	+ a^2 a_{L_1} \sfP^{(1)}_{0,-1}
	\)\otimes 1\,.
\end{align} 
These relations follow immediately from the action of the operators $\sfP_{i,j}$ on $\sfW_{\lambda,\emptyset}$.

\subsubsection{The filling $L_{dd}$}\label{sssec: L_dd}
The recursion is found in Section \ref{ssec:middle middle} after the change of coordinates.

The solution is
\begin{equation}
	\sfZ = \( \(\sfPsi^{(0)}_{\text{di}} \overline{\sfPsi}^{(0)}_{\text{di}}[a^2] \) \otimes \(\sfPsi^{(1)}_{\text{di}} \overline{\sfPsi}^{(1)}_{\text{di}}[a^2]\) \) \widetilde{\sfPsi}^{(0,0)}_{\text{an}}\,,
\end{equation}
where
\begin{equation}
	\widetilde{\sfPsi}^{(0,0)}_{\text{an}} = \sum_\lambda \sfW_{\lambda,\emptyset} \otimes \sfW_{\emptyset,\overline{\lambda}}
\end{equation}

This is seen as follows: using the relative recursion relations for the disks, it follows that for any partition function $B$, we have that
\begin{equation}
	\begin{split}
		\(\sfP_{1,-1} - a^2 \sfP_{1,-2}\) \(B \star \sfPsi^{(0)}_{\text{di}} \star \overline{\sfPsi}^{(0)}_{\text{di}}[a^2] \) &= \( \sfP_{1,-1} - \sfP_{1,0} \)(B) \star \sfPsi^{(0)}_{\text{di}} \star \overline{\sfPsi}^{(0)}_{\text{di}}[a^2] \,,\\
		\(\sfP_{1,0} - a^2 \sfP_{1,-1}\)\(B \star \sfPsi^{(0)}_{\text{di}} \star \overline{\sfPsi}^{(0)}_{\text{di}}[a^2] \) &= \( \sfP_{1,0} - \sfP_{1,1} \)(B) \star \sfPsi^{(0)}_{\text{di}} \star \overline{\sfPsi}^{(0)}_{\text{di}}[a^2] \,,\\
		\(\sfP_{-1,0} - \sfP_{-1,1}\) \(B \star \sfPsi^{(0)}_{\text{di}} \star \overline{\sfPsi}^{(0)}_{\text{di}}[a^2] \) &= \( \sfP_{-1,0} - a^2 \sfP_{-1,-1} \)(B) \star \sfPsi^{(0)}_{\text{di}} \star \overline{\sfPsi}^{(0)}_{\text{di}}[a^2] \,,\\
		\(\sfP_{-1,-1} - a^2 \sfP_{-1,-2}\) \(B \star \sfPsi^{(1)}_{\text{di}} \star \overline{\sfPsi}^{(1)}_{\text{di}}[a^2] \) &= \( \sfP_{-1,-1} - \sfP_{-1,0} \)(B) \star \sfPsi^{(1)}_{\text{di}} \star \overline{\sfPsi}^{(1)}_{\text{di}}[a^2] \,,\\
		\(\sfP_{-1,0} - a^2 \sfP_{-1,-1}\)\(B \star \sfPsi^{(1)}_{\text{di}} \star \overline{\sfPsi}^{(1)}_{\text{di}}[a^2] \) &= \( \sfP_{-1,0} - \sfP_{-1,1} \)(B) \star \sfPsi^{(1)}_{\text{di}} \star \overline{\sfPsi}^{(1)}_{\text{di}}[a^2] \,,\\
		\(\sfP_{1,0} - \sfP_{1,1}\) \(B \star \sfPsi^{(1)}_{\text{di}} \star \overline{\sfPsi}^{(1)}_{\text{di}}[a^2] \) &= \( \sfP_{1,0} - a^2 \sfP_{1,-1} \)(B) \star \sfPsi^{(1)}_{\text{di}} \star \overline{\sfPsi}^{(1)}_{\text{di}}[a^2] \,.
	\end{split}
\end{equation}

This implies that $\sfZ$ is annihilated by $\sfA_1,\sfA_2,$ and $\sfA_3$ if and only if $\widetilde{\sfPsi}^{(0,0)}_{\text{an}}$ is annihilated by the operators
\begin{align}
	\sfA'_1
	\ &= \
	\(
	\sfP^{(1)}_{0,0}
	- (a_{L_1})^{-1} \sfP^{(1)}_{0,-1}
	+ \sfP^{(1)}_{1,-1}
	- \sfP^{(1)}_{1,0}
	\)\otimes a_{L_2}\\ \notag
	&\quad- a_{L_1}\otimes
	\(
	\sfP^{(2)}_{0,0}
	+ \sfP^{(2)}_{-1,-1}
	- \(a_{L_1}\)^{-2} a_{L_2} \sfP^{(2)}_{0,-1}
	- \sfP^{(2)}_{-1,0}
	\),\\ 
	\sfA'_2 
	\ &= \
	\(
	- a_{L_1} \(a_{L_2}\)^{-2} \sfP^{(1)}_{0,1}
	+ \sfP^{(1)}_{0,0}
	- \sfP^{(1)}_{1,0}
	+ \sfP^{(1)}_{1,1}
	\)\otimes a_{L_2} \\\notag
	&\quad +
	a_{L_1} \otimes \(
	\(a_{L_2}\)^{-1} \sfP^{(2)}_{0,1}
	+ \sfP^{(2)}_{-1,0}
	- \sfP^{(2)}_{0,0}
	- \sfP^{(2)}_{-1,1}
	\),\\ 
	\sfA'_3
	\ &= \
	a_{L_2} \otimes
	\(
	a^2 \sfP^{(2)}_{1,-1}
	+ \sfP^{(2)}_{0,0}
	-\sfP^{(2)}_{1,0}
	- a^2 (a_{L_1})^2 (a_{L_2})^{-1} \sfP^{(2)}_{0,-1}
	\) \\\notag
	&\quad+a_{L_1}
	\(
	- a^2 \sfP^{(1)}_{-1,-1}
	+ \sfP^{(1)}_{-1,0}
	- \sfP^{(1)}_{0,0}
	+ a^2 a_{L_1} \sfP^{(1)}_{0,-1}
	\)\otimes 1\,.
\end{align}
As before, the corresponding relations are be verified by a straightforward computation.

\subsubsection{The filling $L_{0}$}\label{sssec: L_0}
To finish the proof we consider the link complement filling, then the operators \eqref{eq:A1-complement}, \eqref{eq:A2-complement}, and \eqref{eq:A3-complement} in the recursion relation can be written purely as linear combinations of commutators $\left[\sfP_{i,j},\cdot \right]$. Hence, $\sfZ = 1$ is a solution. This is a manifestation of the fact that $L$ can by made disjoint from the zero-section via an exact Lagrangian isotopy.

The results in Sections \ref{sssec: L_lm}, \ref{sssec: L_dl}, \ref{sssec: L_dd}, and \ref{sssec: L_0} establish the desired result. \qed

\subsection{A closed form expression for the Hopf link in all colors}
In this section we show that the partition function of all colored HOMFLYPT polynomials of the Hopf link, i.e., the partition function of $L_{ll}$ (or $L_{mm}$) can be expressed as the product of four basic disks and the partition function from Section \ref{ssec:threeholedsphere} that we conjecture to correspond to a basic holomorphic three holed sphere. More precisely, we have the following.

\begin{thm}\label{thm:alternative quiver formula for Hopf}
	The HOMFLYPT partition function of the conormal filling $L_{ll}$ (and that of $L_{mm}$) of the Hopf link can be expressed as 
	\be\label{eq:alternative_quiver_formula_for_hopf}
	\sfZ_{L_{ll}} \ = \ \left(\left(\sfPsi^{(0)}_\textnormal{di} \otimes \sfPsi^{(0)}_\textnormal{di}  \right) \star_{-1} \widetilde{\sfPsi}_{(0,3)}^{(-1,-1)}[a] \right) \star_{0} \left(\sfPsi_\textnormal{di}^{(1)}[a^2] \otimes \sfPsi_\textnormal{di}^{(1)}[a^2] \right).
	\ee
\end{thm}

\begin{proof}
	Denote the right hand side of \eqref{eq:alternative_quiver_formula_for_hopf} by $\sfZ$. We show that $\sfZ$ satisfies the recursion relations for the Hopf link. The uniqueness part of Theorem \ref{t:Hopf} then gives the result. More precisely we claim that $\sfZ$ is annihilated by $\sfA_1$ and $\sfA_3$:
	\begin{align}\label{eq:recursion for neg Hopf link,first identity}
		&\left( \left( a_{L_1}^{-1}\left( \sfP_{1,0}^{(1)} - \sfP_{0,0}^{(1)} \right)- a^{2}a_{L_1}^{-1} \sfP_{1,1}^{(1)} + \sfP_{0,1}^{(1)} \right) \otimes 1 \right) \sfZ\\\notag 
		&\qquad = \left( 1 \otimes \left( a_{L_2}^{-1}\left(\sfP_{1,0}^{(2)} - \sfP_{0,0}^{(2)}\right) - a^{2} a_{L_2}^{-1} \sfP_{1,1}^{(2)} + \sfP_{0,1}^{(2)} \right) \right) \sfZ,
	\end{align}
	and
	\begin{align}\label{eq:recursion for neg Hopf link,second identity}
		&\left( \left(a_{L_1} \sfP_{1,-1}^{(1)} - \sfP_{0,-1}^{(1)} + a_{L_1} \left( \sfP_{0,0}^{(1)} - a^2 \sfP_{1,0}^{(1)} \right) \right) \otimes 1 \right) \sfZ\\\notag 
		&\qquad =  \left( 1 \otimes  a_{L_2} \left(\sfP_{0,0}^{(2)} - \sfP_{-1,0}^{(2)}\right) + a_{L_1}^2\otimes a_{L_2} \left(\sfP_{-1,1}^{(2)} - a^2 \sfP_{0,1}^{(2)} \right) \right) \sfZ.
	\end{align}
	The relation from $\sfA_2$ then follows by symmetry. 
	
	From the relative recursion relation 
	\be\label{eq:relative recursion disk}
	\left(\sfP'_{0,0} - \sfP'_{1,0} - \sfP'_{1,1} \right)\left(\sfPsi_\textnormal{di}^{(1)}\right) = 0
	\ee
	which is a framed\footnote{In addition to changing $\sfP'_{0,1} \to \sfP'_{1,1}$ we have also rescaled $\sfP'_{1,1} \to a^{-1}_L \sfP'_{1,1}$ to obtain the correct action on the framed solid torus, see Section \ref{ssec:skeinsolidtorus}.} version of \eqref{eq:relrecSktoricbrane}, 
	we have that 
	\begin{equation}\label{eq:P_{1,0} on cdot Psi^(0,1)(a)}
		\sfP_{1,0} \left( A \star_0 \sfPsi_\textnormal{di}^{(1)}[a^2] \right) = \sfP_{1,0} \left( A \right) \star_0 \sfPsi_\textnormal{di}^{(1)}[a^2] + a^{2} \sfP_{1,1} \left( A \star_0 \sfPsi_\textnormal{di}^{(1)}[a^2] \right)
	\end{equation}
	for any element $A \in \widehat{\textnormal{Sk}}(L)$ in the completed skein of the solid torus $L$.
	
	Similarly, the relative recursion relation 
	\be\label{eq:relative recursion inverse disk}
	\left(\sfP'_{0,0} - \sfP'_{1,0} - a_{L}\sfP'_{0,1} \right)\left(\sfPsi_\textnormal{di}^{(0)}\right) = 0
	\ee
	implies that 
	\begin{equation}\label{eq:P_{1,0} on fcdot{-1} Psi^-(0,1)(a^{-1})}
		\sfP_{1,0} \left( A \star_{-1} \sfPsi^{(0)}_\textnormal{di} \right) = \sfP_{1,0}(A) \star_{-1} \sfPsi^{(0)}_\textnormal{di} - a_L \sfP_{0,1} \left(A \star_{-1} \sfPsi^{(0)}_\textnormal{di}  \right).
	\end{equation}
	We find that 
	\begin{align*}
		&\left( \left( a_{L_1}^{-1}\left( \sfP^{(1)}_{1,0} - \sfP^{(1)}_{0,0} \right)- a^{2} a_{L_1}^{-1} \sfP^{(1)}_{1,1} + \sfP^{(1)}_{0,1} \right) \otimes 1 \right) \sfZ
		\ \overset{\eqref{eq:P_{1,0} on cdot Psi^(0,1)(a)},\eqref{eq:P_{1,0} on fcdot{-1} Psi^-(0,1)(a^{-1})}}{=}\\ 
		&\quad\left(\left(\sfPsi^{(0)}_\textnormal{di} \otimes \sfPsi^{(0)}_\textnormal{di} \right) \star_{-1} \left(\left(\sfP^{(1)}_{1,0} - \sfP^{(1)}_{0,0}\right)\otimes a_{L_1}^{-1}\right)\left( \widetilde{\sfPsi}_{(0,3)}^{(-1,-1)}[a] \right) \right) \star_0 \left(\sfPsi_\textnormal{di}^{(1)}[a^2] \otimes \sfPsi_\textnormal{di}^{(1)}[a^2] \right) \
		\overset{\eqref{eq:relation for Z_{(0,3)}),first identity}}{=}\\ 
		&\quad\left(\left(\sfPsi^{(0)}_\textnormal{di} \otimes \sfPsi^{(0)}_\textnormal{di}  \right) \star_{-1} \left(a_{L_2}^{-1} \otimes \left(\sfP^{(2)}_{1,0} - \sfP^{(2)}_{0,0}\right)\right)\left( \widetilde{\sfPsi}_{(0,3)}^{(-1,-1)}[a] \right) \right) \star_0 \left(\sfPsi_\textnormal{di}^{(1)}[a^2] \otimes \sfPsi_\textnormal{di}^{(1)}[a^2] \right) \
		\overset{\eqref{eq:P_{1,0} on cdot Psi^(0,1)(a)},\eqref{eq:P_{1,0} on fcdot{-1} Psi^-(0,1)(a^{-1})}}{=}\\ 
		&\quad\left(1 \otimes \left( a_{L_2}^{-1}\left(\sfP^{(2)}_{1,0} - \sfP^{(2)}_{0,0}\right) - a^{2} a_{L_2}^{-1} \sfP^{(2)}_{1,1} + \sfP^{(2)}_{0,1} \right) \right) \sfZ,
	\end{align*}
	and \eqref{eq:recursion for neg Hopf link,first identity} follows.

	The relative recursion relation \eqref{eq:relative recursion disk} implies the identities
	\begin{equation}\label{eq:P_{-1,1} on cdot Psi^(0,1)(a)} 
		\sfP_{-1,1} \left( A \star_0 \sfPsi_\textnormal{di}^{(1)}[a^2] \right) = \sfP_{-1,1} \left( A \right) \star_0 \sfPsi_\textnormal{di}^{(1)}[a^2] - a^2 \sfP_{-1,2} \left( A \star_0 \sfPsi_\textnormal{di}^{(1)}[a^2] \right),
	\end{equation}
	\begin{equation}\label{eq:P_{-1,0} on cdot Psi^(0,1)(a)}
		\sfP_{-1,0} \left( A \star_0 \sfPsi_\textnormal{di}^{(1)}[a^2] \right) = \sfP_{-1,0} \left( A \right) \star_0 \sfPsi_\textnormal{di}^{(1)}[a^2] - a^2 \sfP_{-1,1} \left( A \right) \star_0 \sfPsi_\textnormal{di}^{(1)}[a^2],
	\end{equation}
	and 
	\begin{equation}\label{eq:P_{1,-1} on cdot Psi^(0,1)(a)}
		\sfP_{1,-1} \left( A \star_0 \sfPsi_\textnormal{di}^{(1)}[a^2] \right) = \sfP_{1,-1} \left( A \right) \star_0 \sfPsi_\textnormal{di}^{(1)}[a^2] + a^2 \sfP_{1,0} \left( A \star_0 \sfPsi_\textnormal{di}^{(1)}[a^2]\right)
	\end{equation}
	for any $A \in \widehat{\textnormal{Sk}}(L)$. Similarly, the relation \eqref{eq:relative recursion inverse disk} implies the identities
	\begin{equation}\label{eq:P_{0,1} on fcdot{-1} Psi^-(0,1)(a^{-1})}
		\sfP_{0,1} \left( A \star_{-1} \sfPsi^{(0)}_\textnormal{di} \right) = \sfP_{0,1}(A) \star_{-1} \sfPsi^{(0)}_\textnormal{di} - a_{L} \sfP_{-1,2} \left(A \star_{-1} \sfPsi^{(0)}_\textnormal{di}  \right),
	\end{equation}
	\begin{equation}\label{eq:P_{-1,0} on fcdot{-1} Psi^-(0,1)(a^-1)}
		\sfP_{-1,0} \left( A \star_{-1} \sfPsi^{(0)}_\textnormal{di} \right) = \sfP_{-1,0} \left( A \right) \star_{-1} \sfPsi^{(0)}_\textnormal{di} + a_L \sfP_{-2,1} \left( A \right) \star_{-1} \sfPsi^{(0)}_\textnormal{di},
	\end{equation}
	and
	\begin{equation}\label{eq:P_{0,-1} on fcdot{-1} Psi^-(0,1)(a^-1)}
		\sfP_{0,-1} \left( A \star_{-1} \sfPsi^{(0)}_\textnormal{di} \right) = \sfP_{0,-1} \left( A \right) \star_{-1} \sfPsi^{(0)}_\textnormal{di} + a_L \sfP_{-1,0} \left( A \right) \star_{-1} \sfPsi^{(0)}_\textnormal{di}.
	\end{equation}
	Furthermore, we observe that for $A = B \star_{-1} C$, 
	\begin{equation}\label{eq:P_{k,l} and framed product}
		\sfP_{-1,1}(A) = \sfP_{-1,1}(B) \star_{-1} C\quad\text{ and }\quad
		\sfP_{1,-1}(A) = \sfP_{1,-1}(B) \star_{-1} C,
	\end{equation}
	where in the right hand sides $\sfP_{\pm1,\mp1}(B)$ should be understood as inserting $\sfP_{\pm1,\mp1}$ on the boundary of a small tubular neighborhood of $B$.	
	Combining \eqref{eq:P_{-1,1} on cdot Psi^(0,1)(a)}, \eqref{eq:P_{-1,0} on cdot Psi^(0,1)(a)}, \eqref{eq:P_{0,1} on fcdot{-1} Psi^-(0,1)(a^{-1})}, \eqref{eq:P_{-1,0} on fcdot{-1} Psi^-(0,1)(a^-1)}, and \eqref{eq:P_{k,l} and framed product} we obtain
	\begin{align*}
		&\left( a_{L_2} \left(\sfP_{0,0}^{(2)} - \sfP_{-1,0}^{(2)}\right) + a_{L_1}^2 \left(a_{L_2} \sfP_{-1,1}^{(2)} - a^2 \sfP_{0,1}^{(2)} \right) \right) \left( \left( \sfPsi^{(0)}_\textnormal{di}  \star_{-1} A \right) \star_0 \left(\sfPsi_\textnormal{di}^{(1)}[a^2] \right) \right)\\
		&= \Bigl( \sfPsi^{(0)}_\textnormal{di} \ \star_{-1}\\ 
		&\qquad\left( a_{L_2} \left(\sfP_{0,0}^{(2)} - \sfP_{-1,0}^{(2)}\right) + a_{L_2} \left(a^{2}+a_{L_1}^2\right) \sfP_{-1,1}^{(2)} - a_{L_2}^{2} \sfP_{-2,1}^{(2)} - a^2a_{L_1}^2 \sfP_{0,1}^{(2)} \right) (A) \Bigr)\\ 
		&\quad\quad\star_0 \left(\sfPsi_\textnormal{di}^{(1)}[a^2] \right).
	\end{align*}
	Then combining \eqref{eq:P_{1,-1} on cdot Psi^(0,1)(a)}, \eqref{eq:P_{0,-1} on fcdot{-1} Psi^-(0,1)(a^-1)}, and \eqref{eq:P_{k,l} and framed product} we obtain
	\begin{align*}
		&\left( a_{L_1} \sfP_{1,-1}^{(1)} - \sfP_{0,-1}^{(1)} + a_{L_1} \left( \sfP_{0,0}^{(1)} - a^2 \sfP_{1,0}^{(1)} \right) \right) \left( \left( \sfPsi^{(0)}_\textnormal{di}  \star_{-1} A \right) \star_0 \left(\sfPsi_\textnormal{di}^{(1)}[a^2] \right) \right)\\
		&= \Bigl( \sfPsi^{(0)}_\textnormal{di} \ \star_{-1} \\ 
		&\qquad\left( a_{L_1} \sfP_{1,-1}^{(1)} - \sfP_{0,-1}^{(1)} + a_{L_1} \left(\sfP_{0,0}^{(1)} - \sfP_{-1,0}^{(1)} \right) \right) (A) \Bigr) \\ 
		&\quad\quad \star_0  \left(\sfPsi_\textnormal{di}^{(1)}[a^2] \right).
	\end{align*}
	Thus, relation \eqref{eq:recursion for neg Hopf link,second identity} is equivalent to relation \eqref{eq:relation for Z_{(0,3)}),second identity} in Proposition \ref{prop:relation for Z_{(0,3)}}. The theorem follows.
\end{proof}

\subsection{A conjectured skein valued quiver for the Hopf link in all colors}\label{sec:skein valued-quiver}
We propose a quiver-like partition function that we conjecture is equivalent to the skein valued partition function for the filling $L_{ll}$ (and then also $L_{mm}$). The partition function is constructed from the unknot basic disks on each knot conormal and an annulus and an anti-annulus with boundaries linking in the conormal branes, see Figure \ref{fig:hopf-link-quiver-description}, as follows.

Recall the unknot skein valued partition function in zero framing~\cite{Ekholm:2020csl}, see Section \ref{sec:unknot}.
We observe that \eqref{eq:Z-unknot-explicit} in framing zero factorizes is a product of two unlinked disks
\be\label{eq: unknot skein quiver}
	\sfZ^{(0)}_U  = \sum_{\lambda}H_{U;\lambda}(a,q)\, \sfW_\lambda 
	= 
	\sfPsi_{\mathrm{di}}^{(0)}[a^{-1}]\,  \cdot \, \sfPsi_{\mathrm{di}}^{(1)}[a]\,.
\ee
In generic framing this becomes
\be
\begin{split}
	\sfZ_U^{(f)} &= 
	\sum_{\lambda_1,\lambda_2}
	\left[(-1)^{|\lambda_1|} q^{\kappa(\lambda_1)}\right]^{f} 
	\left[(-1)^{|\lambda_2|} q^{\kappa(\lambda_2)}\right]^{f+1}
	\\ 
	&\qquad \times\left[\prod_{\ydiagram{1}\in\lambda_1 } \frac{-q^{-c(\ydiagram{1})}}{q^{h(\ydiagram{1})} - q^{-h(\ydiagram{1})}}\right]
	\left[\prod_{\ydiagram{1}\in\lambda_2 } \frac{-q^{-c(\ydiagram{1})}}{q^{h(\ydiagram{1})} - q^{-h(\ydiagram{1})}}\right]
	a^{|\lambda_2|-|\lambda_1|}\, \sfW^{}_{\lambda_1}\star_f \sfW^{}_{\lambda_2}
\end{split}
\ee
We introduce notation for the unknot HOMFLYPT $H^O_{\lambda}$ colored by the partition $\lambda$ and for the annulus $H^{+}_{\nu}$ and anti-annulus $H^{-}_{\mu}$ in colors $\nu$ and $\mu$ and in appropriate framing and homology classes as follows 
\be
\begin{split}
	H^O_{\lambda}&= \  a^{|\lambda|} \, H_{U;\lambda}(a,q)\ =\  \prod_{\ydiagram{1}\in\lambda} \frac{a^2 q^{c(\ydiagram{1})} - q^{-c(\ydiagram{1})}}{q^{h(\ydiagram{1})}-q^{-h(\ydiagram{1})}}\\
	H^-_\mu&= (-a^{2})^{|\mu|}\\
	H_\nu^+&= q^{-2\kappa(\nu)}.\\
\end{split}
\ee
These are the coefficients of the following partition functions 
\be
\begin{split}
	\sum_\lambda H^O_{\lambda}\, \sfW_\lambda
	& = 
	\sfZ^{(0)}_U[a] 
	=
	\sfPsi_{\mathrm{di}}^{(0)}\,  \cdot \, \sfPsi_{\mathrm{di}}^{(1)}[a^2]
	\\
	\sum_\mu H^-_\mu\, \sfW_\mu\otimes \sfW_{\mu^t}
	& = \sfPsi_{\overline{\mathrm{an}}}^{(0,0)} [a^2]
	\\
	\sum_{\nu} H_\nu^+ \, \sfW_\nu\otimes \sfW_\nu 
	& = \sfPsi_{{\mathrm{an}}}^{(-1,-1)} 
\end{split}
\ee
Our conjectural form of the Hopf link partition function starts from the unknot disks on each brane:
\be
\begin{split}
	\sfZ_{\text{disks}} 
	& = \sum_{\lambda_1,\lambda_2}
	\(H^{O}_{\lambda_1} \sfW_{\lambda_1}\otimes 1\) \star_0 \(H^{O}_{\lambda_2} 1\otimes \sfW_{\lambda_2}\),
	\\
	& = \sfZ^{(0)}_{U_1}[a] \otimes \sfZ^{(0)}_{U_2}[a],  \\
	& = \left(\sfPsi_{\mathrm{di}}^{(0)}\,  \cdot \, \sfPsi_{\mathrm{di}}^{(1)}[a^2]\right)\otimes \left(\sfPsi_{\mathrm{di}}^{(0)}\,  \cdot \, \sfPsi_{\mathrm{di}}^{(1)}[a^2]\right).
\end{split}
\ee
Next to these we place an unlinked anti-annulus
\be
\begin{split}
	\sfZ_{\text{disks+a.a.}} 
	& = 
	\sfPsi_{\overline{\mathrm{an}}}^{(0,0)} [a^2]
	\star_0 
	\sfZ_{\text{disks}}
	\\
	& = 
	\sfPsi_{\overline{\mathrm{an}}}^{(0,0)} [a^2]
	\star_0 
	\left[
		\left(\sfPsi_{\mathrm{di}}^{(0)}\,  \cdot \, \sfPsi_{\mathrm{di}}^{(1)}[a^2]\right)\otimes \left(\sfPsi_{\mathrm{di}}^{(0)}\,  \cdot \, \sfPsi_{\mathrm{di}}^{(1)}[a^2]\right)
	\right]
\end{split}
\ee
Finally add a framed annulus that links all disks and the anti-annulus with $-1$ units of linking on each brane, and obtain the conjectured skein valued partition function of the Hopf Link conormal filling:
\be\label{eq:Hopf-quiver-skein}
\begin{split}
	\sfZ_{ll} & = 
	\sfPsi_{{\mathrm{an}}}^{(-1,-1)} 
	\star_{-1} 
	\sfZ_{\text{disks+a.a.}}
	\\
	& = 
	\sfPsi_{{\mathrm{an}}}^{(-1,-1)} 
	\star_{-1} 
	\left[ 
	\sfPsi_{\overline{\mathrm{an}}}^{(0,0)} [a^2]
	\star_0 
	\left[
		\left(\sfPsi_{\mathrm{di}}^{(0)}\,  \cdot \, \sfPsi_{\mathrm{di}}^{(1)}[a^2]\right)\otimes \left(\sfPsi_{\mathrm{di}}^{(0)}\,  \cdot \, \sfPsi_{\mathrm{di}}^{(1)}[a^2]\right)
	\right]
	\right]
\end{split}
\ee
Developing the product, we arrive at the following formula for the Hopf link HOMFLYPT polynomials, in terms of contributions from unknot basic disks and the annulus and the anti-annulus
\be\label{eq:Hopf-skein-quiver}
	{
	H^{\text{quiver}}_{\sigma_1,\sigma_2} = 
	\sum_{\lambda_1,\lambda_2,\mu,\nu}\sum_{\rho_1,\rho_2}
	q^{-(\kappa^{\sigma_1}_{\rho_1\nu})}
	q^{-(\kappa^{\sigma_2}_{\rho_2\nu})}
	H^O_{\lambda_1}H^O_{\lambda_2} H^-_\mu H_\nu^+
	c_{\lambda_1\mu}^{\rho_1} c_{\lambda_2\mu^t}^{\rho_2}
	c_{\rho_1\nu}^{\sigma_1} c_{\rho_2\nu}^{\sigma_2}
	}\,,
\ee
where
\[
\kappa^{\sigma}_{\rho\nu}=\kappa(\sigma)-\kappa(\rho)-\kappa(\nu).
\]
We have checked numerically that these agree with the known expression for HOMFLYPT polynomials, obtained e.g. from the topological vertex, see~\eqref{eq:Hopf-HOMFLY-top-vert}.

\subsubsection{Remarks on the skein-valued knots-quivers correspondence}
The quiver description that we found involves a finite number of basic holomorphic curves, that we denote `skein-basic' curves. In the specialization to the $U(1)$ skein, these become the basic holomorphic curves of the standard quiver descriptions of HOMFLYPT polynomials of knots and links (generalized to include annuli), see \cite{Kucharski:2017ogk, Ekholm:2018eee}.
As observed in \cite{Ekholm:2019lmb}, basic holomorphic curves in the $U(1)$ skein of a knot conormal filling correspond to LMOV invariants (which count open M2 branes) of disks winding once around the longitude of $L$, and all other LMOV invariants are encoded by their linking data. 

The skein-basic curves appearing in \eqref{eq:Hopf-quiver-skein} generalize this notion, implying in particular that the LMOV invariants of basic curves `stabilize' when the number of branes on each component of the Hopf link conormal is taken to infinity. 
Unlike the basic curves appearing in the $U(1)$ quiver description, skein-basic curves encode the HOMFLYPT polynomials colored by arbitrary partitions. 
Correspondingly, they also encode information about LMOV invariants colored by arbitrary partitions.

\subsection{Moduli space bifurcations and the Hopf link partition function}\label{sec:mdliproof}
In this section we outline a proof of Conjecture~\ref{c: skein quiver}. We first illustrate the argument to derive the description~\eqref{eq: unknot skein quiver} of the uknot partition function.

\subsubsection{The unknot partition function from bifurcations}
Recall the proof of the Ooguri-Vafa conjecture from~\cite{Ekholm:2019yqp,Ekholm:2021colored}, see Section \ref{ssec:reviewskeinsonbranes}: shift the conormal of the knot off of the zero section. For small shifts there is only the basic annulus stretching between the conormal and the zero section, after sufficient SFT stretching all boundary components shrink and curves leave the zero section and can be transferred to the conifold. 

Applying this to the unknot, it is possible to arrange the complex structure so that the boundary of the initial annulus shrinks to a point, without other skein moves. At the time when the boundary component becomes a point there must be already a corresponding moduli space of disks that crosses the Lagrangian at the shrinking moment. Furthermore, right before this moment the whole partition function is carried by the annulus, which means there is also another disk that cancels the contribution of the crossing disk before the crossing.

Now, the disk that does not cross the Lagrangian is the only curve without $a$-powers and hence accounts for the extremal HOMFLYPT of the unknot. We know the corresponding recursion relation and conclude this disk and its multiple covers contribute one of the factors in~\eqref{eq: unknot skein quiver}. Then in order to cancel this contribution, the crossing disk is the corresponding anti-disk (corresponding to framing $1$) and when crossing that disk picks up a factor of $a^{2}$. After the crossing moment these two disks account for the whole partition function and we conclude that~\eqref{eq: unknot skein quiver} holds.

\subsubsection{The Hopf link partition function from bifurcations}
Consider now instead the Hopf link, represented by two linked unknots where we take one of them very small. Consider shifted off conormals and corresponding basic annuli of the two components. Deform the complex structures so that the two annuli cross. After that time there are two unlinked annuli and one three holed sphere, with a framed unknot close to a standard unknot as boundary in $S^{3}$. We now stretch until all curve boundaries contract. Before the contraction, crossing disk and annuli moduli spaces that account for the extremal Hopf link (i.e., the curves with no $a$-powers in them) must appear. Since the extremal part of the (formal) partition function of a three holed sphere in Section \ref{ssec:threeholedsphere} is an annulus partition function, Theorem \ref{thm:alternative quiver formula for Hopf} implies that all curves are accounted for by multiple covers of two disks and one annulus. Further, since the whole partition function is carried by the curves with boundary before the crossing the curves that cross are the anti-curves of the extremal Hopf. We conclude that the partition function is generated by four disks and two annuli that come in curve/anti-curve pairs. Finally, consider the extremal curves in the standard framing. It is straightforward to check that the configuration of curves in Figure~\ref{fig:hopf-link-quiver-description} gives the partition function $1$ at $a=1$: disks first cancel and then the pair of annulus/anti-annulus is a framed version of the standard canceling pair. If we knew that all holomorphic curves were multiple covers of the basic curves also after stretching, it would follow that the partition functions is generated as claimed.

\appendix

\section{The partition function of $L_{ll}$ from the topological vertex}\label{eq:HOMFLY-Z-Hopf}
The partition function of HOMFLYPT polynomials for the Hopf link \eqref{eq: Hopf HOMFLY partitionfunction}, which corresponds to the conormal-conormal filling $L_{ll}$, can also be obtained from the topological vertex.

More precisely, we have checked that \eqref{eq: Hopf HOMFLY partitionfunction}, obtained as the unique normalized solution to  $\sfA_i\cdot \sfZ_L=0$ with $i=1,2,3$, agrees with the open topological string partition function of two toric branes on external legs in the resolved conifold 
\be\label{eq:Hopf-HOMFLY-top-vert}
H_{\Gamma;\lambda,\mu}(a,q) =  \frac{q^{-\kappa(\lambda)} }{(a^2;q^2;q^2)_\infty} \sum_{\nu} C_{\lambda\mu^t\nu}(q) C_{00\nu^t}(q)\, (-a^2)^{|\nu|} 
\ee
where $C_{\lambda\mu\nu}(q)$ is the topological vertex function~\cite{Aganagic:2003db}.
The normalization $q^{-\kappa(\lambda)} $ corresponds to a choice of framing such that HOMFLYPT polynomials enjoy a $\mathbb{Z}_2$ symmetry $H_{\Gamma;\lambda,\mu} = H_{\Gamma;\mu,\lambda}$. 

In a more general choice of framing, the partition function becomes
\be
\begin{split}
	\sfZ_{ll}^{(f_1, f_2)}
	&= \Phi_{L_1}^{f_1}\otimes \Phi_{L_2}^{f_2} \, (\sfZ_\Gamma)
	\\
	&= \sum_{\lambda,\mu}
	\left[(-1)^{|\lambda_1|} q^{\kappa(\lambda_1)}\right]^{f_1} 
	\left[(-1)^{|\lambda_2|} q^{\kappa(\lambda_2)}\right]^{f_2} 
	H_{\Gamma;\lambda_1,\lambda_2}(a,q)\, \sfW_{\lambda_1}\otimes \sfW_{\lambda_2}\,,
\end{split}
\ee
where the framing operators $\Phi_{L_i}^{f_i}$ were defined in \eqref{eq:framed-Z}.
This obeys the framed skein valued recursion relations 
\be
\sfA_i^{(f_1,f_2)} \cdot \sfZ_{ll}^{(f_1, f_2)} = 0\,,
\quad \text{ for framings }(f_1, f_2)\in \IZ^2\,,
\ee
%%%
defined by the operators \eqref{eq:A1-skein-framed}-\eqref{eq:A2-skein-framed}-\eqref{eq:A3-skein-framed},
after specializing $a_{\Lambda_i} = a_{L_i}$, and after introducing an additional rescaling of $P^{(\alpha)}_{i,j}$ by appropriate powers of $a_{L_{\alpha}}^{-j f_\alpha}$ as given by 4-chain intersections, see \eqref{eq:aL-Pij-framed}.

\bibliography{biblio}{}
\bibliographystyle{JHEP}
\end{document}